\documentclass[a4paper,11pt,reqno]{customart}

\usepackage[utf8x]{inputenc}
\usepackage[margin=1in]{geometry}
\usepackage{verbatim}
\usepackage{color}
\usepackage[table]{xcolor}
\usepackage{listings}
\usepackage{amssymb,amsfonts,amsthm,amsmath}
\usepackage{url}
\usepackage{enumerate}
\usepackage{todonotes}
\usepackage[all]{xy}
\usepackage{multirow}
\usepackage{tikz}
\usepackage[underline=false]{pgf-umlsd}
\usepackage{stmaryrd}
\usepackage{footnote}
\makesavenoteenv{tabular}
\usepackage{hyperref}

\makeatletter
\newcommand\footnoteref[1]{\protected@xdef\@thefnmark{\ref{#1}}\@footnotemark}
\makeatother

\hypersetup{
	pdfborder={0 0 0}
}

\definecolor{grey}{rgb}{0.95,0.95,0.95}
\definecolor{green}{rgb}{0.2,0.6,0.4}

\renewcommand{\O}{\textbf{0}}

\newcommand{\restr}{\upharpoonright}

\newcommand{\parop}{\mathtt{parts}}
\newcommand{\iniop}{\mathtt{init}}
\newcommand{\boxop}{\mathtt{box}}
\newcommand{\unboxop}{\mathtt{unbox}}
\newcommand{\satop}{\mathtt{sat}}
\newcommand{\progop}{\mathtt{prog}}
\newcommand{\queryop}{\mathop{Query}}
\newcommand{\ansop}{\mathop{Ans}}
\newcommand{\no}{\mathtt{No}}
\newcommand{\yes}{\mathtt{Yes}}
\newcommand{\Ub}{\mathbb{U}}
\newcommand{\Sb}{\mathbb{S}}

\newcommand{\Psf}{\mathsf{P}}
\newcommand{\Qsf}{\mathsf{Q}}

\newcommand{\imp}{\rightarrow}

\newcommand{\Mb}{\mathbb{M}}
\newcommand{\Nb}{\mathbb{N}}
\newcommand{\Pb}{\mathbb{P}}

\newcommand{\dbf}{\mathbf{d}}
\newcommand{\Acal}{\mathcal{A}}

\newcommand{\Ccal}{\mathcal{C}}
\newcommand{\Dcal}{\mathcal{D}}

\newcommand{\Ical}{\mathcal{I}}

\newcommand{\Mcal}{\mathcal{M}}
\newcommand{\Pcal}{\mathcal{P}}

\newcommand{\Tcal}{\mathcal{T}}

\newcommand{\uh}{{\upharpoonright}}

\renewcommand{\setminus}{\smallsetminus}

\newcommand{\tuple}[1]{\left\langle #1 \right\rangle}

\newcommand{\cond}[1]{\left\{\begin{array}{ll} #1 \end{array}\right.}

\newcommand{\s}[1]{\ensuremath{\sf{#1}}}

\DeclareMathOperator{\rca}{\s{RCA}_0}
\DeclareMathOperator{\aca}{\s{ACA}}
\DeclareMathOperator{\wkl}{\s{WKL}}

\DeclareMathOperator{\bst}{\s{B}\Sigma^0_2}
\DeclareMathOperator{\ist}{\s{I}\Sigma^0_2}

\DeclareMathOperator{\rt}{\s{RT}}
\DeclareMathOperator{\srt}{\s{SRT}}

\DeclareMathOperator{\ads}{\s{ADS}}
\DeclareMathOperator{\sads}{\s{SADS}}

\DeclareMathOperator{\cac}{\s{CAC}}

\DeclareMathOperator{\coh}{\s{COH}}

\DeclareMathOperator{\fip}{\s{FIP}}

\DeclareMathOperator{\sts}{\s{STS}}

\DeclareMathOperator{\emo}{\s{EM}}
\DeclareMathOperator{\semo}{\s{SEM}}

\DeclareMathOperator{\opt}{\s{OPT}}
\DeclareMathOperator{\amt}{\s{AMT}}

\usetikzlibrary{shapes,arrows}
\usetikzlibrary{decorations.markings}
\definecolor{lightblue}{HTML}{e6e6e6}
\definecolor{lightred}{HTML}{eca6a6}
\definecolor{lightgreen}{RGB}{164,244,140}

\newtheoremstyle{custom}
  {10pt}
  {10pt}
  {\normalfont}
  {}
  {\bfseries}
  {}
  { }
  {}

\theoremstyle{custom}

\usepackage{xcolor}	
\usepackage{soul}

\newtheorem{theorem}{Theorem}[section]
\newtheorem{lemma}[theorem]{Lemma}
\newtheorem{definition}[theorem]{Definition}

\newtheorem{question}[theorem]{Question}

\begin{document}

\title[Dominating the Erd\H{o}s-Moser theorem]{Dominating the Erd\H{o}s-Moser theorem\\ in reverse mathematics}
\author{
  Ludovic Patey
}

\begin{abstract}
The Erd\H{o}s-Moser theorem ($\emo$) states that every infinite tournament has an infinite
transitive subtournament. This principle plays an important role in the understanding of
the computational strength of Ramsey's theorem for pairs ($\rt^2_2$) by providing an alternate
proof of~$\rt^2_2$ in terms of~$\emo$ and the ascending descending sequence principle ($\ads$).
In this paper, we study the computational weakness of $\emo$ and construct
a standard model ($\omega$-model) of simultaneously~$\emo$, weak K\"onig's lemma and the cohesiveness
principle, which is not a model of the atomic model theorem. 
This separation answers a question of Hirschfeldt, Shore and Slaman,
and shows that the weakness of the Erd\"os-Moser theorem goes beyond the separation
of~$\emo$ from~$\ads$ proven by Lerman, Solomon and Towsner.
\end{abstract}

\maketitle

\section{Introduction}

Reverse mathematics is a mathematical program whose goal is
to classify theorems in terms of their provability strength.
It uses the framework of subsystems of second-order arithmetic,
with the base theory $\rca$, standing for Recursive Comprehension Axiom.
$\rca$ is composed of the basic first-order Peano axioms,
together with $\Delta^0_1$-comprehension and $\Sigma^0_1$-induction schemes.
$\rca$ is usually thought of as capturing \emph{computational mathematics}.
This program led to two important observations:
First, most ``ordinary'' (i.e.\ non set-theoreric) theorems require only very weak set existence axioms.
Second, many of those theorems are actually \emph{equivalent}
to one of five main subsystems over $\rca$, known as the ``Big Five''~\cite{Montalban2011Open}.

However, Ramsey theory is known to provide a large class of theorems escaping this phenomenon.
Indeed, consequences of Ramsey's theorem for pairs ($\rt^2_2$) usually belong to their own subsystem.
Therefore, they received a lot of attention from the reverse mathematics 
community~\cite{Cholak2001strength,Hirschfeldt2008strength,Hirschfeldt2007Combinatorial,Seetapun1995strength}.
This article focuses on Ramseyan principles below the arithmetic comprehension axiom~($\aca$). 
See Soare~\cite{Soare2016Turing} for a general introduction to computability theory,
and Hirschfeldt~\cite{Hirschfeldt2015Slicing} for a good introduction
to the reverse mathematics below~$\aca$.

\subsection{Cohesiveness}

Cohesiveness is a statement playing a central role in the analysis of Ramsey's theorem for pairs~\cite{Cholak2001strength}.
It can be seen as a sequential version of Ramsey's theorem for singletons and admits
characterizations in terms of degrees whose jump computes a path through a $\Pi^{0,\emptyset'}_1$ class~\cite{Jockusch1993cohesive}.
The decomposition of $\rt^2_2$ in terms of $\coh$ and stable Ramsey's theorem for pairs ($\srt^2_2$)
has been reused in the analysis of many consequences of Ramsey's theorem~\cite{Hirschfeldt2007Combinatorial}.
The link between cohesiveness and~$\srt^2_2$ is still an active research subject
\cite{Chong2014metamathematics,DzhafarovStrong,Hirschfeldt2016notions,Patey2016weakness}.

\begin{definition}[Cohesiveness]
An infinite set $C$ is $\vec{R}$-cohesive for a sequence of sets $\vec{R} = R_0, R_1, \dots$
if for each $i \in \omega$, $C \subseteq^{*} R_i$ or $C \subseteq^{*} \overline{R_i}$.
A set $C$ is \emph{p-cohesive} if it is $\vec{R}$-cohesive where
$\vec{R}$ is an enumeration of all primitive recursive sets.
$\coh$ is the statement ``Every uniform sequence of sets $\vec{R}$
has an $\vec{R}$-cohesive set.''
\end{definition}

Jockusch and Stephan~\cite{Jockusch1993cohesive} studied the degrees of unsolvability of cohesiveness
and proved that~$\coh$ admits a universal instance whose solutions
are the p-cohesive sets. They characterized their degrees as those whose
jump is PA relative to~$\emptyset'$.
The author extended this analysis to every computable instance of~$\coh$ and studied their degrees
of unsolvability~\cite{Patey2016weakness}.
Cholak, Jockush and Slaman~\cite{Cholak2001strength} proved that $\rt^2_2$ is computably equivalent to $\srt^2_2+\coh$.
Mileti~\cite{Mileti2004Partition} and Jockusch and Lempp [unpublished]
formalized this equivalence over $\rca$.
Hirschfeldt, Jockusch, Kjos-Hanssen, Lempp and Slaman~\cite{Hirschfeldt2008strength} proved that~$\coh$ contains a model
with no diagonally non-computable function, thus $\coh$ does not imply~$\srt^2_2$ over~$\rca$.
Cooper~\cite{Cooper1973Minimal} proved that every degree above~$\mathbf{0'}$ is the jump of a minimal degree.
Therefore there exists a p-cohesive set of minimal degree.

\subsection{The Erd\H{o}s-Moser theorem}

The Erd\H{o}s-Moser theorem is a principle coming from graph theory.
It provides together with the ascending descending principle~($\ads$) an alternative proof of
Ramsey's theorem for pairs ($\rt^2_2$). Indeed, every coloring~$f : [\omega]^2 \to 2$
can be seen as a tournament~$R$ such that~$R(x,y)$ holds if~$x < y$ and~$f(x,y) = 1$, or~$x > y$ and~$f(y, x) = 0$.
Every infinite transitive subtournament induces a linear order whose infinite ascending or descending
sequences are homogeneous for~$f$.

\begin{definition}[Erd\H{o}s-Moser theorem]
A tournament $T$ on a domain $D \subseteq \omega$ is an irreflexive binary relation on~$D$ such that for all $x,y \in D$ with $x \not= y$, exactly one of $T(x,y)$ or $T(y,x)$ holds. A tournament $T$ is \emph{transitive} if the corresponding relation~$T$ is transitive in the usual sense. A tournament $T$ is \emph{stable} if $(\forall x \in D)[(\forall^{\infty} s) T(x,s) \vee (\forall^{\infty} s) T(s, x)]$.
$\emo$ is the statement ``Every infinite tournament $T$ has an infinite transitive subtournament.''
$\semo$ is the restriction of $\emo$ to stable tournaments.
\end{definition}

Bovykin and Weiermann~\cite{Bovykin2005strength} introduced the Erd\H{o}s-Moser theorem in reverse mathematics
and proved that $\emo$ together with the chain-antichain principle ($\cac$) is equivalent to $\rt^2_2$ over $\rca$. 
This was refined into an equivalence between $\emo+\ads$ and $\rt^2_2$ by Montalb\'an (see~\cite{Bovykin2005strength}),
and the equivalence still holds between the stable versions of the statements.
Lerman, Solomon and Towsner~\cite{Lerman2013Separating} proved that~$\emo$ 
is strictly weaker than~$\rt^2_2$ by constructing an $\omega$-model
of~$\emo$ which is not a model of the stable ascending descending sequence~($\sads$). $\sads$
is the restriction of~$\ads$ to linear orders of order type~$\omega+\omega^{*}$~\cite{Hirschfeldt2007Combinatorial}. 
The author noticed in~\cite{Patey2013note} that their construction can be adapted to obtain a separation of~$\emo$ from the stable thin set theorem for pairs~($\sts(2)$).
Wang strengthened this separation by constructing in~\cite{Wang2014Definability} a standard model of many theorems,
including~$\emo$, $\coh$ and weak K\"onig's lemma ($\wkl$) which is neither a model of~$\sts(2)$
nor a model of~$\sads$. The author later refined in~\cite{Patey2015Iterative,Patey2016weakness} the forcing technique of Lerman, Solomon and Towsner and showed that it is strong enough to obtain the same separations as Wang.

On the lower bounds side, Lerman, Solomon and Towsner~\cite{Lerman2013Separating}\ showed that~$\emo$ 
implies the omitting partial types principle ($\opt$)
over~$\rca + \bst$ and Kreuzer proved in~\cite{Kreuzer2012Primitive} that $\semo$ implies
$\bst$ over $\rca$. The statement $\opt$ can be thought of as a stating 
for every set~$X$ the existence of a set hyperimmune relative to~$X$.
Finally, the author proved in~\cite{Patey2015Somewhere} that $\rca \vdash \emo \imp [\sts(2) \vee \coh]$.
In particular, every model of~$\emo$ which is not a model of~$\sts(2)$ is also a model of $\coh$.
This fact will be reused in this paper since $\sts(2)$ implies the atomic model theorem over~$\rca$~\cite{Patey2015Somewhere}.

\subsection{Domination and the atomic model theorem}\label{subsect:dominating-amt}

The atomic model theorem is a statement coming from model theory. It has been
introduced by Hirschfeldt, Shore and Slaman~\cite{Hirschfeldt2009atomic} in the settings of reverse mathematics.

\begin{definition}[Atomic model theorem]
A formula $\varphi(x_1, \dots, x_n)$ of $T$ is an \emph{atom} of a theory $T$ if for each formula $\psi(x_1, \dots, x_n)$
  we have $T \vdash \varphi \imp \psi$ or $T \vdash \varphi \imp \neg \psi$ but not both.
  A theory $T$ is \emph{atomic} if, for every formula $\psi(x_1, \dots, x_n)$ consistent with $T$,
  there is an atom $\varphi(x_1, \dots, x_n)$ of $T$ extending it, i.e., one such that $T \vdash \varphi \imp \psi$.
  A model $\Acal$ of $T$ is \emph{atomic} if every $n$-tuple from $\Acal$ satisfies an atom of $T$. 
$\amt$ is the statement ``Every complete atomic theory has an atomic model''.
\end{definition}

This strength of the atomic model theorem received a lot of attention from the reverse mathematics community 
and was subject to many refinements.
On the upper bound side, Hirschfeldt, Shore and Slaman~\cite{Hirschfeldt2009atomic} 
proved that~$\amt$ is a consequence of~$\sads$ over~$\rca$. 
The author~\cite{Patey2015Somewhere} proved that the stable thin set theorem for pairs ($\sts(2)$) implies~$\amt$ over~$\rca$.

On the lower bound side, Hirschfeldt, Shore and Slaman~\cite{Hirschfeldt2009atomic} proved that~$\amt$
implies the omitting partial type theorem ($\opt$) over~$\rca$. 
Hirschfeldt and Greenberg, and independently Day, Dzhafarov and Miller, 
strengthened this result by proving that $\amt$ implies the finite intersection property ($\fip$) over~$\rca$ (see~\cite{Hirschfeldt2015Slicing}).
The principle $\fip$ was first introduced by Dzhafarov and Mummert~\cite{Dzhafarov2013strength}. 
Later, Downey, Diamondstone, Greenberg and Turetsky~\cite{Downey2012Finite} and 
Cholak, Downey and Igusa~\cite{Cholak2015Any} proved that~$\fip$ is equivalent to the principle asserting, for every set $X$,
the existence of a 1-generic relative to~$X$. In particular, every model of $\amt$ contains 1-generic reals.

The computable analysis of the atomic model theorem revealed the genericity flavor of the statement.
More precisely, the atomic model theorem admits a pure computability-theoretic characterization
in terms of hyperimmunity relative to a fixed $\Delta^0_2$ function. 

\begin{definition}[Escape property]
For every $\Delta^0_2$ function $f$, there exists a function~$g$ such that $f(x) < g(x)$ for infinitely many $x$.
\end{definition}

The escape property is a statement in between hyperimmunity relative to~$\emptyset'$ and hyperimmunity.
The atomic model theorem is computably equivalent to the escape property, that is, for every complete atomic theory $T$,
there is a $\Delta^{0,T}_2$ function $f$ such that for every function $g$ satisfying the escape property for~$f$,
$T \oplus g$ computes an atomic model of~$T$. Conversely, for every $\Delta^0_2$ approximation $\tilde{f}$
of a function $f$, there is a $\tilde{f}$-computable complete atomic theory such that for every atomic model $\Mcal$,
$\tilde{f} \oplus \Mcal$ computes a function satisfying the escape property for~$f$.
In particular, the $\omega$-models satisfying $\amt$
are exactly the ones satisfying the escape property. However the formalization of this equivalence requires
more than the $\Sigma^0_1$ induction scheme.
It was proven to hold over~$\rca + \ist$ but not~$\rca + \bst$~\cite{Hirschfeldt2009atomic,Conidis2008Classifying}, 
where~$\ist$ and~$\bst$ are respectively the~$\Sigma^0_2$ induction scheme and
the $\Sigma^0_2$ bounding scheme.

Hirschfeldt, Shore and Slaman~\cite{Hirschfeldt2009atomic} asked the following question.

\begin{question}
Does the cohesiveness principle imply the atomic model theorem over~$\rca$?
\end{question}

Note that $\amt$ is not computably reducible to~$\coh$, since there exists a cohesive set of minimal degree~\cite{Cooper1973Minimal},
and a computable atomic theory whose computable atomic models bound 1-generic reals~\cite{Hirschfeldt2015Slicing}, 
but no minimal degree bounds a 1-generic real~\cite{Yu2006Lowness}.

In this paper, we answer this question negatively.
We shall take advantage of the characterization of $\amt$ by the escape property to create an $\omega$-model $\Mcal$ of $\emo$, $\wkl$
and $\coh$ simultaneously, together with a $\Delta^0_2$ function $f$ dominating every function in~$\Mcal$.
Therefore, any $\Delta^0_2$ approximation $\tilde{f}$ of the function $f$ 
is a computable instance of the escape property belonging to $\Mcal$, but with no solution in $\Mcal$.
The function $f$ witnesses in particular that~$\Mcal \not \models \amt$.
Our main theorem is the following.

\begin{theorem}[Main theorem]\label{thm:main-theorem}
$\coh \wedge \emo \wedge \wkl$ does not imply $\amt$ over~$\rca$.
\end{theorem}

The proof techniques used to prove the main theorem will be introduced
progressively by considering first computable non-reducibility, and then generalizing
the diagonalization to Turing ideals by using an effective iterative forcing.

\subsection{Definitions and notation}

\emph{String, sequence}.
Fix an integer $k \in \omega$.
A \emph{string} (over $k$) is an ordered tuple of integers $a_0, \dots, a_{n-1}$
(such that $a_i < k$ for every $i < n$). The empty string is written $\varepsilon$. A \emph{sequence}  (over $k$)
is an infinite listing of integers $a_0, a_1, \dots$ (such that $a_i < k$ for every $i \in \omega$).
Given $s \in \omega$,
$k^s$ is the set of strings of length $s$ over~$k$ and
$k^{<s}$ is the set of strings of length $<s$ over~$k$. Similarly,
$k^{<\omega}$ is the set of finite strings over~$k$
and $k^{\omega}$ is the set of sequences (i.e. infinite strings)
over~$k$. 
Given a string $\sigma \in k^{<\omega}$, we denote by $|\sigma|$ its length.
Given two strings $\sigma, \tau \in k^{<\omega}$, $\sigma$ is a \emph{prefix}
of $\tau$ (written $\sigma \preceq \tau$) if there exists a string $\rho \in k^{<\omega}$
such that $\sigma \rho = \tau$. Given a sequence $X$, we write $\sigma \prec X$ if
$\sigma = X \uh n$ for some $n \in \omega$, where $X \uh n$ denotes the restriction of $X$ to its first $n$ elements.
A \emph{binary string} is a \emph{string} over~$2$.
A \emph{real} is a sequence over~$2$.
We may identify a real with a set of integers by considering that the real is its characteristic function.

\emph{Tree, path}.
A tree $T \subseteq k^{<\omega}$ is a set downward-closed under the prefix relation.
A \emph{binary} tree is a tree $T \subseteq 2^{<\omega}$.
A sequence $P \in k^\omega$ is a \emph{path} though~$T$ if for every $\sigma \prec P$,
$\sigma \in T$. A string $\sigma \in k^{<\omega}$ is a \emph{stem} of a tree $T$
if every $\tau \in T$ is comparable with~$\sigma$.
Given a tree $T$ and a string $\sigma \in T$,
we denote by $T^{[\sigma]}$ the subtree $\{\tau \in T : \tau \preceq \sigma \vee \tau \succeq \sigma\}$.

\emph{Sets, partitions}.
Given two sets $A$ and $B$, we denote by $A < B$ the formula
$(\forall x \in A)(\forall y \in B)[x < y]$
and by $A \subseteq^{*} B$ the formula $(\forall^{\infty} x \in A)[x \in B]$,
meaning that $A$ is included in $B$ \emph{up to finitely many elements}.
Given a set~$X$ and some integer~$k$, a~\emph{$k$-cover of~$X$}
is a $k$-uple $A_0, \dots, A_{k-1}$ such that~$A_0 \cup \dots \cup A_{k-1} = X$.
We may simply say~\emph{$k$-cover} when the set~$X$ is unambiguous. 
A \emph{$k$-partition} is a $k$-cover whose sets are pairwise disjoint.
A \emph{Mathias condition} is a pair $(F, X)$
where $F$ is a finite set, $X$ is an infinite set
and $F < X$.
A condition $(F_1, X_1)$ \emph{extends } $(F, X)$ (written $(F_1, X_1) \leq (F, X)$)
if $F \subseteq F_1$, $X_1 \subseteq X$ and $F_1 \setminus F \subset X$.
A set $G$ \emph{satisfies} a Mathias condition $(F, X)$
if $F \subset G$ and $G \setminus F \subseteq X$.
We refer the reader to Chapter 2 in Hirschfeldt~\cite{Hirschfeldt2015Slicing} for a gentle introduction
to effective forcing.

\section{The weakness of cohesiveness under computable reducibility}

Before proving that~$\coh$ does not imply $\amt$ over~$\rca$,
we illustrate the key features of our construction by 
showing that $\amt$ does not reduce to~$\coh$ in one step.
This one-step reducibility is known as \emph{computable reducibility}~\cite{DzhafarovStrong,Hirschfeldt2016notions,Patey2016weakness}.
The general construction will consist of iterating this 
one-step diagonalization to construct a Turing ideal
whose functions are dominated by a single $\Delta^0_2$ function.

\begin{definition}[Computable reducibility]
A principle $\Psf$ is \emph{computably reducible} to another principle~$\Qsf$ (written $\Psf \leq_c \Qsf$)
if every~$\Psf$-instance~$I$ computes a~$\Qsf$-instance~$J$ such that for every solution~$X$ to~$J$,
$X \oplus I$ computes a solution to~$I$.
\end{definition}

The remainder of this section is devoted to the proof of the following theorem.

\begin{theorem}\label{thm:amt-comp-reduc-coh}
$\amt \not \leq_c \coh$
\end{theorem}

In order to prove Theorem~\ref{thm:amt-comp-reduc-coh},
we need to construct a $\Delta^0_2$ function $f$ such that 
for every uniformly computable sequence of sets~$\vec{R} = R_0, R_1, \dots$,
there is an $\vec{R}$-cohesive set~$G$ such that every~$G$-computable
function is dominated by~$f$. Thankfully, Jockusch and Stephan~\cite{Jockusch1993cohesive}
proved that for every such sequence of sets~$\vec{R}$, 
every p-cohesive set computes an infinite $\vec{R}$-cohesive set.
The sequence of all primitive recursive sets is therefore called a \emph{universal instance}.
Hence we only need to build a $\Delta^0_2$ function~$f$ and a p-cohesive set~$G$
such that every $G$-computable function is dominated by~$f$ to obtain Theorem~\ref{thm:amt-comp-reduc-coh}.

Given some uniformly computable sequence of sets~$\vec{R} = R_0, R_1, \dots$,
the usual construction of an $\vec{R}$-cohesive set~$G$ is done by a computable Mathias forcing.
The forcing conditions are pairs $(F,X)$, where~$F$ is a finite set representing the finite
approximation of~$G$ and~$X$ is an infinite, computable reservoir such that~$max(F) < min(X)$.
The construction of the~$\vec{R}$-cohesive set is obtained by building
an infinite, decreasing sequence of Mathias conditions, starting with~$(\emptyset, \omega)$
and interleaving two kinds of steps.
Given some condition~$(F,X)$,
\begin{itemize}
	\item[(S1)] the \emph{extension} step consists of taking an element $x$ from $X$ and adding it to~$F$,
	thereby forming the extension $(F \cup \{x\}, X \setminus [0,x])$;
	\item[(S2)] the \emph{cohesiveness} step consists of deciding which one of $X \cap R_i$
	and $X \cap \overline{R}_i$ is infinite, and taking the chosen one as the new reservoir.
\end{itemize}
The first step ensures that the constructed set~$G$ will be infinite, whereas
the second step makes the set $G$ $\vec{R}$-cohesive.
Looking at the effectiveness of the construction, the step (S1) is computable,
assuming we are given some Turing index of the set~$X$.
The step (S2), on the other hand, requires to decide which one of two computable sets
is infinite, knowing that at least one of them is. This decision
requires the computational power of a PA degree relative to~$\emptyset'$ (see \cite[Lemma 4.2]{Cholak2001strength}).
Since we want to build a $\Delta^0_2$ function~$f$ dominating every $G$-computable function,
we would like to make the construction of~$G$ $\Delta^0_2$. Therefore the step (S2) has to be revised.

\subsection{Effectively constructing a cohesive set}

The above construction leads to two observations.
First, at any stage of the construction, the reservoir~$X$ of the Mathias condition~$(F, X)$
has a particular shape. Indeed, after the first application of stage~(S2), 
the set $X$ is, up to finite changes, of the form $\omega \cap R_0$
or $\omega \cap \overline{R_0}$. After the second application of (S2), it is in one of the following forms: $\omega \cap R_0 \cap R_1$,
$\omega \cap R_0 \cap \overline{R}_1$, $\omega \cap \overline{R}_0 \cap R_1$,
$\omega \cap \overline{R}_0 \cap \overline{R}_1$, and so on. More generally, given some string~$\sigma \in 2^{<\omega}$,
we can define~$R_\sigma$ inductively as follows:
First, $R_\varepsilon = \omega$, and then, if $R_\sigma$ has already been defined for some string $\sigma$ of length~$i$,
$R_{\sigma 0} = R_\sigma \cap \overline{R}_i$ and~$R_{\sigma 1} = R_\sigma \cap R_i$.
By the first observation, we can replace Mathias conditions by pairs ~$(F, \sigma)$, where $F$ is a finite set
and $\sigma \in 2^{<\omega}$. The pair~$(F, \sigma)$ denotes the Mathias condition~$(F, R_\sigma \setminus [0, max(F)])$.
A pair $(F, \sigma)$ is \emph{valid} if $R_\sigma$ is infinite.
The step (S2) can be reformulated as choosing, given some valid condition $(F, \sigma)$, which one of $(F, \sigma 0)$
and $(F, \sigma 1)$ is valid.






Second, we do not actually need to decide which one of~$R_{\sigma 0}$ and~$R_{\sigma 1}$ is infinite
assuming that~$R_{\sigma}$ is infinite. Our goal is to dominate every~$G$-computable function with a $\Delta^0_2$ function $f$.
Therefore, given some $G$-computable function~$g$, it is sufficient to find a finite set $S$ of candidate values for~$g(x)$ 
and make~$f(x)$ be greater than the maximum of $S$. Instead of choosing which one of~$R_{\sigma 0}$ and~$R_{\sigma 1}$ is infinite,
we will explore both cases in parallel. The step (S2) will split some condition $(F, \sigma)$
into two conditions~$(F, \sigma 0)$ and $(F, \sigma 1)$. Our new forcing conditions are therefore tuples $(F_\sigma : \sigma \in 2^n)$
which have to be thought of as $2^n$ parallel Mathias conditions $(F_\sigma, \sigma)$ for each~$\sigma \in 2^n$.
Note that $(F_\sigma, \sigma)$ may not denote a valid Mathias condition in general since $R_\sigma$ may be finite.
Therefore, the step (S1) becomes~$\Delta^0_2$, since we first have to check whether $R_\sigma$ is non-empty
before picking an element in~$R_\sigma$. The whole construction is $\Delta^0_2$ and yields a $\Delta^0_2$ infinite
binary tree $T$. In particular, any degree PA relative to $\emptyset'$ bounds an infinite path though $T$ and therefore bounds a $G$-cohesive set. However, the degree of the set $G$ is not sensitive in our argument. We only care about the effectiveness of
the tree~$T$.

\subsection{Dominating the functions computed by a cohesive set}

We have seen in the previous section how to make the construction of a cohesive set more effective
by postponing the choices between forcing~$G \subseteq^{*} R_i$ and~$G \subseteq^{*} \overline{R}_i$
to the end of the construction. We now show how to dominate every $G$-computable function
for every infinite path~$G$ through the $\Delta^0_2$ tree constructed in the previous section.
To do this, we will interleave a third step deciding whether $\Phi^G_e(n)$ halts, and if so, collecting the 
candidate values of~$\Phi^G_e(n)$.
Given some Mathias precondition~$(F, X)$ (a precondition is a condition 
where we do not assume that the reservoir is infinite) and some $e,x \in \omega$, we can $\Delta^0_2$-decide 
whether there is some set $E \subseteq X$ such that~$\Phi^{F \cup E}_e(x) \downarrow$.
If this is the case, then we can effectively find this a finite set~$E \subseteq X$ and
compute the value~$\Phi^{F \cup E}_e(x)$. If this is not the case, then for every infinite set~$G$ satisfying
the condition~$(F, X)$, the function~$\Phi^G_e$ will not be defined on input~$x$. In this case,
our goal is vacuously satisfied since $\Phi^G_e$ will not be a function and therefore
we do not need do dominate~$\Phi^G_e$. 
Let us go back to the previous construction.
After some stage, we have constructed a condition~$(F_\sigma : \sigma \in 2^n)$ inducing a finite tree of depth~$n$. 
The step (S3) acts as follows for some~$x \in \omega$:
\begin{itemize}
	\item[(S3)] Let~$S = \{0\}$. For each~$\sigma \in 2^n$ and each~$e \leq x$, decide whether 
	there is some finite set~$E \subseteq R_\sigma \setminus [0, max(F_\sigma)]$ such that~$\Phi^{F_\sigma \cup E}_e(x) \downarrow$.
	If this is the case, add the value of $\Phi^{F_\sigma \cup E}_e(x)$ to $S$ and set~$\tilde{F}_\sigma = F_\sigma \cup E$, otherwise set $\tilde{F}_\sigma = F_\sigma$.
	Finally, set~$f(x) = max(S)+1$ and take~$(\tilde{F}_\sigma : \sigma \in 2^n)$ as the next condition.
\end{itemize}
Note that the step (S3) is $\Delta^0_2$-computable uniformly in the condition~$(F_\sigma : \sigma \in 2^n)$.
The whole construction therefore remains~$\Delta^0_2$ and so does the function~$f$.
Moreover, given some $G$-computable function~$g$, there is some Turing index~$e$ such that~$\Phi^G_e = g$.
For each $x \geq e$, the step (S3) is applied at a finite stage and decides whether~$\Phi^G_e(x)$ halts or not
for every set satisfying one of the leaves of the finite tree. In particular, this is the case for the set $G$ and 
therefore $\Phi^G_e(x) \in S$. By definition of $f$, $f(x) \geq max(S) \geq \Phi^G_e(x)$. Therefore $f$ dominates the function~$g$.

\subsection{The formal construction}

Let~$\vec{R} = R_0, R_1, \dots$ be the sequence of all primitive recursive sets.
We define a $\Delta^0_2$ decreasing sequence of conditions~$(\emptyset, \varepsilon) \geq c_0 \geq c_1 \dots$
such that for each~$s \in \omega$
\begin{itemize}
	\item[(i)] $c_s = (F^s_\sigma : \sigma \in 2^s)$ and~$|F^s_\sigma| \geq s$ if~$R_\sigma \setminus [0, max(F^s_\sigma)] \neq \emptyset$.
	\item[(ii)] For every~$e \leq s$ and every~$\sigma \in 2^s$, either $\Phi^{F^s_\sigma}_e(s) \downarrow$
	or $\Phi^G_e(s) \uparrow$ for every set~$G$ satisfying $(F^s_\sigma, R_\sigma)$.
\end{itemize}
Let~$P$ be a path through the tree $T = \{ \sigma \in 2^{<\omega} : R_\sigma \mbox{ is infinite} \}$
and let~$G = \bigcup_s F^s_{P \restr s}$. By (i), for each~$s \in \omega$, $|F^s_{P \restr s}| \geq s$
since~$R_{P \restr s}$ is infinite. Therefore the set~$G$ is infinite.
Moreover, for each~$s \in \omega$, the set~$G$ satisfies the condition~$(F^{s+1}_{P \restr s+1}, R_{P \restr s+1})$,
so~$G \subseteq^{*} R_{P \restr s+1} \subseteq R_s$ if~$P(s) = 1$ and
$G \subseteq^{*} R_{P \restr s+1} \subseteq \overline{R}_s$ if~$P(s) = 0$. 
Therefore~$G$ is $\vec{R}$-cohesive.

For each~$s \in \omega$, let~$f(s) = 1 + max(\Phi^{F^s_\sigma}_e(s) : \sigma \in 2^s, e \leq s)$.
The function~$f$ is $\Delta^0_2$. We claim that it dominates every $G$-computable function.
Fix some~$e$ such that~$\Phi^G_e$ is total. For every~$s \geq e$, let~$\sigma = P \restr s$. By (ii), either
$\Phi^{F^s_\sigma}_e(s) \downarrow$ or $\Phi^G_e(s) \uparrow$ for every
set~$G$ satisfying $(F^s_\sigma, R_\sigma)$. Since $\Phi^G_e(s) \downarrow$,
the first case holds. By definition of~$f$, $f(s) \geq \Phi^{F^s_\sigma}_e(s) = \Phi^G_e(s)$.
Therefore~$f$ dominates the function $\Phi^G_e$. This completes the proof of Theorem~\ref{thm:amt-comp-reduc-coh}.

\section{The weakness of~$\emo$ under computable reducibility}\label{sect:emo-computable-reducibility}

We now strengthen the analysis of the previous section by
proving that the atomic model theorem is not computably reducible to the Erd\H{o}s-Moser theorem.
Theorem~\ref{thm:amt-comp-reduc-coh} is an immediate consequence of this result since
$[\amt \vee \coh] \leq_c \emo$ (see~\cite{Patey2015Somewhere}).
After this section, we will be ready to iterate the construction
in order to build an $\omega$-model of~$\emo \wedge \coh$ which is not a model of~$\amt$.

\begin{theorem}\label{thm:amt-comp-reduc-em-coh}
$\amt \not \leq_c \emo$
\end{theorem}

Before proving Theorem~\ref{thm:amt-comp-reduc-em-coh},
we start with an analysis of the combinatorics of the Erd\H{o}s-Moser theorem.
Just as we did for cohesiveness, we will show how to build solutions to $\emo$ through $\Delta^0_2$ constructions, 
postponing the $\Pi^0_2$ choices to the end.

\subsection{The combinatorics of the Erd\H{o}s-Moser theorem}\label{subsect:combi-em}

The standard way of building an infinite object by forcing consists of defining an increasing
sequence of finite approximations, and taking the union of them. Unlike
$\coh$ where every finite set can be extended to an infinite cohesive set,
some finite transitive subtournaments may not be extensible to an infinite one.
We therefore need to maintain some extra properties which will guarantee that
the finite approximations are extendible.
The nature of these properties constitue the core of the combinatorics of $\emo$.

Lerman, Solomon and Towsner~\cite{Lerman2013Separating} proceeded to an analysis
of the Erd\H{o}s-Moser theorem.
They showed in particular that it suffices to ensure that the finite transitive subtournament~$F$
has infinitely many \emph{one-point extensions}, that is, infinitely many elements~$x$ such that
$F \cup \{x\}$ is transitive, to extend~$F$ to an infinite transitive subtournament (see~\cite[Lemma 3.4]{Lerman2013Separating}).
This property is sufficient to add elements one by one to the finite approximation.
However, when adding elements by block, we shall maintain a stronger invariant. We will require that
the reservoir is included in a minimal interval of the finite approximation~$F$.
In this section, we reintroduce the terminology of Lerman, Solomon and Towsner~\cite{Lerman2013Separating}
and give a presentation of the combinatorics of the Erd\H{o}s-Moser theorem
motivated by its computational analysis.

\begin{definition}[Minimal interval]
Let $R$ be an infinite tournament and $a, b \in R$
be such that $R(a,b)$ holds. The \emph{interval} $(a,b)$ is the
set of all $x \in R$ such that $R(a,x)$ and $R(x,b)$ hold.
Let $F \subseteq R$ be a finite transitive subtournament of $R$.
For $a, b \in F$ such that $R(a,b)$ holds, we say that $(a,b)$
is a \emph{minimal interval of $F$} if there is no $c \in F \cap (a,b)$,
i.e., no $c \in F$ such that $R(a,c)$ and $R(c,b)$ both hold.
\end{definition}

Fix a computable tournament~$R$, and consider a pair $(F, X)$ where
\begin{itemize}
	\item[(i)] $F$ is a finite $R$-transitive set representing the \emph{finite approximation}
	of the infinite $R$-transitive subtournament we want to construct
	\item[(ii)] $X$ is an infinite set disjoint from $F$, included in a minimal interval of~$F$
	and such that $F \cup \{x\}$ is $R$-transitive for every $x \in X$. In other words,
	$X$ is an infinite set of one-point extensions.
	Such a set $X$ represents the \emph{reservoir}, that is, a set of candidate
	elements we may add to~$F$ later on.
\end{itemize}

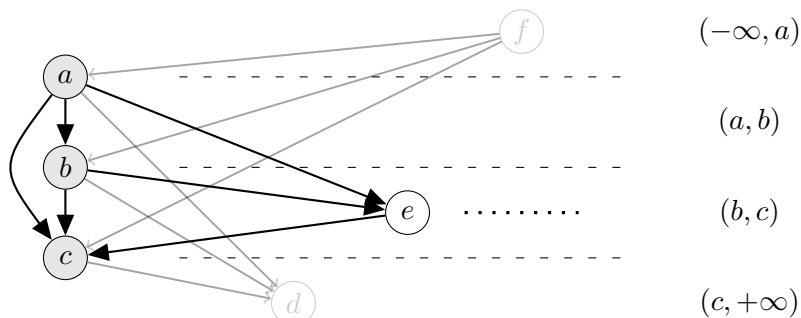
\begin{figure}[htbp]
\begin{center}
\begin{tikzpicture}[x=1.5cm, y=1.2cm, 
		node/.style={circle, draw, fill=lightblue, inner sep=0pt, minimum size=1.5em}, 
		arrow/.style={black,thick,->,-triangle 45},
		good/.style={fill=white},
		hiddenarrow/.style={black,thick,->, opacity=0.2},
		hidden/.style={opacity=0.2}
]

  \node[node] (a) at (1, 2) {$a$};
	\node[node] (b) at (1, 1) {$b$};
	\node[node] (c) at (1, 0) {$c$};
	\node[node,good,hidden] (d) at (3, -0.5) {$d$};
	\node[node,good] (e) at (4, 0.5) {$e$};
	\node[node,good,hidden] (f) at (5, 2.5) {$f$};

	\node at (7, 2.5) {$(-\infty, a)$};
	\node at (7, 1.5) {$(a, b)$};
	\node at (7, 0.5) {$(b, c)$};
	\node at (7, -0.5) {$(c, +\infty)$};

	\draw[arrow] (a) -- (b);
	\draw[arrow] (b) -- (c);
	\draw[arrow] (a)  .. controls (0.4,1) .. (c);

	\draw[hiddenarrow] (a) -- (d);
	\draw[hiddenarrow] (c) -- (d);
	\draw[hiddenarrow] (b) -- (d);

	\draw[arrow] (a) -- (e);
	\draw[arrow] (e) -- (c);
	\draw[arrow] (b) -- (e);

	\draw[hiddenarrow] (f) -- (a);
	\draw[hiddenarrow] (f) -- (c);
	\draw[hiddenarrow] (f) -- (b);

	\draw[very thick, loosely dotted] (4.5,0.5) -- (5.5,0.5);

	\draw[loosely dashed] (2, 2) -- (6, 2);
	\draw[loosely dashed] (2, 1) -- (6, 1);
	\draw[loosely dashed] (2, 0) -- (6, 0);
\end{tikzpicture}
\end{center}
\caption{In this figure, $F = \{a, b, c\}$ is a transitive set,
$X = \{d, e, f, \dots \}$ a set of one-point extensions, $(b, c) = \{e, \dots \}$ a minimal interval of~$F$ 
and $(F, X \cap (b, c))$ an EM condition. The elements~$d$ and~$f$ are not part of the minimal interval~$(b, c)$.} 
\end{figure}

The infinite set~$X$ ensures extensibility of the finite set~$F$ into an infinite $R$-transitive
subtournament. Indeed, by applying the Erd\H{o}s-Moser theorem to $R$ over the domain $X$, there exists an infinite $R$-transitive
subtournament $H \subseteq X$. One easily checks that $F \cup H$ is $R$-transitive.
The pair $(F, X)$ is called an Erd\H{o}s-Moser condition in~\cite{Patey2015Degrees}.
A set~$G$ \emph{satisfies} an EM condition~$(F, X)$ if it is $R$-transitive and satisfies the Mathias condition~$(F, X)$.
In order to simplify notation, given a tournament $R$ and two sets~$E$ and~$F$,
we denote by $E \to_R F$ the formula $(\forall x \in E)(\forall y \in F) R(x,y)$.

Suppose now that we want to add a finite number of elements of~$X$ into $F$ to obtain
a finite $T$-transitive set $\tilde{F} \supseteq F$,
and find an infinite subset $\tilde{X} \subseteq X$ such that $(\tilde{F}, \tilde{X})$
has the above mentioned properties. We can do this in a few steps:

\begin{itemize}
	\item[1.] Choose a finite (not necessarily $R$-transitive) set $E \subset X$.
	\item[2.] Any element $x \in X \setminus E$ induces a 2-partition $\tuple{E_0, E_1}$ of $E$
	by setting $E_0 = \{y \in E : R(y, x) \}$ and $E_1 = \{y \in E : R(x, y)\}$.
	Consider the coloring $f$ which associates to any element of $X \setminus E$ the corresponding 2-partition $\tuple{E_0, E_1}$ of $E$.
	\item[3.]
	As~$E$ is finite, there exists finitely many 2-partitions of~$E$, so $f$ colors each element of $X \setminus E$ into
	finitely many colors. By Ramsey's theorem for singletons applied to~$f$, there exists a 2-partition $\tuple{E_0, E_1}$ of $E$
	together with an infinite subset $\tilde{X} \subseteq X \setminus E$ such that for every $x \in \tilde{X}$, $f(x) = \tuple{E_0, E_1}$.
	By definition of~$f$ and~$E_i$, $E_0 \to_R \tilde{X} \to_R E_1$.
	
	\item[4.] Take any $R$-transitive subset $F_1 \subseteq E_i$ for some~$i < 2$ and set $\tilde{F} = F \cup F_1$.
	The pair $(\tilde{F}, \tilde{Y})$ satisfies the required properties (see~\cite[Lemma 5.9]{Patey2015Degrees} for a proof).
\end{itemize}

From a computational point of view, if we start with a computable condition~$(F, X)$, that is, where~$X$ is a computable set,
we end up with a computable extension~$(\tilde{F}, \tilde{Y})$.
Remember that our goal is to define a $\Delta^0_2$ function~$f$ which will dominate every $G$-computable function
for some solution~$G$ to~$R$. For this, we need to be able to~$\emptyset'$-decide
whether~$\Phi^G_e(n) \downarrow$ or $\Phi^G_e(n) \uparrow$ for every solution~$G$ to~$R$ satisfying 
some condition~$(F, X)$. 
More generally, given some~$\Sigma^0_1$ formula $\varphi$, we focus on the computational power required to decide a question of the form

\smallskip
{\itshape
Q1: Is there an $R$-transitive extension $\tilde{F}$ of $F$ in $X$ such that $\varphi(\tilde{F})$ holds?
}
\smallskip

Trying to apply naively the algorithm above requires a lot of computational power.
In particular, step 3 requires to choose a true formula among finitely many $\Pi^{0, X}_2$ formulas.
Such a step needs the power of PA degree relative to the jump of~$X$.
We shall apply the same trick as for cohesiveness, consisting in not trying to choose a true $\Pi^{0,X}_2$ formula,
but instead parallelizing the construction. Given a finite set $E \subset X$, 
instead of finding an infinite subset $\tilde{Y} \subset X \setminus E$
whose members induce a 2-partition of $E$, we will construct
as many extensions of $(F, X)$ as there are 2-partitions of~$E$. The question now becomes

\smallskip
{\itshape
Q2: Is there a finite set $E \subseteq X$ such that for every 2-partition $\tuple{E_0, E_1}$ of~$E$,
there exists an $R$-transitive subset $F_1 \subseteq E_i$ for some $i < 2$ such that $\varphi(F \cup F_1)$ holds?
}
\smallskip

This question is $\Sigma^{0,X}_1$, which is good enough for our purposes.
If the answer is positive, we will try the witness $F_1$ associated to each 2-partition of $E$ in parallel.
Note that there may be some 2-partition $\tuple{E_0, E_1}$ of~$E$
such that the set $Y = \{ x \in X \setminus E : E_0 \to_R \{x\} \to_R E_1 \}$ is finite,
but this is not a problem since there is \emph{at least}
one good 2-partition such that the corresponding set is infinite. 
The whole construction yields again a tree of pairs~$(F, X)$.

If the answer is negative, we want to ensure that
$\varphi(\tilde{F})$ will not hold at any further stage of the construction.
For each~$n \in \omega$, let $H_n$ be the set of the $n$ first elements of~$X$.
Because the answer is negative, for each~$n \in \omega$, there exists a 2-partition $\tuple{E_0, E_1}$
of~$H_n$ such that for every $R$-transitive subset $F_1 \subseteq E_i$ for any $i < 2$, $\varphi(F \cup F_1)$ does not hold.
Call such a 2-partition an \emph{avoiding} partition of $H_n$. 
Note that if $\tuple{E_0, E_1}$ is an avoiding partition of $H_{n+1}$, then $\tuple{E_0 \uh n, E_1 \uh n}$
is an avoiding partition of $H_n$. So the set of avoiding 2-partitions of some $H_n$
forms an infinite tree~$T$. Moreover, the predicate ``$\tuple{E_0, E_1}$ is an avoiding partition of $H_n$''
is $\Delta^{0, H_n}_1$ so the tree $T$ is $\Delta^{0,X}_1$. The collection of the infinite paths through $T$
forms a non-empty $\Pi^{0,X}_1$ class $\Ccal$ defined as the collection of 2-partitions $Z_0 \cup Z_1 = X$
such that for every $i < 2$ and every $R$-transitive subset $F_1 \subseteq Z_i$, $\varphi(F \cup F_1)$
does not hold.

The natural next step would be to apply weak K\"onig's lemma to obtain a 2-partition of~$X$
such that for every finite $R$-transitive subset $F_1$ of any of its parts, $\varphi(F \cup F_1)$ does not hold.
By the low basis theorem, we could take the 2-partition to be low over~$X$ and the whole construction would remain~$\Delta^0_2$.
However, when iterating the construction, we will be given only finite pieces of tournaments
since the tournament may depend on an oracle being constructed at a previous iteration. In this setting,
it will be impossible to compute a member of the $\Pi^{0,X}_1$ class $\Ccal$ of 2-partitions, since 
we will have access to only a finite piece of the corresponding tree~$T$.
In order to get progressively prepared to the iterated forcing, we will not apply $\wkl$
and will work with $\Pi^0_1$ classes of 2-partitions.
Therefore, if the answer is negative, we duplicate the finite $R$-transitive $F$ into two sets $F_0 = F_1 = F$,
and commit~$F_i$ to take from now on its next elements from~$X_i$ for some 2-partition
$X_0 \cup X_1 = X$ belonging to the $\Pi^0_1$ class~$\Ccal$ of 2-partitions witnessing the negative answer.
Iterating the process by asking several questions leads to tuples $(F_0, \dots, F_{k-1}, \Ccal)$
where $F_i$ is a finite $R$-transitive set taking its elements from the $i$th part of the class~$\Ccal$ of $k$-partitions.
This notion of forcing will be defined formally in a later section.

\subsection{Enumerating the computable infinite tournaments}

Proving that some principle~$\Psf$ does not computably reduce to~$\Qsf$
requires to create a $\Psf$-instance~$X$ such that \emph{every} $X$-computable $\Qsf$-instance
has a solution~$Y$ such that~ $Y \oplus X$ does not compute a solution to~$X$.
In the case of $\amt \not \leq_c \coh$, we have been able to restrict ourselves to only one instance of $\coh$,
since Jockusch and Stephan~\cite{Jockusch1993cohesive} showed it admits a universal instance.
It is currently unknown whether the Erd\H{o}s-Moser theorem admits a universal instance, that is, a computable infinite tournament
such that for every infinite transitive subtournament $H$ and for every computable infinite tournament $T$,
$H$ computes an infinite transitive $T$-subtournament. See~\cite{Patey2015Degrees} for an extensive study of the existence
of universal instances for principles in reverse mathematics.

Since we do not know whether $\emo$ admits a universal instance, we will need to diagonalize against
the solutions to every computable $\emo$-instance. In fact, we will prove a stronger result. We will construct
a $\Delta^0_2$ function~$f$ and an infinite set~$G$ which is eventually transitive simultaneously for every computable infinite tournament,
and such that $f$ dominates every $G$-computable function. There exists no computable sequence of sets
containing all computable sets. Therefore it is not possible to computably enumerate every infinite computable tournament.
However, one can define an infinite, computable, binary tree such that every infinite path
computes such a sequence. 
See the notion of sub-uniformity defined by Mileti in~\cite{Mileti2004Partition} for details.
By the low basis theorem, there exists a low set bounding a sequence containing, 
among others, every infinite computable tournament.
As we shall prove below,  for every set~$C$ and every uniformly $C$-computable sequence of infinite tournaments~$\vec{R}$,
there exists a set~$G$ together with a $\Delta^{0, C}_2$ function $f$ such that
\begin{itemize}
	\item[(i)] $G$ is eventually $R$-transitive for every $R \in \vec{R}$
	\item[(ii)] If $\Phi^{G \oplus C}_e$ is total, then it is dominated by $f$ for every $e \in \omega$.
\end{itemize}
Thus it suffices to choose~$C$ to be our low set and $\vec{R}$ to be a uniformly $C$-computable sequence
of infinite tournaments containing every computable tournament to deduce the existence of a set~$G$ together
with a $\Delta^0_2$ function $f$ such that 
\begin{itemize}
	\item[(i)] $G$ is eventually $R$-transitive for every infinite, computable tournament $R$
	\item[(ii)] If $\Phi^{G \oplus C}_e$ is total, then it is dominated by $f$ for every $e \in \omega$
\end{itemize}

By the computable equivalence between $\amt$ and the escape property,
there exists a computable atomic theory $T$ such that every atomic model computes
a function~$g$ not dominated by~$f$. If $\amt \leq_c \emo$, then there exists
an infinite, computable tournament~$R$ such that every infinite $R$-transitive subtournament 
computes a model of~$T$, hence computes a function~$g$ not dominated by~$f$.
As the set~$G$ is, up to finite changes, an infinite $R$-transitive subtournament,
$G$ computes such a function~$g$, contradicting our hypothesis. Therefore $\amt \not \leq_c \emo$.

\subsection{Cover classes}

In this part, we introduce some terminology about classes of $k$-covers.
Recall that a $k$-cover of some set $X$ is a $k$-uple $A_0, \dots, A_{k-1}$ such that~$A_0 \cup \dots \cup A_{k-1} = X$.
In particular, the sets are not required to be pairwise disjoint.

\smallskip
\emph{Cover class}.
We identify a $k$-cover $Z_0 \cup \dots \cup Z_{k-1}$ of some set $X$ with the $k$-fold join of its parts
$Z = \bigoplus_{i < k} Z_i$, and refer this as a \emph{code} for the cover.
A \emph{$k$-cover class} of some set~$X$ is a tuple $\tuple{k, X, \Ccal}$
where $\Ccal$ is a collection of codes of $k$-covers of~$X$. 
We will be interested in $\Pi^0_1$ $k$-cover classes.
A \emph{part} of a $k$-cover class $\tuple{k, X, \Ccal}$ is a number $\nu < k$. Informally, a part $\nu$
represents the collection of all $Z_\nu$, where $Z_0 \oplus \dots \oplus Z_{k-1} \in \Ccal$.
For the simplicity of notation, we may use the same letter~$\Ccal$ to denote both a $k$-cover class~$(k, X, \Ccal)$
and the actual collection of $k$-covers~$\Ccal$. We then write $dom(\Ccal)$ for $X$
and $parts(\Ccal)$ for~$k$.

\smallskip
\emph{Restriction of a cover}. Given some $k$-cover $Z = Z_0 \oplus \dots \oplus Z_{k-1}$ of some set~$X$ and given some set~$Y \subseteq X$, we write $Z \restr Y$ for the $k$-cover $(Z_0 \cap Y) \oplus \dots \oplus (Z_{k-1} \cap Y)$ of~$Y$.
Similarly, given some cover class~$(k, X, \Ccal)$ and some set~$Y \subseteq X$, we denote by $\Ccal \restr Y$
the cover class~$(k, Y, \Dcal)$ where $\Dcal = \{ Z \restr Y : Z \in \Ccal \}$.
Given some part~$\nu$ of $\Ccal$ and some set~$E$, we write~$\Ccal^{[\nu, E]}$
for the cover class~$(k, X, \Dcal)$ where 
$\Dcal = \{ Z_0 \oplus \dots \oplus Z_{k-1} \in \Ccal : E \subseteq Z_\nu \}$.

\smallskip
\emph{Refinement}. The collection of cover classes can be given a natural partial order as follows.
Let~$m \geq k$ and $f : m \to k$. An $m$-cover $V_0 \oplus \dots \oplus V_{m-1}$ of $Y$ \emph{$f$-refines}
a $k$-cover $Z_0 \oplus \dots \oplus Z_{k-1}$ of $X$ if $Y \subseteq X$ and $V_\nu \subseteq Z_{f(\nu)}$ for each~$\nu < m$.
Given two cover classes $(k, X, \Ccal)$ and~$(m, Y, \Dcal)$
and some function $f : m \to k$, we say that $\Dcal$ \emph{$f$-refines} $\Ccal$
if for every $V \in \Dcal$, there is some~$Z \in \Ccal$ such that $V$ $f$-refines $Z$.
In this case, we say that \emph{part $\nu$ of $\Dcal$ refines part~$f(\nu)$ of~$\Ccal$}.

\smallskip
\emph{Acceptable part}. 
We say that part $\nu$ of $\Ccal$ is \emph{acceptable} if there exists some $Z_0 \oplus \dots \oplus Z_{k-1} \in \Ccal$
such that $Z_\nu$ is infinite. Part $\nu$ of $\Ccal$ is \emph{empty} if 
for every $Z_0 \oplus \dots \oplus Z_{k-1} \in \Ccal$, $Z_\nu = \emptyset$.
Note that if $\Ccal$ is non-empty and $dom(\Ccal)$ is infinite, then $\Ccal$ has at least one acceptable part.
Moreover, if~$\Dcal \leq_f \Ccal$ and part~$\nu$ of~$\Dcal$ is acceptable, then so is part~$f(\nu)$ of $\Ccal$.
The converse does not hold in general.

\subsection{The forcing notion}

We now get into the core of our forcing argument by defining
the forcing notion which will be used to build an infinite set eventually
transitive for every infinite computable tournament.
Fix a set $C$ and a uniformly $C$-computable sequence of infinite tournaments $R_0, R_1, \dots$
We construct our set~$G$ by a forcing whose conditions are tuples $(\alpha, \vec{F}, \Ccal)$ where
\begin{itemize}
	\item[(a)] $\Ccal$ is a non-empty $\Pi^{0,C}_1$ $k$-cover class of $[t, +\infty)$ 
	for some $k, t \in \omega$ ; $\alpha \in t^{<\omega}$
	\item[(b)] $(F_\nu \setminus [0, \alpha(i))) \cup \{x\}$ is $R_i$-transitive for every $Z_0 \oplus \dots \oplus Z_{k-1} \in \Ccal$,
	every $x \in Z_\nu$, every $i < |\alpha|$ and each $\nu < k$
	\item[(c)] $Z_\nu$ is included in a minimal $R_i$-interval of $F_\nu \setminus [0, \alpha(i))$
	for every $Z_0 \oplus \dots \oplus Z_{k-1} \in \Ccal$, every $i < |\alpha|$ and each $\nu < k$.
\end{itemize}
A condition $(\beta, \vec{E}, \Dcal)$ \emph{extends} 
$(\alpha, \vec{F}, \Ccal)$
(written $(\beta, \vec{E}, \Dcal) \leq (\alpha, \vec{F}, \Ccal)$) if $\beta \succeq \alpha$
and there exists a function $f : parts(\Dcal) \to parts(\Ccal)$ such that the following holds:
\begin{itemize}
	\item[(i)] $(E_\nu, dom(\Dcal))$ Mathias extends $(F_{f(\nu)}, dom(\Ccal))$ for each $\nu < parts(\Dcal)$ 
	\item[(ii)] $\Dcal$ $f$-refines $\Ccal^{[f(\nu), E_\nu \setminus F_{f(\nu)}]}$ for each $\nu < parts(\Dcal)$
\end{itemize}

One may think of a condition $(\alpha, \vec{F}, \Ccal)$ with, say, $parts(\Ccal) = k$,
as $k$ parallel Mathias conditions which are,
up to finite changes, Erd\H{o}s-Moser conditions simultaneously for the tournaments $R_0, \dots, R_{|\alpha|-1}$.
Given some $i < |\alpha|$, the value $\alpha(i)$ indicates at which point the sets $\vec{F}$
start being $R_i$-transitive.
More precisely, for every part $\nu < k$ and every $k$-cover $Z_0 \oplus \dots \oplus Z_{k-1} \in \Ccal$,
$(F_\nu \setminus [0, \alpha(i)), Z_\nu)$ is an Erd\H{o}s-Moser condition for $R_i$ for each $i < |\alpha|$.
Indeed, because of clause~(i), the elements $E_\nu \setminus F_{f(\nu)}$ added to $E_\nu$ come from $dom(\Ccal)$
and because of clause~(ii), these elements must come from the part $f(\nu)$ of the class~$\Ccal$,
otherwise $\Ccal^{[f(\nu), E_\nu \setminus F_{f(\nu)}]}$ would be empty and so would be $\Dcal$.

Of course, there may be some parts~$\nu$ of $\Ccal$ which are non-acceptable, that is, such that $Z_\nu$ is finite
for every $k$-cover $Z_0 \oplus \dots \oplus Z_{k-1} \in \Ccal$. However, by the infinite pigeonhole principle, 
$Z_\nu$ must be infinite for at least one $\nu < k$.
Choosing $\alpha$ to be in $t^{<\omega}$ instead of $\omega^{<\omega}$
ensures that all elements added to~$\vec{F}$ will have to be $R_i$-transitive
simultaneously for each~$i < |\alpha|$, as the elements are taken from $dom(\Ccal)$
and therefore are greater than the threshold $\alpha(i)$ for each $i < |\alpha|$.
A \emph{part} of a condition $c = (\alpha, \vec{F}, \Ccal)$ is a pair $\tuple{c, \nu}$,
where $\nu < parts(\Ccal)$. For the simplicity of notation, we may identify a part $\tuple{c, \nu}$
of a condition with the part $\nu$ of the corresponding cover class $\Ccal$. It must however be clear
that a part depends on the condition~$c$.

We start with a few basic lemmas reflecting the combinatorics described in 
the subsection~\ref{subsect:combi-em}.
They are directly adapted from the basic properties of an Erd\H{o}s-Moser condition
proven in~\cite{Patey2015Degrees}.
The first lemma states that each element of the finite transitive tournaments $\vec{F}$ behaves
uniformly with respect to the elements of the reservoir, that is, is beaten by every element
of the reservoir or beats all of them.

\begin{lemma}\label{lem:em-comp-reduc-uniform-behaviour}
For every condition~$c = (\alpha, \vec{F}, \Ccal)$,
every $Z_0 \oplus \dots \oplus Z_{k-1} \in \Ccal$, 
every part $\nu$ of $\Ccal$, every $i < |\alpha|$ and every $x \in F_\nu \setminus [0, \alpha(i))$, 
either $\{x\} \to_{R_i} Z_\nu$ or $Z_\nu \to_{R_i} \{x\}$.
\end{lemma}
\begin{proof}
By property (c) of the condition~$c$, there exists a minimal $R_i$-interval
$(u, v)$ of $F_\nu \setminus [0, \alpha(i))$ containing $Z_\nu$.
Here, $u$ and $v$ may be respectively $-\infty$ and $+\infty$.
By definition of an interval, $\{u\} \to_{R_i} Z_\nu \to_{R_i} \{v\}$.
By definition of a minimal interval, $R_i(x, u)$ or $R_i(v, x)$ holds.
Suppose the former holds. By transitivity of $F_\nu \setminus [0, \alpha(i))$,
for every $y \in Z_\nu$, $R_i(x, y)$ holds, since both $R_i(x, u)$ and~$R_i(u, y)$ hold. 
Therefore $\{x\} \to_{R_i} Z_\nu$. In the latter case, by symmetry, $Z_\nu \to_{R_i} \{x\}$.
\end{proof}

The second lemma is the core of the combinatorics of the Erd\H{o}s-Moser theorem. It provides
sufficient properties to obtain a valid extension of a condition. Properties (i) and (ii)
are simply the definition of an extension. Properties (iii) and (iv) help to propagate
properties (b) and (c) from a condition to its extension. We shall see empirically that 
properties (iii) and (iv) are simpler to check than (b) and (c), 
as the former properties match exactly the way we add elements to our finite tournaments $\vec{F}$. 
Therefore, ensuring that these properties
are satisfied usually consists of checking that we followed the standard process of adding elements
to~$\vec{F}$. 

\begin{lemma}\label{lem:em-comp-reduc-sufficient-cond-ext}
Fix a condition~$c = (\alpha, \vec{F}, \Ccal)$ where $\Ccal$ is a $k$-cover class of~$[t, +\infty)$. 
Let $E_0, \dots, E_{m-1}$ be finite sets, $\Dcal$ be a non-empty $\Pi^{0,C}_1$ $m$-cover class of $[t', +\infty)$
for some~$t' \geq t$ and $f : m \to k$ be a function such that for each~$i < |\alpha|$ and $\nu < m$,
\begin{itemize}
	\item[(iii)] $E_\nu$ is $R_i$-transitive
	\item[(iv)] $V_\nu \to_{R_i} E_\nu$ or $E_\nu \to_{R_i} V_\nu$ for each $V_0 \oplus \dots \oplus V_{m-1} \in \Dcal$
\end{itemize}
Set $H_\nu = F_{f(\nu)} \cup E_\nu$ for each $\nu < m$.
If properties (i) and (ii) of an extension are satisfied for~$d = (\alpha, \vec{H}, \Dcal)$ with witness $f$,
then~$d$ is a valid condition extending~$c$.
\end{lemma}
\begin{proof}
All we need is to check properties (b) and (c) for~$d$ in the definition of a condition.
We prove property (b). Fix an $i < |\alpha|$, some part $\nu$ of $\Dcal$, and an $x \in V_\nu$
for some $V_0 \oplus \dots \oplus V_{m-1} \in \Dcal$. In order to prove that 
$(F_{f(\nu)} \cup E_\nu) \setminus [0, \alpha(i)) \cup \{x\}$
is $R_i$-transitive, it is sufficient to check that the set contains no 3-cycle.
Fix three elements $u < v < w \in (F_{f(\nu)} \cup E_\nu) \setminus [0, \alpha(i)) \cup \{x\}$.
\begin{itemize}
	\item Case 1: $\{u, v, w\} \cap F_{f(\nu)} \setminus [0, \alpha(i)) \neq \emptyset$. 
	Then $u \in F_{f(\nu)} \setminus [0, \alpha(i))$ as $F_{f(\nu)} < E_\nu < \{x\}$ and $u < v < w$.
	By property (ii), there is some $Z_0 \oplus \dots \oplus Z_{k-1} \in \Ccal$ such that $E_\nu \cup \{x\} \subseteq Z_{f(\nu)}$.
	If $v \in F_{f(\nu)}$, then by property (b) of the condition~$c$ on~$Z_{f(\nu)}$, $\{u, v, w\}$ is $R_i$-transitive.
	If $v \not \in F$, then by Lemma~\ref{lem:em-comp-reduc-uniform-behaviour}, $\{u\} \to_{R_i} Z_{f(\nu)}$
	or $Z_{f(\nu)} \to_{R_i} \{u\}$, so $\{u, v, w\}$ is $R_i$-transitive since~$v, w \in Z_{f(\nu)}$.

	\item Case 2: $\{u, v, w\} \cap  F_{f(\nu)} \setminus [0, \alpha(i)) = \emptyset$. 
	Then at least $u, v \in E_\nu$ because $E_\nu < \{x\}$.
	If $w \in E_\nu$ then $\{u, v, w\}$ is $R_i$-transitive by $R_i$-transitivity of $E_\nu$.
	In the other case, $w = x \in V_\nu$. As $E_\nu \to_{R_i} V_\nu$ or $V_\nu \to_{R_i} E_\nu$,
	$\{u, v\} \to_{R_i} \{w\}$ or $\{w\} \to_{R_i} \{u, v\}$ and $\{u, v, w\}$ is $R_i$-transitive.
\end{itemize}

We now prove property (c) for $d$. Fix some $V_0 \oplus \dots \oplus V_{m-1} \in \Dcal$, 
some part $\nu$ of~$\Dcal$ and some $i < |\alpha|$.
By property (ii), there is some~$Z_0 \oplus \dots \oplus Z_{k-1} \in \Ccal$ such that $E_\nu \cup V_\nu \subseteq Z_{f(\nu)}$.
By property (c) of the condition~$c$, $Z_{f(\nu)}$ (and so $V_\nu$) is included in a minimal $R_i$-interval $(u, v)$ of 
$F_{f(\nu)} \setminus [0, \alpha(i))$.
Here again, $u$ and $v$ may be respectively $-\infty$ and $+\infty$. 
By assumption, either $E_\nu \to_{R_i} V_\nu$ or $V_\nu \to_{R_i} E_\nu$. As $E_\nu$ is a finite $R_i$-transitive set,
it has a minimal and a maximal element, say~$x$ and~$y$. If $E_\nu \to_{R_i} V_\nu$
then $V_\nu$ is included in the $R_i$-interval $(y, v)$.
Symmetrically, if $V_\nu \to_{R_i} E_\nu$ then 
$V_\nu$ is included in the $R_i$-interval $(u, x)$.
To prove minimality for the first case, assume that some $w$ is in the interval $(y, v)$.
Then $w \not \in F_{f(\nu)} \setminus [0, \alpha(i))$ by minimality of the interval $(u, v)$ with respect to 
$F_{f(\nu)} \setminus [0, \alpha(i))$, and $w \not \in E_\nu$ by maximality of~$y$.
Minimality for the second case holds by symmetry.
\end{proof}

Now we have settled the necessary technical lemmas, we start proving
lemmas which will be directly involved in the construction of the transitive subtournament.
The following simple progress lemma states that we can always find an extension of a condition
in which we increase both the finite approximations corresponding to the acceptable parts 
and the number of tournaments for which we are transitive simultaneously.
Moreover, this extension can be found uniformly.

\begin{lemma}[Progress]\label{lem:em-comp-reduc-ext}
For every condition~$c = (\alpha, \vec{F}, \Ccal)$ and every $s \in \omega$,
there exists an extension $d = (\beta, \vec{E}, \Dcal)$ such that $|\beta| \geq s$ and
$|E_\nu| \geq s$ for every acceptable part $\nu$ of~$\Dcal$.
Furthermore, such an extension can be found $C'$-effectively, uniformly in~$c$ and~$s$.
\end{lemma}
\begin{proof}
Fix a condition $c = (\alpha, \vec{F}, \Ccal)$.
First note that for every $\beta \succeq \alpha$ such that $\beta(i) > max(F_\nu : \nu < parts(\Ccal))$
whenever $|\alpha| \leq i < |\beta|$, $(\beta, \vec{F}, \Ccal)$ is a condition extending~$c$.
Therefore it suffices to prove that for every such condition~$c$ and every part $\nu$ of $\Ccal$,
we can $C'$-effectively find a condition~$d = (\alpha, \vec{H}, \Dcal)$ refining~$c$
with witness~$f : parts(\Dcal) \to parts(\Ccal)$ such that $f$ forks only parts refining part $\nu$ of $\Ccal$,
and either every such part $\mu$ of $\Dcal$ is empty or $|H_\mu| > |F_\nu|$.
Iterating the process finitely many times enables us to conclude.

Fix some part $\nu$ of $\Ccal$ and let~$\Dcal$ be the collection of $Z_0 \oplus \dots \oplus Z_{k-1} \in \Ccal$
such that $Z_\nu = \emptyset$. We can $C'$-decide whether or not $\Dcal$ is empty.
If $\Dcal$ is non-empty, then $(\alpha, \vec{F}, \Dcal)$ is a valid extension of~$c$
with the identity function as witness and such that part $\nu$ of $\Dcal$ is empty.
If $\Dcal$ is empty, we can $C'$-computably find some $Z_0 \oplus \dots \oplus Z_{k-1} \in \Ccal$
and pick some~$x \in Z_\nu$.
Consider the $C$-computable $2^{|\alpha|}$-partition $(X_\rho : \rho \in 2^{|\alpha|})$ of $\omega$ defined by
$$
X_\rho = \{ y \in \omega : (\forall i < |\alpha|)[R_i(y, x) \leftrightarrow \rho(i) = 1] \} 
$$
Let $\tilde{\Dcal}$ be the cover class refining $\Ccal^{[\nu, x]}$ such that
part $\nu$ of $\tilde{\Dcal}$ has $2^{|\alpha|}$ forks induced by the
$2^{|\alpha|}$-partition~$\vec{X}$. Define $\vec{H}$ by
$H_\mu = F_\mu$ if $\mu$ refines a part different from $\nu$,
and $H_\mu = F_\nu \cup \{x\}$ if $\mu$ refines part $\nu$ of~$\Ccal$.
The forking according to~$\vec{X}$ ensures that property (iv) of Lemma~\ref{lem:em-comp-reduc-sufficient-cond-ext} holds.
By Lemma~\ref{lem:em-comp-reduc-sufficient-cond-ext}, $d = (\alpha, \vec{H}, \tilde{\Dcal})$ is a valid extension of~$c$.
\end{proof}

\subsection{The strategy}

Thanks to Lemma~\ref{lem:em-comp-reduc-ext}, we can define an infinite, $C'$-computable
decreasing sequence of conditions $(\varepsilon, \emptyset, \{\omega\}) \geq c_0 \geq c_1 \geq \dots$
such that for each~$s \in \omega$, 
\begin{itemize}
	\item[1.] $|\alpha_s| \geq s$.
	\item[2.] $|F_{s, \nu}| \geq s$ for each acceptable part~$\nu$ of~$\Ccal_s$
\end{itemize}
where $c_s = (\alpha_s, \vec{F}_s, \Ccal_s)$.
As already noticed, if some acceptable part $\mu$ of $\Ccal_{s+1}$ refines some part $\nu$ of~$\Ccal_s$,
part $\nu$ of~$\Ccal_s$ is also acceptable.
Therefore, the set of acceptable parts forms an infinite, finitely branching $C'$-computable tree~$\Tcal$.
Let $P$ be any infinite path through~$\Tcal$. 
The set $H(P) = (\bigcup_s F_{s, P(s)})$ is infinite,
and $H(P) \setminus [0, \alpha_{i+1}(i))$ is $R_i$-transitive for each $i \in \omega$.

Our goal is to build a $C'$-computable function dominating every function computed
by $H(P)$ for at least one path $P$ trough~$\Tcal$. However, it requires too much
computational power to distinguish acceptable parts from non-acceptable ones,
and even some acceptable part may have only finitely many extensions. Therefore,
we will dominate the functions computed by~$H(P)$ for \emph{every} path $P$ trough~$\Tcal$.
 
At a finite stage, a condition contains finitely many parts, each one representing
the construction of a transitive subtournament.
As in the construction of a cohesive set, it suffices to check one by one whether 
there exists an extension of our subtournaments which will
make terminate a given functional at a given input.
In the next subsection, we develop the framework necessary to decide such a termination
at a finite stage.

\subsection{Forcing relation}

As a condition $c = (\alpha, \vec{F}, \Ccal)$ corresponds to the construction of multiple
subtournaments $F_0, F_1, \dots$ at the same time, the forcing relation will depend on which
subtournament we are considering. In other words, the forcing relation depends on the part $\nu$ of~$\Ccal$
we focus on.

\begin{definition}\label{def:em-comp-reduc-forcing-relation}
Fix a condition $c = (\alpha, \vec{F}, \Ccal)$, a part~$\nu$ of~$\Ccal$ and two integers~$e$, $x$.
\begin{itemize}
	\item[1.] $c \Vdash_\nu \Phi_e^{G \oplus C}(x) \uparrow$ if $\Phi_e^{(F_\nu \cup F_1) \oplus C}(x) \uparrow$
	for all $Z_0 \oplus \dots \oplus Z_{k-1} \in \Ccal$ and all subsets $F_1 \subseteq Z_\nu$
	such that $F_1$ is $R_i$-transitive simultaneously for each $i < |\alpha|$.
	\item[2.] $c \Vdash_\nu \Phi_e^{G \oplus C}(x) \downarrow$ if $\Phi_e^{F_\nu \oplus C}(x) \downarrow$.
\end{itemize}
\end{definition}

The forcing relations defined above satisfy the usual forcing properties.
In particular, let $c_0 \geq c_1 \geq \dots$ be an infinite decreasing
sequence of conditions. This sequence induces an infinite, finitely branching tree of acceptable parts~$\Tcal$.
Let~$P$ be an infinite path trough~$\Tcal$. If 
$c_s \Vdash_{P(s)} \Phi_e^{G \oplus C}(x) \uparrow$ (respectively $c_s \Vdash_{P(s)} \Phi_e^{G \oplus C}(x) \downarrow$)
at some stage~$s$, then $\Phi_e^{H(P) \oplus C}(x) \uparrow$ (respectively $\Phi_e^{H(P) \oplus C}(x) \downarrow$).

Another important feature of this forcing relation is that we can decide $C'$-uniformly in its parameters
whether there is an extension forcing~$\Phi^{G \oplus C}_e(x)$ to halt or to diverge. 
Deciding this relation with little computational power is useful because our $C'$-computable dominating function will
need to decide termination $\Gamma^{G \oplus C}(x)$ to check whether it has to dominate the value 
outputted by $\Gamma^{G \oplus C}(x)$.

\begin{lemma}\label{lem:em-comp-reduc-force-dense}
For every condition~$c = (\alpha, \vec{F}, \Ccal)$ and every pair of integers $e, x \in \omega$,
there exists an extension~$d = (\alpha, \vec{H}, \Dcal)$ such that for each part~$\nu$ of~$\Dcal$
$$
d \Vdash_\nu \Phi_e^{G \oplus C}(x) \uparrow \hspace{10pt} \vee \hspace{10pt} d \Vdash_\nu \Phi_e^{G \oplus C}(x) \downarrow
$$
Furthermore, such an extension can be found $C'$-effectively, uniformly in~$c$, $e$ and~$x$.
\end{lemma}
\begin{proof}
Given a condition~$c$ and two integers $e, x \in \omega$,
let $I_{e,x}(c)$ be the set of parts $\nu$ of~$c$
such that $c \not \Vdash_\nu \Phi_e^{G \oplus C}(x) \downarrow$ and $c \not \Vdash_\nu \Phi_e^{G \oplus C}(x) \uparrow$.
Note that $I_{e,x}(c)$ is $C'$-computable uniformly in~$c$, $e$ and~$x$.
It suffices to prove that given such a condition~$c$ and a part~$\nu \in I_{e,x}(c)$, one can $C'$-effectively
find an extension~$d$ with witness $f$ such that $f(I_{e,x}(d)) \subseteq I_{e,x}(c) \setminus \{\nu\}$.
Applying iteratively the operation enables us to conclude.

Fix a condition~$c = (\alpha, \vec{F}, \Ccal)$ where $\Ccal$ is a $k$-cover class, and fix some part~$\nu \in I_{e,x}(c)$.
The strategy is the following: either we can fork part~$\nu$ of $\Ccal$ into enough parts so that we 
force~$\Phi_e^{G \oplus C}(x)$ to diverge
on each forked part, or we can find an extension forcing $\Phi_e^{G \oplus C}(x)$ to converge on part~$\nu$ without forking.
Hence, we ask the following question.

\smallskip
{\itshape
Q2: Is it true that for every $k$-cover~$Z_0 \oplus \dots \oplus Z_{k-1} \in \Ccal$,
for every~$2^{|\alpha|}$-partition $\bigcup_{\rho \in 2^\alpha} X_\rho = Z_\nu$,
there is some~$\rho \in 2^{|\alpha|}$ and some finite set~$F_1$ which is $R_i$-transitive
for each~$i < |\alpha|$ simultaneously, and such that~$\Phi_e^{(F_\nu \cup F_1) \oplus C}(x) \downarrow$?
}
\smallskip

If the answer is no, then by forking the part~$\nu$ of~$\Ccal$ into $2^{|\alpha|}$ parts,
we will be able to force~$\Phi_e^{G \oplus C}(x)$ to diverge.
Let~$m = k+2^{|\alpha|}-1$ and define the function~$f : m \to k$ by $f(\mu) = \mu$ 
if $\mu < k$ and $f(\mu) = \nu$ otherwise.
Let~$\Dcal$ be the collection of all~$m$-covers $V_0 \oplus \dots \oplus V_{m-1}$ which $f$-refine
some $Z_0 \oplus \dots \oplus Z_{k-1} \in \Ccal$ and such that for every part $\mu$ of $\Dcal$ 
$f$-refining part~$\nu$ of $\Ccal$ and every subset $F_1 \subseteq V_\mu$
which is $R_i$-transitive simultaneously for each~$i < |\alpha|$, $\Phi_e^{F_\nu \cup F_1}(x) \uparrow$.
Note that~$\Dcal$ is a $\Pi^{0,C}_1$ $m$-cover class $f$-refining $\Ccal$. Moreover~$\Dcal$
is non-empty since the answer to~{\itshape Q2} is no.
Let~$\vec{E}$ be defined by $E_\mu = F_\mu$ if $\mu < k$
and $E_\mu = F_\nu$ otherwise. The condition~$d = (\alpha, \vec{E}, \Dcal)$ extends~$c$
with witness~$f$. For every part~$\mu$ of $\Dcal$ $f$-refining part $\nu$ of $\Ccal$, $d \Vdash_\mu \Phi_e^{G \oplus C}(x) \uparrow$,
therefore $f(I_{e,x}(d)) \subseteq I_{e,x}(c) \setminus \{\nu\}$.

Suppose now that the answer is yes. By compactness, we can $C'$-effectively find a finite set~$E \subseteq Z_\nu$
for some~$Z_0 \oplus \dots \oplus Z_{k-1} \in \Ccal$ such that for every $2^{|\alpha|}$-partition $(E_\rho : \rho \in 2^{|\alpha|})$
of $E$, there is some $\rho \in 2^{|\alpha|}$ and some set $F_1 \subseteq E_\rho$ which is $R_i$-transitive
simultaneously for each $i < |\alpha|$ and such that $\Phi_e^{(F_\nu \cup F_1) \oplus C}(x) \downarrow$.
There are finitely many $2^{|\alpha|}$-partitions of $E$. Let~$n$ be the number of such partitions. 
These partitions induce a finite $C$-computable $n$-partition of~$dom(\Ccal)$
defined for each $(E_\rho : \rho \in 2^{|\alpha|})$ by
$$
X_{\tuple{E_\rho : \rho \in 2^{|\alpha|}}} = \left\{ y \in dom(\Ccal) : (\forall i < |\alpha|) 
	\cond{
	\mbox{ if } \rho(i) = 0 \mbox{ then } E_\rho \to_{R_i} \{y\} \\ 
	\mbox{ if } \rho(i) = 1 \mbox{ then } \{y\} \to_{R_i} E_\rho} \right\}
$$

Let~$\tilde{\Dcal}$ be the $\Pi^{0,C}_1$ $(k+n-1)$-cover class refining $\Ccal^{[\nu, E]}$
and such that part~$\nu$ of~$\Ccal^{[\nu, E]}$ is refined accordingly to the above partition of~$dom(\Ccal)$.
Let~$f : k+n-1 \to k$ be the refining function witnessing it. 
Define~$\vec{H}$ as follows. For every part~$\mu$ of $\Dcal$, refining part~$\nu$ of $\Ccal^{[\nu, E]}$,
by definition of~$\tilde{\Dcal}$, there is some~$2^{|\alpha|}$-partition $\tuple{E_\rho : \rho \in 2^{|\alpha|}}$ of~$E$
such that for every $V_0 \oplus \dots V_{k+n-2} \in \tilde{\Dcal}$, $V_\mu \subseteq X_{\tuple{E_\rho : \rho \in 2^{|\alpha|}}}$.
By choice of~$E$, there exists some set $F_1 \subseteq E_\rho$ for some~$\rho \in 2^{|\alpha|}$
which is $R_i$-transitive simultaneously for each $i < |\alpha|$ and such that
$\Phi_e^{(F_\nu \cup F_1) \oplus C}(x) \downarrow$.
This set $F_1$ can be found $C'$-effectively. Set $H_\mu = F_\nu \cup F_1$.
For every part~$\mu$ of $\tilde{\Dcal}$ which refines some part~$\xi$ of $\Ccal^{[\nu, E]}$ different from~$\nu$,
set~$H_\mu = F_\xi$.
By Lemma~\ref{lem:em-comp-reduc-sufficient-cond-ext}, $d = (\alpha, \vec{H}, \tilde{\Dcal})$ is a valid condition
extending~$c$. Moreover, for every part $\mu$ of~$\tilde{\Dcal}$ refining part~$\nu$ of $\Ccal$,
$d \Vdash_\mu \Phi_e^{G \oplus C}(x) \downarrow$. Therefore $f(I_{e,x}(d)) \subseteq I_{e,x}(c) \setminus \{\nu\}$.
\end{proof}

\subsection{Construction}

We are now ready to construct our infinite transitive subtournament~$H(P)$ together
with a $C'$-computable function~$f$ dominating every~$H(P) \oplus C$-computable function.
Thanks to Lemma~\ref{lem:em-comp-reduc-ext} and Lemma~\ref{lem:em-comp-reduc-force-dense}, we can $C'$-compute an infinite
descending sequence of conditions $(\epsilon, \emptyset, 1^{<\omega}) \geq c_0 \geq c_1 \geq \dots$
such that at each stage $s \in \omega$,
\begin{itemize}
	\item[1.] $|\alpha_s| \geq s$
	\item[2.] $|F_{s, \nu}| \geq s$ for each acceptable part~$\nu$ of~$\Ccal_s$
	\item[3.] $c_s \Vdash_\nu \Phi_e^{G \oplus C}(x) \downarrow$ or $c_s \Vdash_\nu \Phi_e^{G \oplus C}(x) \uparrow$
	for each part~$\nu$ of~$\Ccal_s$ if $\tuple{e, x} = s$
\end{itemize}
where $c_s = (\alpha_s, \vec{F}_s, \Ccal_s)$.
Property 1 ensures that the resulting set with be eventually transitive
for every tournament in~$\vec{R}$. Property~2 makes the subtournaments infinite.
Last, property 3 enables us to $C'$-decide at a finite stage whether a functional terminates on a given
input, with the transitive subtournament as an oracle.

Define the $C'$-computable function $f : \omega \to \omega$ as follows:
On input~$x$, the function~$f$ looks at all stages~$s$ such that $s = \tuple{e,x}$ for
some $e \leq x$. For each such stage~$s$, and each part~$\nu$ in~$\Ccal_s$,
the function $C'$-decides whether $c_s \Vdash_\nu \Phi^{G \oplus C}_e(x) \downarrow$
or $c_s \Vdash_\nu \Phi^{G \oplus C}_e(x) \uparrow$. 
In the first case, $f$ computes the value $\Phi^{F_{s, \nu} \oplus C}_e(x)$.
Having done all that, $f$ returns a value greater than the maximum of the computed values.

Fix any infinite path~$P$ trough the infinite tree $\Tcal$ of the acceptable parts induced
by the infinite descending sequence of conditions. 
We claim that $f$ dominates every function computed by~$H(P) \oplus C$.
Fix any Turing index $e \in \omega$ such that $\Phi_e^{H(P) \oplus C}$ is total.
Consider any input~$x \geq e$ and the corresponding stage $s = \tuple{e,x}$. 
As $\Phi_e^{H(P) \oplus C}$ is total, $c_s \not \Vdash_{P(s)} \Phi_e^{G \oplus C}(x) \uparrow$,
hence by property 3, $c_s \Vdash_{P(s)} \Phi_e^{G \oplus C}(x) \downarrow$.
By construction, $f(x)$ computes the value of $\Phi_e^{F_{s,P(s)} \oplus C}(x)$ and returns
a greater value. As $F_{s,P(s)}$ is an initial segment of $H(P)$, 
$\Phi_e^{F_{s,P(s)} \oplus C}(x) = \Phi_e^{H(P) \oplus C}(x)$
and therefore $f(x) > \Phi_e^{H(P) \oplus C}(x)$.
This completes the proof of~$\amt \not \leq_c \emo$.

We identify a $k$-cover $Z_0 \cup \dots \cup Z_{k-1}$ of some set $X$ with the $k$-fold join of its parts

\section{The domination framework}

The actual proof of Theorem~\ref{thm:amt-comp-reduc-em-coh} is slightly stronger than its statement
as it creates a degree~$\dbf$ bounding~$\emo$ together with a computable instance~$X$ of~$\amt$
such that $\dbf$ bounds no solution to~$X$. Therefore, having solutions to multiple tournaments in parallel
is not enough to compute a solution to~$X$.
One may however ask whether \emph{sequential} applications of~$\emo$
(that is, defining a tournament such that every transitive subtournament will be used to define another tournament
and so on) is enough to compute a solution to~$X$.

Answering negatively this question requires to diagonalize against solutions $Y_0$ to computable instances of~$\emo$,
but also against solutions $Y_1$ to $Y_0$-computable instances of~$\emo$ and so on. The difficulty comes
from the fact that diagonalizations happen at finite stages, at which we have only access to a finite approximation
of~$Y_0$, and so to a finite part of the $Y_0$-computable instances of~$\emo$.
Thankfully, we only need a finite piece of an $\emo$-instance to diagonalize against its solutions.

In this section, we develop a framework for building an $\omega$-structure~$\Mcal$ satisfing some principle $\Psf$
such that every function in $\Mcal$ is dominated by a single $\emptyset'$-computable function.
Since by definition, the first-order part of an $\omega$-structure is the set of standard natural numbers,
$\omega$-structures are characterized by their second-order part. An $\omega$-structure
satisfies $\rca$ if and only if its second-order part is a Turing ideal, i.e.,
a set of reals $\Ical$ closed under the effective join and the Turing reduction.

The whole construction will be done by iterating uniformly and $\emptyset'$-effectively
the forcing constructions presented in the previous sections.
We will not directly deal with the concrete forcing notion used for constructing solutions
to $\emo$-instances. Instead, we will manipulate an abstract partial order of forcing conditions.
Abstracting the construction has several advantages:
\begin{itemize}
	\item[1.] It enables the reader to focus on the operations which are the essence of the construction.
	The reader will not be distracted by the implementation subtleties of $\emo$ which are not insightful 
	to understand the overall structure.
	\item[2.] The construction is more modular. We will be able to implement modules for~$\emo$
	and~$\wkl$ independently, and combine them in section~\ref{sect:separating-amt-combined} 
	to obtain a proof that~$\emo \wedge \allowbreak \wkl$ does not imply~$\amt$, 
	and this without changing the main construction.
	This also enable reverse mathematicians to prove that other principles do not imply~$\amt$ without having to reprove
	the administrative aspects of the construction.
\end{itemize}

We shall illustrate our definitions with the case of~$\coh$ in order to give a better
intuition about the abstract operators we will define. As explained in section~\ref{sect:emo-computable-reducibility}, 
the separation of~$\coh$ from~$\amt$
is already a consequence of the separation of~$\emo$ from~$\amt$. Therefore implementing the framework
with~$\coh$ is only for illustration purposes.

\subsection{Support}

The first step consists of defining the abstract partial order which will represent the partial
order of forcing conditions. We start with an analysis of the common aspects of the different
forcing notions encountered until yet, in order to extract their essence and define the abstract operators.
In the following, we shall employ \emph{stage} to denote a temporal step in the construction.
An \emph{iteration} is a spatial step representing progress in the construction of the Turing ideal.
Multiple iterations are handled at a single stage. 
\smallskip

\emph{Parts of a condition}. When constructing cohesive sets for~$\coh$
or transitive subtournaments for~$\emo$, we have been working in both cases with \emph{conditions}
representing parallel Mathias conditions. We shall therefore associate to our abstract notion of condition
a notion of~\emph{part} representing one of the solutions we are building.
A single abstract condition will have multiple \emph{parts} representing the various candidate
solutions constructed in parallel for the same instance.

For example, in the forcing notion for~$\coh$, a condition~$c = (F_\nu : \nu \in 2^n)$
can be seen as $2^n$ parallel Mathias conditions $(F_\nu, R_\nu)$ where $R_0, R_1, \dots$
is the universal instance of $\coh$. In this setting, the parts of~$c$ are the pairs $\tuple{c, \nu}$ for each~$\nu \in 2^n$.
One may be tempted to get rid of the notion of condition and directly deal with its parts
since in $\coh$, a condition is only the tuple of its parts.
However, in the forcing notion $c = (\vec{F}, \Ccal)$ for~$\emo$, the parts are interdependent
since adding element to some~$F_\nu$ will remove inconsistent covers from~$\Ccal$
and therefore may restrict the reservoirs of the other parts.
\smallskip

\emph{Satisfaction}. As explained, a part represents the construction of one solution, whereas a condition
denotes multiple solutions in parallel. We can formalize this intuition by defining
a \emph{satisfaction function} which, given a part of a condition, returns the collection of the sets
satisfying it. For example, a set~$G$ satisfies part~$\nu$ of the $\coh$ condition~$c = (F_\nu : \nu \in 2^n)$
if it satisfies the Mathias condition $(F_\nu, R_\nu \setminus [0, max(F)])$,
in other words, if $F_\nu \subseteq G$ and $G \setminus F_\nu \subseteq R_\nu \setminus [0, max(F_\nu)]$.
\smallskip

\begin{figure}[htbp]
\begin{center}
\scalebox{0.9}{
\begin{tikzpicture}[x=1.5cm, y=1cm, thick,
		node/.style={inner sep=3pt, minimum size=1.5em, font=\large{#1}},
		stage/.style={node,rotate=25,opacity=0.7},
		label/.style={node,rotate=-3.5,opacity=0.7}
]

  \node[node] (c0) at (0, 10) {$c_0$};
	\node[node] (c1) at (2, 9.8) {$c_1$};
	\node[node] (c2) at (4, 9.6) {$c_2$};
	\node[node] (c3) at (6, 9.4) {$c_3$};
	
	\node[node] (d00) at (1, 8) {$d_{0,0}$};
	\node[node] (d01) at (2.5, 8.6) {$d_{0,1}$};
	\node[node] (d10) at (3, 7.8) {$d_{1,0}$};
	\node[node] (d11) at (4.5, 8.4) {$d_{1,1}$};
	\node[node] (d20) at (5, 7.2) {$d_{2,0}$};
	\node[node] (d21) at (5.8, 7.7) {$d_{2,1}$};
	\node[node] (d22) at (6.5, 8.2) {$d_{2,2}$};

	\node[node] (e00) at (5, 5.7) {$e_{0,0}$};
	\node[node] (e01) at (5.8, 6.2) {$e_{0,1}$};
	\node[node] (e02) at (6.5, 6.7) {$e_{0,2}$};

	\node[label] (t1) at (8, 9.2) {\small First iteration};
	\node[label] (t2) at (8.5, 8) {\small Second iteration};
	\node[label] (t3) at (8.5, 6.5) {\small Third iteration};

	\node[stage] (l1) at (0.5,10.7) {\small Stage 0};
	\node[stage] (l2) at (2.5,10.5) {\small Stage 1};
	\node[stage] (l3) at (4.5,10.3) {\small Stage 2};
	\node[stage] (l4) at (6.5,10.1) {\small Stage 3};
	
	\draw[dotted] (c0) -- (c1) -- (c2) -- (c3) -- (7,9.3);

	\draw (c1) -- (d00);
	\draw (c1) -- (d01);

	\draw (c2) -- (d10);
	\draw (c2) -- (d11);
	\draw[dotted] (d00) -- (d10);
	\draw[dotted] (d01) -- (d11);

	\draw (c3) -- (d20);
	\draw (c3) -- (d21);
	\draw (c3) -- (d22);
	\draw[dotted] (d10) -- (d20);
	\draw[dotted] (d10) -- (d21);
	\draw[dotted] (d11) -- (d22);

	\draw (d20) -- (e00);
	\draw (d21) -- (e01);
	\draw (d22) -- (e02);

	\draw[dotted] (d22) -- (7.5, 8.1);
	\draw[dotted] (d21) -- (6.8, 7.6);
	\draw[dotted] (d20) -- (6, 7.1);
	\draw[dotted] (e00) -- (6,5.6);
	\draw[dotted] (e01) -- (6.8, 6.1);
	\draw[dotted] (e02) -- (7.5, 6.6);

\end{tikzpicture}
}
\end{center}
\caption{Example of construction of a Turing ideal by an iterative forcing
	in which conditions may have multiple parts. The nodes are the conditions,
	the dotted edges are condition extensions and the plain
	edges are the parts of the conditions.} 
\end{figure}
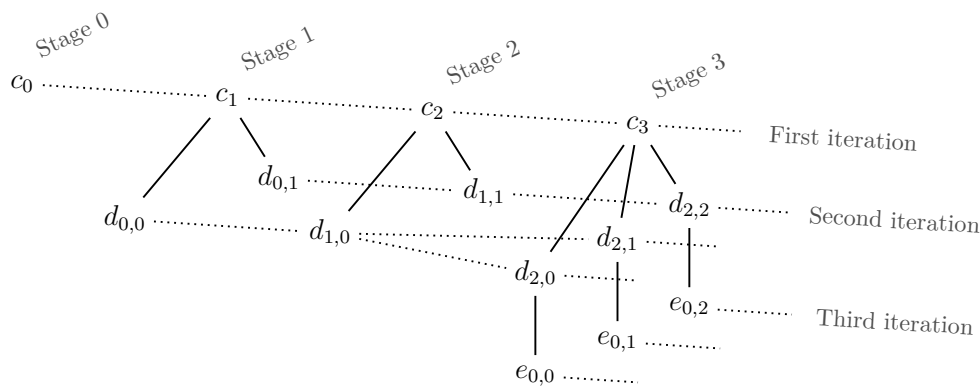

\emph{Initial condition}. 
In a standard (i.e.\ non-iterative) forcing, we build an infinite decreasing sequence
of conditions, starting from one initial condition~$c_0$. In $\coh$, this initial condition is~$(\emptyset, \varepsilon)$,
where~$\varepsilon$ is the empty string. Since~$R_\varepsilon = \omega$, this coincides with the standard
initial Mathias condition~$(\emptyset, \omega)$.
In an iterative forcing, we add progressively new iterations by starting a new decreasing sequence of conditions
below each part of the parent condition. Since $\coh$ admits a universal instance, there is no need to choose
which instance we want to solve at each iteration. However, in the case of~$\emo$, we will take a new $\emo$-instance
each time, so that the resulting Turing ideal is the second-order part of an $\omega$-model satisfying $\emo$.
Therefore, an $\emo$-condition is in fact a condition~$(\vec{F}, \Ccal, R)$ where $R$ is an instance of~$\emo$.
The chosen instance of~$\emo$ will be decided at the initialization of a new iteration and will be preserved
by condition extension. The choice of the instance depends only on the iteration level. Therefore we can
define an initialization function which, given some integer, returns the initial condition together with the chosen instance.
\smallskip

\emph{Parameters}.
The difficulty of the iterative forcing comes from the fact that an instance of the principle~$\Psf$
may depend on the previous iterations. During the construction, the partial approximations
of the previous iterations become more and more precise, enabling the instance at the next iteration to be
defined on a larger domain. In the definition of our abstract partial order, we will use a formal parameter $D$
which will represent the join of the constructed solutions in the previous iterations.
For example, in the formal definition of the partial order for~$\coh$, we will say that
some condition $d = (E_\mu : \mu \in 2^m)$ extends another condition $c = (F_\nu : \nu \in 2^n)$ if~$m \geq n$,
and $(E_\mu, R^D_\mu \setminus [0, max(E_\mu)])$ 
Mathias extends $(F_\nu, R^D_\nu \setminus [0,max(F_\nu)])$ for each~$\nu \preceq \mu$. 
This syntactic constraints has to be understood as $(E_\mu, R^X_\mu \setminus [0, max(E_\mu)])$ 
Mathias extends $(F_\nu, R^X_\nu \setminus [0,max(F_\nu)])$
for every set $X = X_0 \oplus \dots \oplus X_{n-1}$ such that~$X_i$ satisfies the ancestor of~$d$ in the iteration axis
at the $i$th level. In the case of $\coh$, only a finite initial segment of $X$ is needed to witness
the extension.

We are now ready to define the notion of module support. 

\begin{definition}[Module support]
A \emph{Module support} is a tuple~$\tuple{\Pb, \Ub, \parop, \iniop, \satop}$ where
\begin{itemize}
	\item[(1)] $(\Pb, \leq_\Pb)$ is a partial order. The set~$\Pb$ has to be thought of as the set of forcing conditions.
	Therefore, the elements of~$\Pb$ will be called \emph{conditions}.
	\item[(2)] $\Ub$ is a set of \emph{parts}. 
	The notion of part is due to the fact that most of our forcing conditions represent multiple
	objects built in parallel.
	\item[(3)] $\parop : \Pb \to \Pcal_{fin}(\Ub)$ is a computable function which, given some condition~$c \in \Pb$,
	gives the finite set of parts associated to~$c$.
	\item[(4)] $\iniop : \Nb \to \Pb$ is a computable function which, given some integer $n$ representing
	the iteration level, provides the initial condition of the forcing at the $n$th iteration.
	\item[(5)] $\satop : \Ub \to \Pcal(2^\omega)$ is a function which, given some part~$\nu$ of some condition~$c$,
	returns the collections of sets satisfying it.
\end{itemize}
Furthermore, a module support is required to satisfy the following property:
\begin{itemize}
	\item[(a)] If~$d \leq_\Pb c$ for some~$c, d \in \Pb$, then there is a function~$f : \parop(d) \to \parop(c)$
such that $\satop(\nu) \subseteq \satop(f(\nu))$ for each~$\nu \in \parop(d)$. We may write it $d \leq_f c$
	and say that $f$ is the \emph{refinement function witnessing $d \leq_\Pb c$}.
\end{itemize}
\end{definition}

Given two conditions~$c, d \in \Pb$ such that~$d \leq_f c$, 
we say that $f$ \emph{forks} part~$\nu$ of $c$ if $|f^{-1}(\nu)| \geq 2$. 
This forking notion will be useful in the definition of a module.
Let us illustrate the notion of module support by defining one for~$\coh$.
\smallskip

\emph{Module support for~$\coh$}.
Define the tuple  $\tuple{\Pb, \Ub, \parop, \iniop, \satop}$ as follows:
$\Pb$ is the collection of all conditions~$(F_\nu : \nu \in 2^n)$ where~$F_\nu$ is a finite set of integers.
Given some~$d = (E_\mu : \mu \in 2^m)$ and $c = (F_\nu : \nu \in 2^n)$, $d \leq_\Pb c$ if~$m \geq n$,
and $(E_\mu, R^D_\mu \setminus [0, max(E_\mu)])$ Mathias extends 
$(F_\nu, R^D_\nu \setminus [0, max(F_\nu)])$ for each~$\nu \preceq \mu$.
Let~$\Ub$ be the set of all pairs $\tuple{(F_\nu : \nu \in 2^n), \nu}$ where~$\nu \in 2^n$.
Given some condition~$c = (F_\nu : \nu \in 2^n) \in \Pb$, $\parop(c) = \{\tuple{c, \nu} : \nu \in 2^n\}$.
For every level~$n \geq 0$, $\iniop(n) = (\emptyset, \varepsilon)$.
Define $\satop(\tuple{(F_\nu : \nu \in 2^n), \nu})$ to be the collection of all sets satisfying the Mathias condition
$(F_\nu, R_\nu \setminus [0, max(F_\nu)])$.

We now check that property (a) holds.
Let~$d = (E_\mu : \mu \in 2^m)$ and $c = (F_\nu : \nu \in 2^n)$ be such that $d \leq_\Pb c$. In particular, $m \geq n$.
Define~$f : \Ub \to \Ub$ by $f(\tuple{d, \mu}) = \tuple{c, \mu \uh n}$. 
We claim that~$f$ is a refinement function witnessing~$d \leq_\Pb c$.
$\satop(\tuple{d,\mu})$ is the collection of sets satisfying the Mathias condition 
$\tilde{d} = (E_\mu, R_\mu \setminus [0, max(E_\mu)))$
and $\satop(\tuple{c, \mu \uh n})$ is the collection of sets 
satisfying $\tilde{c} = (F_{\mu \uh n}, R_{\mu \uh n} \setminus [0, max(F_{\mu \uh n})))$.
Since $\mu \uh n \preceq \mu$, $\tilde{d}$ Mathias extends $\tilde{c}$ by definition of~$\satop$.
Considering the Mathias conditions, every set satisfying $\tilde{d}$ satisfies $\tilde{c}$,
so $\satop(\tuple{d,\mu}) \subseteq \satop(f(\tuple{d,\mu}))$. Therefore the property (a) of a module support is satisfied.

\subsection{Modules}

We previously defined the abstract structure we shall use as a support of the construction.
The next step consists of enriching this structure with a few more operators which will enable us to decide $\Sigma^0_1$ properties
over the constructed sets. The success or failure in forcing some property will depend on the parts of a condition. Note that at a finite stage, we handle a finite tree of conditions. We can therefore cover all cases by asking finitely many questions.
Let us go back to the $\coh$ example, and more precisely how we decided $\Sigma^0_1$ properties over it.

\emph{Iteration 1}. At the first iteration, we would like to decide whether
the $\Sigma^{0,G}_1$ formula
\[
\psi(G) = (\exists s, m)(\Phi^{G}_{e,s}(n) \downarrow = m)
\]
will hold, where $G$ is a formal parameter denoting the constructed set.
Furthermore, we want to collect the value of $\Phi^G_e(n)$ if it halts.
The formula $\psi(G)$ can be seen as a \emph{query}, whose \emph{answers} are
either $\tuple{\no}$ if~$\psi(G)$ does not hold, or a tuple $\tuple{\yes, s, m}$
such that $\Phi^G_{e,s}(n) \downarrow = m$ if~$\psi(G)$ holds.
Given some condition~$c = (F_\nu : \nu \in 2^n)$, we can ask on each part~$\tuple{c,\nu}$
whether the formula $\varphi(G)$ will hold or not, by \emph{boxing} the query $\psi(G)$
into a $\Sigma^0_1$ query $\phi$ without the formal parameter~$G$, such that $\phi$ holds if and only if
we can find an extension $d$ of~$c$ forcing $\psi(G)$ on the parts of~$d$
refining part $\nu$ of~$c$. Concretely, we can define~$\phi$ as follows:
\[
\phi = (\exists F_1 \subseteq R_\nu \setminus [0, max(F_\nu)])\psi(F_\nu \cup F_1)
\]
This query can be $\emptyset'$-decided. If the formula $\phi$ holds, 
we can effectively find some answer to~$\phi$, that is, 
a tuple $\tuple{\yes, F_1, s, m}$ such that $F_1 \subseteq R_\nu \setminus [0, max(F_\nu)]$
and $\Phi^{F_\nu \cup F_1}_{e,s}(n) \downarrow = m$. The extension $d$ obtained by adding $F_1$ to $F_\nu$
forces $\psi(G)$ to hold for every set $G$ satisfying the part~$\tuple{d,\nu}$ of the condition~$d$.
The answer to~$\psi(G)$ is obtained by forgetting the set~$F_1$ from the answer to~$\phi$.
On the other hand, if the formula $\phi$ does not hold, the answer is~$\tuple{\no}$
and $c$ already forces $\psi(G)$ not to hold.
\smallskip

\emph{Iteration 2}. At the second iteration, we work with conditions $c_1 = (E_\mu : \mu \in 2^m)$
which are below some part~$\nu$ of some condition~$c_0 = (F_\nu : \nu \in 2^n)$ living at the first iteration level.
We want to decide $\Sigma^{0,G_0, G_1}_1$ formulas, where $G_0$ and $G_1$
are formal parameters denoting the sets contructed at the first iteration and at the second iteration, respectively.
We basically want to answer queries of the form
$$
\varphi(G_0, G_1) = (\exists s, m)(\Phi^{G_0 \oplus G_1}_{e,s}(n) \downarrow = m)
$$
We will ask this question on each part of~$c_1$.
By the same boxing process as before applied relative to~$c_1$, we obtain a formula $\psi(G_0)$ getting
rid of the formal parameter~$G_1$, and defined by
\[
\psi(G_0) = (\exists E_1 \subseteq R^{G_0}_\mu \setminus [0, max(E_\mu)])\varphi(G_0, E_\mu \cup E_1)
\]
The formula $\psi(G_0)$ is a now a query at the first iteration level. We can apply
another boxing to~$\psi(G_0)$ relative to~$c_0$ to obtain a $\Sigma^0_1$
formula $\phi$ without any formal parameter.
\[
\phi = (\exists F_1 \subseteq R_\nu \setminus [0, max(F_\nu)])\psi(F_\nu \cup F_1)
\]
This formula can again be $\emptyset'$-decided. If it holds,
an answer $a = \tuple{\yes, F_1, E_1, s, m}$ can be given. At the first iteration level,
we unbox the answer $a$ to obtain a tuple $b = \tuple{\yes, E_1, s, m}$ and an extension $d_0$ of~$c_0$.
The extension $d_0$ forces the tuple~$b$ to answer the query $\psi(G_0)$
and is obtained by adding $F_1$ to~$F_\nu$.
At the second iteration level, we unbox again the answer $b$ to obtain a tuple~$\tuple{\yes, s,m}$
and an extension~$d_1$ to~$c_1$, forcing~$\tuple{\yes, s,m}$ to answer the query $\varphi(G_0, G_1)$.
The whole decision process is summarized in Figure~\ref{fig:boxing-sequence}.

\begin{figure}[htbp]
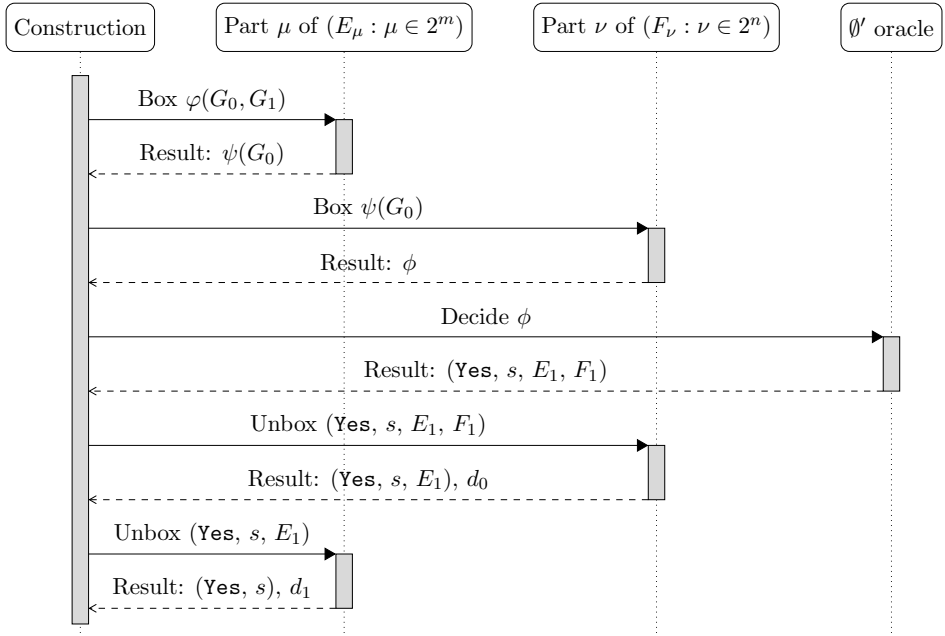

\begin{center}
\scalebox{0.8}{
	\begin{sequencediagram}
		\def\unitfactor{0.9}
		\tikzstyle{inststyle}=[rectangle, draw, anchor=west, minimum
    height=0.8cm, minimum width=1.6cm, fill=white,rounded corners]
    \newthread{cl}{Construction}{Construction}
    \newinst[1]{dpart}{Part $\mu$ of $(E_\mu : \mu \in 2^m)$}{SimControlNode}
    \newinst[1]{cpart}{Part $\nu$ of $(F_\nu : \nu \in 2^n)$}{PhysicsServer}
    \newinst[1]{oracle}{$\emptyset'$ oracle}{SenseServer}

    \begin{call}{cl}{Box $\varphi(G_0, G_1)$}{dpart}{Result: $\psi(G_0)$}
		\end{call}
		\begin{call}{cl}{Box $\psi(G_0)$}{cpart}{Result: $\phi$}
		\end{call}
		\begin{call}{cl}{Decide $\phi$}{oracle}{Result: ($\yes$, $s$, $E_1$, $F_1$)}
		\end{call}
		\begin{call}{cl}{Unbox ($\yes$, $s$, $E_1$, $F_1$)}{cpart}{Result: ($\yes$, $s$, $E_1$), $d_0$}
		\end{call}
		\begin{call}{cl}{Unbox ($\yes$, $s$, $E_1$)}{dpart}{Result: ($\yes$, $s$), $d_1$}
		\end{call}
  \end{sequencediagram}
}

\end{center}
\caption{This sequence diagram shows the boxing process of the $\Sigma^{0,G_0, G_1}_1$ formula
	$\varphi(G_0, G_1) = (\exists s)\Phi^{G_0 \oplus G_1}_{e,s}(x) \downarrow$ into
	a $\Sigma^0_1$ formula without formal parameters in order to decide it.
	The formula $\varphi(G_0, G_1)$ is boxed into a formula  
	$\psi(G_0) = (\exists E_1 \subseteq R^{G_0}_\mu \setminus [0, max(E_\mu)])\varphi(G_0, E_\mu \cup E_1)$
	which is itself boxed into $\phi = (\exists F_1 \subseteq R_\nu \setminus [0, max(F_\nu)])\psi(F_\nu \cup F_1)$.
} \label{fig:boxing-sequence}
\end{figure}

\smallskip
\emph{Progress}. We may also want to force some specific properties required by the principle~$\Psf$.
In the case of Ramsey-type principles, we need to force the set $G$ to be infinite. This can be done
with the following query for each~$k$:
\[
(\exists n)[n \in G \wedge n > k]
\]

The progress query can take various forms, depending on the considered principle. For example,
in $\wkl$, we need to force the path to be infinite by asking the following question for each~$k$:
\[
(\exists \sigma \in 2^k)[\sigma \prec G]
\]
We will therefore define some progress operator which outputs some query that the construction
will force to hold or not. We will choose the actual forcing notions so that the formula
can be forced to hold for at least one part of each condition. The parameter $k$ will not be
given to the operator, since it can be boxed into the current condition, in a monadic style.

We are now ready to define the notion of module as a module support
enriched with some boxing, unboxing and progress abstract operators.
In what follows, $\queryop[\vec{X}]$ is the set of all $\Sigma^0_1$ formulas
with $\vec{X}$ as formal parameters, and $\ansop[\vec{X}]$ is the set of their answers.

\begin{definition}[Module]
A \emph{module} is a tuple~$\tuple{\Sb, \boxop, \unboxop, \progop}$ where
\begin{itemize}
	\item[(1)] $\Sb = \tuple{\Pb, \Ub, \parop, \iniop, \satop}$ is a module support.
	\item[(2)] $\boxop : \Ub \times \queryop[D,G] \to \queryop[D]$ is a computable boxing function which,
	given some part~$\nu$ of some condition~$c \in \Pb$ and some~$\Sigma^0_1$ formula~$\varphi(D,G)$,
	outputs a~$\Sigma^0_1$ formula~$\psi(D)$.
	\item[(3)] $\unboxop : \Ub \times \ansop[D] \to \Pb \times (\Ub \to \Ub) \times (\Ub \to \ansop[D,G])$
	is a computable function which, given some part~$\nu$ of some condition~$c \in \Pb$
	and some answer~$a$ to a $\Sigma^0_1$ formula~$\psi(D)$ encoding a $\Sigma^0_1$ formula~$\varphi(D,G)$, 
	outputs a tuple $\tuple{d, f, g}$ such that 
	$d \leq_f c$ where $f$ forks only part $\nu$ of~$c$,
	and for every part~$\mu$ of~$d$ such that~$f(\mu) = \nu$, 
		and every set~$G \in \satop(\mu)$, $g(\mu)$ is an answer to~$\varphi(D,G)$.
	\item[(3)] $\progop : \Ub \to \queryop[D,G]$ is a computable function which provides a question
	forcing some progress in the solution. It usually asks 
	whether we can force the partial approximation to be defined on a larger domain.
\end{itemize}
\end{definition}

Let us go back to the $\coh$ case.
Define the $\coh$ module $\tuple{\Sb, \boxop, \unboxop, \progop}$ as follows:
$\Sb$ is the $\coh$ module support previously defined.
Given some condition~$c = (F_\nu : \nu \in 2^n)$, some~$\nu \in 2^n$ and some query~$\varphi(D, G)$,
$\boxop(\tuple{c, \nu}, \varphi)$ is the query $\psi(D)$ defined by
\[
\psi(D) = (\exists F_1 \subseteq R^D_\nu 
\setminus [0, max(F_\nu)])\varphi(D, F_\nu \cup F_1)
\]
Set~$\unboxop(\tuple{c,\nu}, \tuple{\no}) = \tuple{c, id, g}$
where $id$ is the identity function and $g(\nu) = \tuple{\no}$.
Given an answer $a = \tuple{\yes, F_1, a_1}$ to the question $\psi(D)$,
$\unboxop(\tuple{c,\nu}, a) = \tuple{d, f, g}$ where $d = (E_\mu : \mu \in 2^n)$
is an extension of~$c$ such that~$E_\nu = F_\nu \cup F_1$, and~$E_\mu = F_\mu$ whenever~$\mu \neq \nu$.
The function~$f : \Ub \to \Ub$ is defined by $f(\tuple{d,\mu}) = \tuple{c,\mu}$ for each~$\mu \in 2^n$.
The function $g : \Ub \to \ansop[D,G]$ is the constant function defined by $g(\tuple{d,\mu}) = \tuple{\yes, a}$.

We claim that $f$ is a refinement function witnessing $d \leq c$. For every~$\mu \neq \nu$,
$E_\mu = F_\mu$ so~$(E_\mu, R^D_\mu \setminus [0, max(E_\mu)])$ 
Mathias extends $(F_\mu, R^D_\mu \setminus [0, max(F_\mu)])$.
By definition of an answer, $F_1 \subseteq R^D_\nu \setminus max(F_\nu))$
so $(E_\nu, R^D_\nu \setminus [0, max(E_\nu)])$ Mathias extends $(F_\nu, R^D_\nu \setminus [0, max(F_\nu)])$.
Therefore $d \leq_\Pb c$.
Last, $\progop(\tuple{c,\nu})$ is the query
\[
\psi(D, G) = (\exists x \in G)[x > max(F_\nu)]
\]

When considering cohesiveness, we must ensure an additional kind of progress. Indeed, we must
partition the reservoir according to $(R_\sigma : \sigma \in 2^n)$
for larger and larger~$n$. We can slightly modify the forcing notion for~$\coh$
and ``hack'' this kind of progress in the $\unboxop$
operator by making it return a condition whose parts are split accordingly.
Since the separation of~$\emo$ from~$\amt$ entails the separation of~$\coh$ from~$\amt$,
we will not go into the details for fixing this progress issue.

\subsection{Construction}\label{subsect:framework-construction}

We will construct an infinite sequence of trees of conditions by stages, such that each level
corresponds to one iteration. We will add progressively more and more iterations,
so that the limit tree is of infinite depth.
In order to simplify the presentation of the construction, we need to introduce some additional
terminology.

\begin{definition}[Stage tree]
A \emph{stage tree} is a finite tree $T$ whose nodes are conditions 
and whose edges are parts of conditions. It is defined inductively as follows:
A \emph{stage tree of depth 0} is a condition. 
A \emph{stage tree of depth~$n+1$} is a tuple $\tuple{c, h}$ where
$c$ is a condition and $h$ is a function such that $h(\nu)$ is a stage tree of depth~$n$ for each~$\nu \in \parop(c)$.
\end{definition}

We consider that the stage subtree $h(\nu)$ is linked to~$c$ by an edge labelled by~$\nu$.
The \emph{root} of~$T$ is itself if~$T$ is a stage tree of depth~0. If $T = \tuple{c,h}$
then the root of $T$ is $c$.
According to our notation on trees, if $T = \tuple{c, h}$,
we write $T^{[\nu]}$ to denote $h(\nu)$.
We also write $T \uh k$ to denote the restriction of~$T$ to its stage subtree of depth $k$.
At each stage $s$ of the construction, we will end up with a stage tree of depth $s$.
The initial stage tree will be $T_0 = \iniop(0)$. There is a natural notion
of stage tree extension induced by the extension of its conditions.

\begin{definition}[Stage tree extension]
A stage tree $T_1$ of depth~$n$ \emph{extends} a stage tree $T_0$ of depth~$0$
if there is a function~$f$ such that $c_1 \leq_f T_0$ where $c_1$ is the root of~$T_1$.
We say that $f$ is a \emph{refinement tree of depth 0} and write $T_1 \leq_f T_0$.
A stage tree $T_1 = \tuple{c_1, h_1}$ of depth~$n+1$ \emph{extends} a stage tree $T_0 = \tuple{c_0, h_0}$
of depth $m+1$ if there is a function $f$ such that $c_1 \leq_f c_0$ and a function $r$ such that
$r(\nu)$ is a refinement tree of depth~$m$
such that $h_1(\nu) \leq_{r(\nu)} h_0(f(\nu))$ for each part~$\nu$ of $c_1$.
The tuple~$R = \tuple{f, r}$ is a \emph{refinement tree of depth~$m+1$}
and we write $T_1 \leq_R T_0$.
\end{definition}

Note that if $T_1 \leq_R T_0$, where $T_1$ is a stage tree of depth $n$
and $T_0$ is a stage tree of depth~$m$, then $n \geq m$.
We may also write $T_1 \leq T_0$ if there is a refinement tree~$R$ of depth~$m$
such that $T_1 \leq_R T_0$.

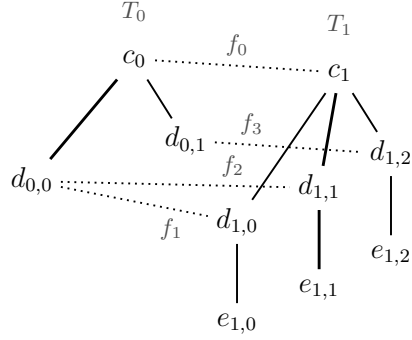
\begin{figure}[htbp]
\begin{center}
\scalebox{0.9}{
\begin{tikzpicture}[x=1.5cm, y=1cm, thick,
		node/.style={inner sep=3pt, minimum size=1.5em, font=\large{#1}},
		stage/.style={node,opacity=0.7},
		label/.style={node,rotate=-3.5,opacity=0.7},
		edgelabel/.style={opacity=0.7}
]

	\node[node] (c2) at (4, 9.6) {$c_0$};
	\node[node] (c3) at (6, 9.4) {$c_1$};

	\node[node] (d10) at (3, 7.8) {$d_{0,0}$};
	\node[node] (d11) at (4.5, 8.4) {$d_{0,1}$};
	\node[node] (d20) at (5, 7.2) {$d_{1,0}$};
	\node[node] (d21) at (5.8, 7.7) {$d_{1,1}$};
	\node[node] (d22) at (6.5, 8.2) {$d_{1,2}$};

	\node[node] (e00) at (5, 5.7) {$e_{1,0}$};
	\node[node] (e01) at (5.8, 6.2) {$e_{1,1}$};
	\node[node] (e02) at (6.5, 6.7) {$e_{1,2}$};

	\node[stage] (l3) at (4,10.3) {\small $T_0$};
	\node[stage] (l4) at (6,10.1) {\small $T_1$};
	
	\draw[dotted] (c2) -- (c3) node [edgelabel, midway, above] {\small $f_0$};

	\draw[very thick] (c2) -- (d10);
	\draw (c2) -- (d11);

	\draw (c3) -- (d20);
	\draw[very thick] (c3) -- (d21);
	\draw (c3) -- (d22);
	\draw[dotted] (d10) -- (d20) node [edgelabel, near end, below] {\small $f_1$};;
	\draw[dotted] (d10) -- (d21) node [edgelabel, near end, above] {\small $f_2$};;
	\draw[dotted] (d11) -- (d22) node [edgelabel, near start, above] {\small $f_3$};;

	\draw (d20) -- (e00);
	\draw[very thick] (d21) -- (e01);
	\draw (d22) -- (e02);

\end{tikzpicture}
}
\end{center}
\caption{In this example, the refinement tree $R$ whose nodes are $\{f_0, f_1, f_2, f_3\}$ witnesses 
the extension of the stage tree $T_0$ whose nodes are $\{c_0, d_{0,0}, d_{0,1}\}$ 
by the stage tree $T_1$ whose nodes are $\{c_1, d_{1,0}, d_{1,1}, d_{1,2}, e_{1,0}, e_{1,1}, e_{1,2}\}$.
The condition~$c_1$ has three parts, and $f_0$-refines the condition~$c_0$ which has two parts.
The conditions $d_{1,0}$, $d_{1,1}$ and~$d_{1,2}$ have only one part.
The path $c_1 \-- d_{1,1} \-- e_{1,1}$ trough the tree $T_1$ $R$-refines the path $c_0 \-- d_{0,0}$
through the tree $T_0$.}
\end{figure}

At each stage, we will extend the current stage tree to a stage tree of larger depth
and whose conditions force more and more properties. The resulting
sequence of stage trees $T_0 \geq T_1 \geq \dots$ can be seen as a 2-dimensional tree
with the following axes:

\begin{itemize}
	\item The \emph{stage axis} is a temporal dimension.
	Let $c_0 \geq c_1 \geq \dots$ be such that $c_s$ is the root of $T_s$ for each stage~$s$.
	As we saw in the computable non-reducibility case, the parts of this sequence forms
	an infinite, finitely branching tree. Let $P$ be any infinite path through this tree.
	More formally, $P$ is a sequence $\nu_0, \nu_1, \dots$ such that $\nu_{s+1}$ is a part of~$c_{s+1}$
	refining the part~$\nu_s$ in~$c_s$ for each~$s$. Consider now the sequence
	of stage trees~$T_0^{[\nu_0]} \geq T_1^{[\nu_1]} \geq \dots$ The sequence lives at the second iteration level,
	below the path~$P$. Its roots induce another infinite, finitely branching tree, and so on.
	Therefore, at each level, we can define an infinite, finitely branching tree of parts,
	once we have fixed the path $P$ through the tree of parts at the previous level.

	\item The \emph{iteration axis} is a spatial (or vertical) dimension corresponding to the depth.
	The notion of stage tree makes explicit the finite tree obtained when fixing a stage.
	A path through a stage tree corresponds to the choices made at each level, between the different
	parts of a condition. We did not define the notion of acceptable part in this framework.
	Therefore, the choice of the part is delegated to the module, which will have to justify that
	at least one of the parts is extensible.
\end{itemize}

\begin{definition}[Partial path]
A \emph{partial path} $\rho$ through a stage tree $T$ of depth~$n$ is defined inductively as follows:
A partial path through a stage tree $T$ of depth 0 is a part of $T$.
A partial path through a stage tree $T = \tuple{c, h}$ of depth $n+1$ is either a part of $c$,
or a sequence $\nu, \rho$ where $\nu$ is a part of $c$ 
and $\rho$ is a partial path through $h(\nu)$.
A \emph{path} through $T$ is a partial path of length~$n+1$.
\end{definition}

We denote by $P(T)$ and by $PP(T)$ the collection of paths and partial paths through $T$, respectively.
Note that a partial path has length at least 1. The notion of refinement between
partial paths is defined in the natural way. We can also extend the notation
$T^{[\rho]}$ to partial paths $\rho$ through~$T$ with the obvious meaning.
There is also a notion of satisfaction of a stage tree induced by the $\satop$ operator.

\begin{definition}[Stage tree satisfaction]
A set $G_0$ \emph{satisfies} a partial path $\nu_0$ 
through a stage tree $T$ of depth~$0$ 
if $G_0 \in \satop(\nu_0)$.
A tuple of sets $G_0, G_1, \dots, G_k$ \emph{satisfies} a partial path $\nu_0, \dots, \nu_k$ 
through a stage tree $T = \tuple{c,h}$ of depth~$n+1$ if $G_0 \in \satop(\nu_0)$
and $k = 0$, or $G_1, \dots, G_k$ satisfies the partial path $\nu_1, \dots, \nu_k$ through
the stage tree $h(\nu_0)$.
A tuple of sets $G_0, G_1, \dots, G_k$ \emph{satisfies} a stage tree $T$ of depth~$n$
if it satisfies a partial path through $T$.
\end{definition}

Again, if $G_0, \dots, G_k$ satisfies a stage tree of depth~$n$,
then $k \leq n$. The notion of satisfaction induces a forcing relation.
We say that $T$ \emph{forces} some formula $\varphi(D, G)$ below a partial path $\rho = \nu_0, \dots, \nu_k$
(written $T \Vdash_\rho \varphi(U, G)$ if for every tuple of sets $G_0, \dots, G_k$
satisfying $\nu_0, \dots, \nu_k$, $\varphi(\bigoplus_{i < k} G_i, G_k)$ holds.

We now prove a few lemmas stating that we can compose locally the abstract operators
to obtain some global behavior. The first trivial lemma shows how to increase the size of a stage tree.
This is where we use the operator $\iniop$.

\begin{lemma}[Growth lemma]\label{lem:growth-lemma}
For every stage tree $T_0$ of depth $n$ and every~$m$,
there is a stage tree $T_1$ of depth $n+1$ such that $T_1 \uh n = T_0$,
and whose leaves are~$\iniop(m)$.
Moreover, $T_1$ can be computably found uniformly in $T_0$.
\end{lemma}
\begin{proof}
The proof is done inductively on the depth of~$T_0$.
In the base case, $T_0$ is a stage tree of depth 0 and is therefore a condition~$c_0$.
Let $h$ be the function such that~$h(\nu) = \iniop(m)$ for each~$\nu \in \parop(c_0)$.
The tuple~$T_1 = \tuple{c_0, h}$ is a stage tree of depth 1 such that $T_1 \uh 0 = c_0 = T_0$.
It can be computably found uniformly in $T_0$ since $\iniop$ and $\parop$ are computable.
Suppose now that $T_0 = \tuple{c_0, h_0}$ is a stage tree of depth $n+1$. By induction hypothesis,
we can define a function $h_1$ such that for each~$\nu \in \parop(c_0)$, $h_1(\nu)$ is a stage tree of depth $n+1$
and $h_1(\nu) \uh n = h_0(\nu)$. The tuple $T_1 = \tuple{c_0, h_1}$ is a stage tree of depth $n+2$
such that $T_1 \uh n+1 = \tuple{c_0, h_0} = T_0$.
\end{proof}

We will always apply the growth lemma in the case~$m = n+1$. However, the full statement
was necessary to apply the induction hypothesis.
Note that, since $T_1 \uh n = T_0$, we have $T_1 \leq T_0$ as witnessed
by taking the refinement tree of identity functions.
The next lemma states that we can, given some stage tree $T_0$
and some query~$\varphi(D, G)$, obtain another stage tree $T_1 \leq T_0$ in which 
we have decided $\varphi(D, G)$ at every part of every condition in $T_0$.
Its proof is non-trivial since when forcing some property, we may increase the number
of branches of the stage tree. We need therefore to define some elaborate decreasing property
to prove termination of the procedure. The query lemma is assumed yet 
and will be proven in subsection~\ref{subsect:queries-detailed}.

\begin{lemma}[Query lemma]\label{lem:query-lemma}
Let $T_0$ be a stage tree of depth~$n$
and ~$q : PP(T) \to \queryop[D,G]$ be a computable function.
There is a stage tree $T_1 \leq T_0$ of depth $n$
such that every partial path $\xi$ through $T_1$ refines a partial path $\rho$ through $T_0$
for which $T_1 \Vdash_\xi q(\rho)$ or $T_1 \Vdash_\xi \neg q(\rho)$,
Moreover, $T_1$ and the function of answers $a : PP(T_1) \to \ansop[D,G]$ can be $\Delta^0_2$-found
uniformly from $T_0$.
\end{lemma}

The following domination lemma is a specialization of the query lemma
by considering queries about termination of programs.

\begin{lemma}[Domination lemma]\label{lem:domination-lemma}
For every stage tree $T_0$ of depth~$n$, there is
a stage of tree $T_1 \leq T_0$ of depth~$n$ and a finite set $U \subset \omega$ such that
for every tuple~$G_0, \dots, G_n$ satisfying $T_1$ and every $e, x, i \leq n$,
$\Phi_e^{G_0 \oplus \dots \oplus G_i}(x) \in U$ whenever $\Phi_e^{G_0 \oplus \dots \oplus G_i}(x)$ halts.
Moreover, $T_1$ and $U$ can be $\Delta^0_2$-found uniformly from $T_0$.
\end{lemma}
\begin{proof}
Apply successively the query lemma with $q(\xi) = (\exists s,m)\Phi_{e,s}^{D \oplus G}(x) \downarrow = m$
for each~$e, x \leq n$, in order to obtain the tree $T_1$ together with an upper bound $k$ to the answers to $q(\rho)$.
We claim that the set $U = [0, k]$ satisfies the desired property.
Let~$G_0, \dots, G_n$ be a tuple satisfying $T_1$, and let~$e, x, i < n$
be such that $\Phi_e^{G_0 \oplus \dots \oplus G_i}(x) \downarrow$.
By definition of satisfaction, there is some partial path $\rho$ through $T_1$
such that $G_0, \dots, G_i$ satisfies $\rho$. By the query lemma,
$T_1 \Vdash_\rho (\exists s,m)\Phi_{e,s}^{D \oplus G}(x) \downarrow = m$
or $T_1 \Vdash_\rho \neg (\exists s,m)\Phi_{e,s}^{D \oplus G}(x) \downarrow = m$.
Since $\Phi_e^{G_0 \oplus \dots \oplus G_i}(x) \downarrow$, the former holds,
and $k$ is greater than the answer to the query, so is greater than $m$.
Uniformity is inherited from the query lemma.
\end{proof}

We construct an infinite $\Delta^0_2$ sequence of finite trees of conditions $T_0 \geq T_1 \geq \dots$ as follows:
At stage 0, we start with a stage tree $T_0$ of depth 0 defined by $\iniop(0)$.
At each stage $s > 0$, assuming we have defined a stage tree $T_{s-1}$ of depth~$s-1$, act as follows:
\begin{itemize}
	\item[(S1)] \emph{Growth}: Apply the growth lemma to obtain a stage tree $T^1_s \leq T_{s-1}$
	of depth $s$.
	Intuitively, this step adds a new iteration and therefore ensures that the construction will have eventually
	infinitely many levels of iteration.

	\item[(S2)] \emph{Progress}: Apply to $T^1_s$ the query lemma with $q = \progop$ to obtain a stage
	tree $T^2_s \leq T^1_s$ such that the progress function is forced at each partial path.
	This step ensures that for every tuple $G_0, G_1, \dots$ such that $G_0, \dots, G_k$
	satisfies each $T_s$, $s \geq k$, the progress query will have been decided
	on $G_i$ infinitely many times.

	\item[(S3)] \emph{Domination}: Apply to $T^2_s$ the domination lemma to obtain
	a stage tree $T_s \leq T^2_s$ and a finite set $U$ such that for every tuple $G_0, \dots, G_s$
	satisfying $T_s$ and every $e, x, i \leq s$, if $\Phi_e^{G_0 \oplus \dots \oplus G_i}(x)$ halts,
	then its value will be in $U$. Since the whole construction is $\Delta^0_2$
	and we uniformly find such a set $U$, this step enables us to define a $\Delta^0_2$
	function which will dominate every function in the Turing ideal.
\end{itemize}

\subsection{Queries}\label{subsect:queries-detailed}

In this section, we develop the tools necessary to prove the query lemma (Lemma~\ref{lem:query-lemma}).
Given some stage tree $T_0$ and some query function $q : PP(T) \to \queryop[D, G]$, the query lemma
states that we can find a stage tree $T_1$ extending $T_0$ and which forces either $q(\rho)$
or its negation on each partial path through $T_1$ refining the partial $\rho$ through $T_0$.
The stage tree $T_0$ is finite and has therefore finitely many partial paths.
The naive algorithm would consist of taking an arbitrary partial path $\rho$ through $T_0$,
then decide $q(\rho)$ thanks to the process illustrated in Figure~\ref{fig:boxing-sequence} and extend $T_0$
into a stage tree $T_1$ which forces $q(\rho)$ or its negation on every path refining $\rho$.
One may expect to obtain the query lemma by iterating this process finitely many times.

The termination of the algorithm depends on the shape of the extension $T_1$ obtained
after deciding $q(\rho)$. We need to ensure that we made some progress so that we will
have covered all paths at some point. Let us look more closely at the construction of the extension $T_1$.
Given some query $\varphi(D, G)$ and some part~$\nu$, we call the $\unboxop(\nu, \varphi)$ operator to obtain another
query $\psi(D)$ getting rid of the forcing variable $G$. Using $\emptyset'$, we obtain an answer $a$
to the formula $\varphi(\emptyset)$ and then call $\unboxop(\nu, a)$ to obtain some extension 
forking only~$\nu$, and forcing either $\varphi(D,G)$ or its negation on every part refining~$\nu$.
This extension may therefore increase the number of parts, but ensures some progress on each of the forked parts.

If $T_0$ is a stage tree of depth 0, the termination of the process is clear.
Indeed, $T_0$ is a condition~$c_0$ and the partial paths through $T_0$ are simply the parts of $c_0$.
We end up with a stage tree $T_1$ of depth 0 corresponding to some condition~$c_1$,
on which we have decided $\varphi(D,G)$ for every part of $c_1$ refining some part~$\nu$ of~$c_0$.
Since we have not forked any other part than $\nu$, the number of undecided parts strictly decreases.
A condition has finitely many parts, so the process terminates after at most $|\parop(c_0)|$
steps.

The progress becomes much less clear if $T_0$ is a stage tree of depth 1.
When trying to decide some query on some path $\nu_0,\nu_1$ through $T_0$, we need to extend
both the root, and the conditions below each part~$\mu$ refining $\nu_0$.
The overall number of undecided paths may increase, and therefore a simple cardinality argument
is not enough to deduce termination.
Note that this algorithm has some common flavor with the \emph{hydra game}
introduced by Kirby and Paris~\cite{Kirby1982Accessible} and whose termination
is not provable in Peano arithmetic. Thankfully, our problem is much simpler
and its termination can be proven by elementary means.

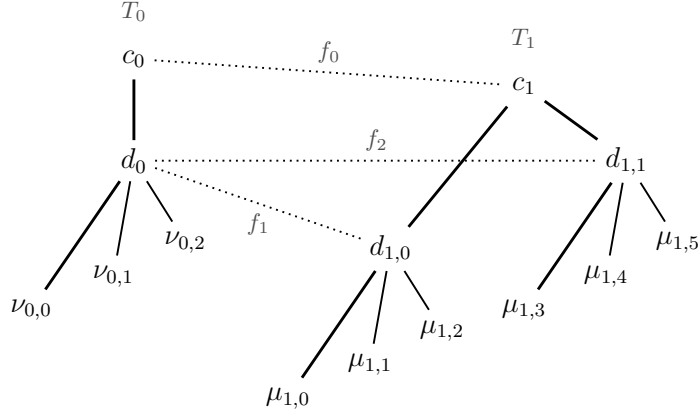
\begin{figure}[htbp]
\begin{center}
\scalebox{0.9}{
\begin{tikzpicture}[x=1.5cm, y=1cm, thick,
		node/.style={inner sep=3pt, minimum size=1.5em, font=\large{#1}},
		stage/.style={node,opacity=0.7},
		label/.style={node,rotate=-3.5,opacity=0.7},
		edgelabel/.style={opacity=0.7}
]

	\node[node] (c0) at (4, 9.8) {$c_0$};
	\node[node] (d0) at (4, 8.3) {$d_0$};
	\node[node] (e00) at (3, 6.1) {$\nu_{0,0}$};
	\node[node] (e01) at (3.8, 6.6) {$\nu_{0,1}$};
	\node[node] (e02) at (4.5, 7.1) {$\nu_{0,2}$};

	\node[node] (c1) at (7.8, 9.4) {$c_1$};
	\node[node] (d10) at (6.5, 7) {$d_{1,0}$};
	\node[node] (d11) at (8.8, 8.3) {$d_{1,1}$};

	\node[node] (e10) at (5.5, 4.8) {$\mu_{1,0}$};
	\node[node] (e11) at (6.3, 5.3) {$\mu_{1,1}$};
	\node[node] (e12) at (7, 5.8) {$\mu_{1,2}$};

	\node[node] (e13) at (7.8, 6.1) {$\mu_{1,3}$};
	\node[node] (e14) at (8.6, 6.6) {$\mu_{1,4}$};
	\node[node] (e15) at (9.3, 7.1) {$\mu_{1,5}$};

	\node[stage] (l3) at (4,10.5) {\small $T_0$};
	\node[stage] (l4) at (7.8,10.1) {\small $T_1$};
	
	\draw[dotted] (c0) -- (c1) node [edgelabel, midway, above] {\small $f_0$};

	\draw[very thick] (c0) -- (d0);
	\draw[very thick] (d0) -- (e00);
	\draw (d0) -- (e01);
	\draw (d0) -- (e02);

	\draw[very thick] (c1) -- (d10);
	\draw[very thick] (c1) -- (d11);

	\draw[very thick] (d10) -- (e10);
	\draw (d10) -- (e11);
	\draw (d10) -- (e12);

	\draw[very thick] (d11) -- (e13);
	\draw (d11) -- (e14);
	\draw (d11) -- (e15);

	\draw[dotted] (d0) -- (d10) node [edgelabel, midway, below] {\small $f_1$};
	\draw[dotted] (d0) -- (d11) node [edgelabel, midway, above] {\small $f_2$};

\end{tikzpicture}
}
\end{center}
\caption{In this example, we start with a stage tree $T_0$ of depth 1 and want to decide some query $\varphi(D,U)$
for each of its three paths. We choose one path $\rho = c_0 \-- d_0 \-- \nu_{0,0}$,
call $\boxop(\nu_{0,0}, \varphi)$ to obtain a query $\psi(D)$, then call $\boxop(\lambda, \psi)$,
where $\lambda$ is the unique part of~$c_0$.
We obtain a query $\phi(D)$, $\emptyset'$-compute some answer $a$ to $\phi(\emptyset)$
and call $\unboxop(\lambda, a)$ to obtain some extension $c_1$ of $c_0$ and some answering function $b : \parop(c_1) \to \ansop[D,G]$. This extension forks the part $\lambda$ into two parts. Below each part $\lambda_i$ in $c_1$, we call $\unboxop(\lambda_i, b(\lambda_i))$
to obtain an extension of $d_{1,i}$ forcing $\varphi(D,G)$ below the parts refining $\nu_{0,0}$.}
\label{fig:refinement-query-progress}
\end{figure}

In Figure~\ref{fig:refinement-query-progress}, we give an example of one step in the decision process,
starting with a stage tree $T_0$ of depth $1$ with three undecided paths, and ending up with
some stage tree $T_1$ having four undecided paths ($c_1 \-- d_{1,0} \-- \mu_{1,1}$, $c_1 \-- d_{1,0} \-- \mu_{1,2}$,
$c_1 \-- d_{1,0} \-- \mu_{1,1}$ and $c_1 \-- d_{1,0} \-- \mu_{1,2}$).
Thankfully, the $\unboxop$ operator forks only the part on which it answers the query. Therefore,
at the next step, we will be able to consider only one of the parts of $c_1$ at a time.
The induced subtree has two undecided paths, so there is also some progress.

We now define some relation $\sqsubset$ between two stage trees $T_0$ and $T_1$ of depth $n$.
It describes the relation between the stage tree $T_0$ and the extension $T_1$ obtained after
applying one step of the query algorithm. More precisely, $T_2 \sqsubset T_0$
if $T_2$ is the subtree of $T_1$ on which we have removed every decided paths.

\begin{definition}
Given two stages trees $T_1 \leq T_0$ of depth~$n$,
we define the relations $T_1 \sqsubseteq T_0$ and $T_1 \sqsubset T_0$ by mutual induction as follows.
If $T_0$ and $T_1$ are conditions and $T_1 \leq_f T_0$, then 
$T_1 \sqsubseteq T_0$ if $f$ does not fork any part of~$T_0$. 
If moreover $f$ is not surjective, then $T_1 \sqsubset T_0$.
If $T_1 = \tuple{c_1, h_1}$, $T_0 = \tuple{c_0, h_0}$ and $c_1 \leq_f c_0$,
then $T_1 \sqsubseteq T_0$ if for every part $\nu$ of $c_1$,
if $f$ forks part~$f(\nu)$ of $c_0$ then $h_1(\nu) \sqsubset h_0(f(\nu))$,
otherwise $h_1(\nu) \sqsubseteq h_0(f(\nu))$. If moreover there is some part~$\mu$ of $c_0$
such that $h_1(\nu) \sqsubset h_0(\mu)$ for every part~$\nu$ of $c_1$ refining $\mu$,
then $T_1 \sqsubset T_0$.
\end{definition}

One easily proves by mutual induction over the depth of the trees the following facts:
\begin{itemize}
	\item[(i)] Both~$\sqsubset$ and~$\sqsubseteq$ are transitive
	\item[(ii)] If~$T_1 \sqsubset T_0$ then~$T_1 \sqsubseteq T_0$
	\item[(iii)] If~$T_2 \sqsubseteq T_1$ and~$T_1 \sqsubset T_0$ then~$T_2 \sqsubset T_0$
	\item[(iv)] If~$T_2 \sqsubset T_1$ and~$T_1 \sqsubseteq T_0$ then~$T_2 \sqsubset T_0$
\end{itemize}
Assuming that $\sqsubset$ truly represents the relation between a stage tree
and its extension after one step of query, the following lemma
can be understood as stating that the naive algorithm used in the proof of the query lemma terminates.

\begin{lemma}\label{lem:sq-well-founded}
The relation $T_1 \sqsubset T_0$ is well-founded.
\end{lemma}
\begin{proof}
By induction over the depth of the stage trees.
Suppose that $T_0 \sqsupset T_1 \sqsupset \dots$ is an infinite
decreasing sequence of stage trees of depth 0.
In particular, the $T$'s are conditions and $T_0 \geq_{f_0} T_1 \geq_{f_1} \dots$
for some functions $f_i$ which are injective, but not surjective. Therefore
the number of parts strictly decreases in $\omega$, contradiction.

Suppose now that $T_0 \sqsupset T_1 \sqsupset \dots$ is an infinite
decreasing sequence of stage trees of depth $n+1$,
where $T_i = \tuple{c_i, h_i}$ and $c_i \geq_{f_i} c_{i+1}$.
Let $S$ be the set of parts~$\nu$ in some~$c_i$
which will fork at a later $c_j$. This $S$ induces a finitely branching tree.
If $S$ is finite, then there is some~$j$ such that no part of~$c_k$ will ever fork for every~$k \geq j$.
By the infinite pigeonhole principle, there we can construct an infinite, decreasing sequence
of trees of depth $n$, contradicting our induction hypothesis. So suppose that $S$ is infinite.
By König's lemma, there is an infinite sequence of parts, the later refining the former,
such that they fork. Each time a conditions fork, the subtree is strictly decreasing, so we can
define an infinite decreasing sequence of stage trees of depth $n$, again contradicting our induction hypothesis.
\end{proof}

Given some stage tree $T_1$ of depth $i < n$, a \emph{completion of $T_1$ to $n$}
is a stage tree $T_2$ of depth $n$ such that $T_1 \uh i = T_0$.
If $T_1 \leq T_0 \uh i$ for some stage tree $T_0$ of depth $n$,
$T_0$ induces a completion $T_2$ of $T_1$ to $n$ by setting $T_2^{[\xi]} = T_0^{[\rho]}$
for every path $\xi$ through $T_1$ refining some path $\rho$ through $T_0 \uh i$.
One easily checks that $T_2 \leq T_0$.
Such a stage tree is called the \emph{trivial completion of $T_1$ by~$T_0$}.
The following technical lemma will be useful for applying 
the induction hypothesis in Lemma~\ref{lem:sq-one-step-forcing}.

\begin{lemma}\label{lem:sq-concatenation}
Let $T_0, T_1$ be two stage trees of depth $n+1$ and $T_2$ be a stage tree of depth $n$ 
and $S_0$ be a set of paths through $T_0 \uh n$ such that
\begin{itemize}
	\item[(i)] $T_2 \sqsubseteq T_0 \uh n$, $P(T_2) \subseteq P(T_1 \uh n)$ and $T_1 \leq T_0$
	\item[(ii)] $S_0$ is the set of paths through $T_0 \uh n$ refined by some path through $T_2$.
	\item[(iii)] For every path $\xi \in P(T_2)$, $T_1^{[\xi]} \sqsubseteq T_0^{[\rho]}$ where $\xi$ refines the path $\rho \in S_0$	
	\item[(iv)] For every path $\xi \in P(T_1 \uh n) \setminus P(T_2)$, $\xi$ refines some path $\rho \in P(T_0 \uh n) \setminus S_0$
	and $T_1^{[\xi]} \sqsubset T_0^{[\rho]}$
\end{itemize}
Then $T_1 \sqsubseteq T_0$. Moreover, if $T_2 \sqsubset T_0 \uh n$ then $T_1 \sqsubset T_0$.
\end{lemma}
\begin{proof}
By induction over $n$. In the base case, $T_0 \uh n$, $T_1 \uh n$ and $T_2$
are conditions $c_0$, $c_1$ and $c_2$ such that $c_2 \leq_f c_0$ and $c_1 \leq_g c_0$ for some refinement functions $f$ and $g$.
We easily have $T_1 \sqsubseteq T_0$ since $T_1^{[\mu]} \sqsubseteq T_0^{[g(\nu)]}$ for every
part $\mu$ of $c_2$ (and therefore of~$c_1$),
and since $T_1^{[\mu]} \sqsubset T_0^{[g(\nu)]}$ whenever $\mu$ is a part of~$c_1$ which is not a part of~$c_2$. 
By $c_2 \sqsubseteq c_0$, the only places where a fork can happen is when $\mu$ is not in $c_2$.

We now want to prove that $T_1 \sqsubset T_0$ whenever~$c_2 \sqsubset c_0$.
Since $c_2 \sqsubset c_0$, $f$ is injective, but not surjective. We need to prove that there is some part~$\nu$ of $T_1$ such that $T_1^{[\nu]} \sqsubset T_0^{[g(\nu)]}$. We have two cases. In the first case, $f$ and $g$ have the same domain. In this case $f = g$ and since $f$ is not surjective, 
there is some part of~$c_0$ witnessing the strictness of $T_1 \sqsubset T_0$.
In the second case, there is some part~$\nu$ in $c_1$ but not $c_2$. By (iv), $g(\nu) \not \in S_1$. The part~$g(\nu)$ of $c_0$ witnesses the stricteness of $T_1 \sqsubset T_0$.

In the induction case, $T_0 \uh n = \tuple{c_0, h_0}$, $T_1 \uh n = \tuple{c_1, h_1}$ and $T_2 = \tuple{c_2, h_2}$
such that $c_2 \leq_f c_0$ and $c_1 \leq_g c_0$ for some refinement functions $f$ and $g$.
For every part~$\nu$ in $c_1$, we have two cases:
In the first case, $\nu$ is not in $c_2$. By (iv), any path $\xi$ through $h_1(\nu)$
refines some path $\rho$ in $h_0(g(\nu))$ such that $h_1(\nu)^{[\xi]} \sqsubset h_0(g(\nu))^{[\rho]}$.
By the induction hypothesis applied to $h_0(\nu)$, $h_1(\nu)$ and the empty tree, $h_1(\nu) \sqsubset h_0(g(\nu))$.
In the second case, $\nu$ is also in $c_2$. By the induction hypothesis applied to $h_0(\nu)$, $h_1(\nu)$
and $h_2(\nu)$, $h_1(\nu) \sqsubseteq h_0(g(\nu))$.
We again easily have $T_1 \sqsubseteq T_0$ since $h_1(\nu) \sqsubseteq h_0(g(\nu))$ for every part~$\nu$ in $c_1$
and since whenever $g$ forks some part~$\mu$ of~$c_0$, either the parts $\nu$ of $c_1$ refining $\mu$
are all in $c_2$ in which case $h_2(\nu) \sqsubset h_0(\mu) \uh n$ by the definition of the partial order and then we have $h_1(\nu) \sqsubset h_0(\mu)$, or none of the parts $\nu$ of $c_1$ refining $\mu$ are in $c_2$,
in which case we have $h_1(\nu) \sqsubset h_0(\mu)$.
By the same case analysis as in the base case, we deduce that $T_1 \sqsubset T_0$
if moreover $T_2 \sqsubset T_0 \uh n$.
\end{proof}

\begin{definition}[Stage tree substration]
Given a stage tree $T$ of depth $n$ and a set $S$ of paths through $T$,
we define $T - S$ inductively as follows:
If $T$ is a stage tree of depth 0, then $S$ is a set of parts of $T$
and $T - S$ is the condition whose parts are $\parop(T) \setminus S$.
If $T = \tuple{c, h}$ is a stage tree of depth $n+1$, then $S$ is a set of paths of the form
$\nu\rho$ where $\nu$ is a part of $c$ and $\rho$ is a path through $h(\nu)$.
For each part~$\nu$, let $S_\nu = \{ \rho : \nu\rho \in S\}$.
The stage tree $T - S$ is defined by $\tuple{c, h_1}$ where $h_1(\nu) = h(\nu) - S_\nu$
for each part $\nu$ of $c$.
\end{definition}

Intuitively, $T - S$ is the maximal subtree of $T$ such that
$P(T - S) = P(T) \setminus S$. Beware, even if we may remove every part of a condition,
we do not remove the condition from the tree.
The following lemma uses the well-founded partial order
defined previously to show that we can make some progress
in deciding the queries. In what follows, the set~$T_0$ can be thought of as
the stage tree we obtain after having applied finitely many steps
of query and $S_0$ are the paths through the tree~$T_0$ for which we have already
decided the query~$\varphi(D, G)$. The lemma describes the relation
between the state~$(T_1, S_1)$ obtained from~$(T_0, S_0)$ after having applied one more step.

\begin{lemma}\label{lem:sq-one-step-forcing}
Let $T_0$ be a stage tree of depth $n$, $S_0$ be a set of paths through $T_0$ and let~$\varphi(D,G)$ be a query.
For every path $\rho \not \in S_0$ through $T_0$, there exists a stage tree $T_1 \leq T_0$ of depth $n$
and a set $S_1$ of paths through $T_1$ such that
\begin{itemize}
	\item[(i)] $T_1 \Vdash_\xi \varphi(D,G)$ or $T_1 \Vdash_\xi \neg \varphi(D, G)$ for every
	path $\xi$ through $T_1$ refining $\rho$.
	\item[(ii)] $T_1 - S_1 \sqsubset T_0 - S_0$
	\item[(iii)] Every path in $S_1$ refines either a path in $S_0$ or $\rho$.
\end{itemize}
Moreover, $T_1$ and the function of answers $a : P(T_1) \to \ansop[D,G]$
can be $\emptyset'$-effectively computed uniformly in $T_1$ and $\varphi(D,G)$.
\end{lemma}
\begin{proof}
By induction over $n$.
If $T_0$ is a stage tree of depth 0, then it is a condition $c_0$
and the paths through $T_0$ are the parts of~$c_0$. Let~$\nu$ be such a part.
Let~$\psi(D)$ be the query $\boxop(\nu, \varphi)$.
We can $\emptyset'$-compute an answer $a_0$ to $\psi(\emptyset)$.
Let~$\tuple{c_1,f,a} = \unboxop(\nu, a_0)$ be such that $c_1 \leq_f c_0$,
$f$ forks only part~$\nu$ of $c_0$ and for every part~$\mu$ of $c_1$ such that $f(\mu) = \nu$,
$c_1 \Vdash_\mu \varphi(D, G)$ or $c_1 \Vdash_\mu \neg \varphi(D, G)$ and $a(\mu)$
answers $\varphi(D,G)$ accordingly. Take $S_1 = \{ \mu : f(\mu) = \nu \vee f(\mu) \in S_0\}$.
The property (i) holds by definition of $c_1$ and (iii) holds by definition of $S_1$.
Since the only forked part is $\nu$ and no part of~$c_1 - S_1$ refines $\nu$,
$c_1 - S_1 \sqsubset c_0 - S_0$, so the property (ii) also holds. This completes the base case.

Suppose now that $T_0$ is a stage tree of depth~$n+1$.
The paths through $T_0$ are of the form $\rho\nu$
where $\rho$ is a path through $T_0 \uh n$ and $\nu$ is a part of the root of $T_0^{[\rho]}$.
Fix any such path. Let~$\psi(D)$ be the query $\boxop(\nu, \varphi)$
and let $\phi(D, G)$ be the formula $\psi(D \oplus G)$.
By induction hypothesis on $T_0 \uh n$, there is a stage tree $T_2 \leq T_0 \uh n$
and a set $S_2$ such that
\begin{itemize}
	\item[(i)] $T_2 \Vdash_\xi \phi$ or $T_2 \Vdash_\xi \neg \phi$ for every
path $\xi$ through $T_2$ refining $\rho$
	\item[(ii)] $T_2 - S_2 \sqsubset T_0 \uh n - S_0 \uh n$
	\item[(iii)] Every path in $S_2$ refines either a path in $S_0 \uh n$
	or $\rho$.
\end{itemize}
Moreover, still by induction hypothesis, we have a function $a : P(T_2) \to \ansop[D, G]$
answering the queries. We define a completion of $T_2$ into a stage tree $T_1$
of depth $n+1$ as follows: For each path $\xi$ through $T_2$ refining $\rho$,
let $T_1^{[\xi]}$ be the condition~$c_\xi$ such that $\tuple{c_\xi, f_\xi, a_\xi} = \unboxop(\nu, a(\xi))$.
For each path $\xi$ through $T_2$ which refines some path $\tau$ through $T_0$ different from $\rho$, let $T_1^{[\xi]} = T_0^{[\tau]}$.
By construction, $T_1 \leq T_0$ since $c_\xi$ $f_\xi$-refines $T_0^{[\rho]}$ whenever $\xi$ refines $\rho$
and since any condition refines itself. Let $S_1$ be the collection of paths $\xi\mu$
through $T_1$ such that $\xi \in S_2$ and either $\xi$ refines $\rho$ and $f_\xi(\mu) = \nu$, or $\xi\mu$
refines a path in $S_0$. Since $(T_0 \uh n - S_0 \uh n) \sqsubseteq (T_0 - S_0) \uh n$,
we have $T_2 - S_2 \sqsubset (T_0 - S_0) \uh n$. We can therefore apply Lemma~\ref{lem:sq-concatenation} to $T_0 - S_0$, $T_1 - S_1$, and $T_2 - S_2$, to obtain $T_1 - S_1 \sqsubset T_0 - S_0$.
Define the answer function $b : P(T_1) \to \ansop[D, G]$ by $b(\xi\mu) = a_\xi(\mu)$
for each path $\xi$ through $T_2$ refining $\rho$. This function $b$ is found $\emptyset'$-effectively
since the $\unboxop$ operator is computable.
\end{proof}

The following lemma simply iterates Lemma~\ref{lem:sq-one-step-forcing}
and uses the well-foundedness of the relation $\sqsubset$ to deduce
that we can find some extension on which the queries are decided for every path.

\begin{lemma}\label{lem:sq-one-level-forcing}
Let $T_0$ be a stage tree of depth $n$ and let~$q : P(T_0) \to \queryop[D,G]$ be a function.
There is a stage tree $T_1 \leq T_0$ of depth $n$ such that
$T_1 \Vdash_\xi q(\rho)$ or $T_1 \Vdash_\xi \neg q(\rho)$ for every
path $\xi$ through $T_1$ refining some path $\rho$ through $T_0$.
Moreover, $T_1$ and the function of answers $a : P(T_1) \to \ansop[D,G]$
can be $\emptyset'$-effectively computed uniformly in $T_1$ and $q$.
\end{lemma}
\begin{proof}
Using Lemma~\ref{lem:sq-one-step-forcing}, 
define a sequence of tuples $\tuple{T_0, S_0, \rho_0, \tau_0}, \tuple{T_1, S_1, \rho_1, \tau_1}, \dots$
starting with $T_0$, $S_0 = \emptyset$, $\rho_0 = \tau_0 \in P(T_0)$
and such that for each~$i$
\begin{itemize}
	\item[(i)] $T_{i+1} \leq T_i$ is a stage tree of depth $n$, $S_i$ is a set of paths through $T_i$
	\item[(ii)] $\rho_i$ is a path through $T_i - S_i$ refining the path $\tau_i$ through $T_0$.
	\item[(iii)] $T_{i+1} \Vdash_\xi q(\tau_i)$ or $T_{i+1} \Vdash_\xi \neg q(\tau_i)$ for every
	path $\xi$ through $T_{i+1}$ refining $\rho_i$
	\item[(iv)] $T_{i+1} - S_{i+1} \sqsubset T_i - S_i$
	\item[(v)] Every path in $S_{i+1}$ refines either a path in $S_i$ or $\rho_i$.
\end{itemize}
By Lemma~\ref{lem:sq-well-founded}, the relation $\sqsubset$ is well-founded, so the sequence has to be finite by (iv). 
Let $k$ be the maximal index of the sequence.
By maximality of $k$ and by Lemma~\ref{lem:sq-one-step-forcing}, $P(T_k) - S_k = \emptyset$.
Therefore, $P(T_k) = S_k$. Since $S_0 = \emptyset$ and by (v), we can prove by induction over $k$ that
for every path $\xi$ through $T_k$, there is some stage $i < k$ such that $\xi$ refines $\rho_i$.
Thus, by (iii) and by stability of the forcing relation under refinement, $T_k \Vdash_\xi q(\tau_i)$ or $T_k \Vdash_\xi \neg q(\tau_i)$.
Therefore $T_k$ satisfies the statement of the lemma.
The uniformity is inherited from the uniformity of Lemma~\ref{lem:sq-one-step-forcing}.
\end{proof}

Last, we prove the query lemma by iterating the previous lemma at every depth of the stage
tree, to decide the queries on the partial paths.

\begin{proof}[Proof of the query lemma]
Let $T_0$ be a stage tree of depth $n$ and $q : PP(T_0) \to \queryop[U,G]$ be a function.
Using Lemma~\ref{lem:sq-one-level-forcing},
define a decreasing sequence of stage trees $T_0 \geq \dots \geq T_n$  of depth $n$
such that for each~$i < n$, 
\begin{itemize}
	\item[(i)] $T_{i+1}$ is the trivial completion of $T_{i+1} \uh i+1$ by $T_i$.
	\item[(ii)] $T_{i+1} \Vdash_\xi q(\tau)$ or $T_{i+1} \Vdash_\xi \neg q(\tau)$ for every
	path $\xi$ through $T_{i+1} \uh i+1$ refining some path $\tau$ through $T_0 \uh i+1$.
\end{itemize}
To do this, at stage $i < n$, apply Lemma~\ref{lem:sq-one-level-forcing} to $T_i$ with the query function $r : PP(T_i) \to \queryop[U,G]$
defined by $r(\rho) = q(\tau)$ for each path $\rho$ through $T_i \uh i+1$ refining some path $\tau$ through $T_0 \uh i+1$.
Since the forcing relation is stable by refinement, the stage tree $T_n$ satisfies the statement of the query lemma.
The uniformity is again inherited from the uniformity of Lemma~\ref{lem:sq-one-level-forcing}.
\end{proof}

This completes the presentation of the framework. We will now define a module for the Erd\H{o}s-Moser
theorem. In section~\ref{sect:separating-amt-combined}, we will see how to compose modules to obtain stronger separations.

\section{The weakness of~$\emo$ over $\omega$-models}

Now we have settled the domination framework, it suffices to implement
the abstract module to obtain $\omega$-structures which do not satisfy $\amt$.
We have illustrated the notion of module by implementing one for $\coh$.
An immediate consequence is the existence of an $\omega$-model of $\coh$
which is not a model of~$\amt$. 
In this section, we shall extend this separation to the Erd\H{o}s-Moser theorem.
As noted before, every $\omega$-model of~$\emo$ which is not a model of~$\amt$
is also a model of~$\coh$.
This section is devoted to the proof of the following theorem.

\begin{theorem}\label{thm:emo-not-amt}
There exists an $\omega$-model of~$\emo$ which is not a model of~$\amt$.
\end{theorem}

At first sight, the forcing notion introduced in section~\ref{sect:emo-computable-reducibility} seems
to have a direct mapping to the abstract notion of forcing defined in the domination framework.
However, unlike cohesiveness where the module implementation was immediate,
the Erd\H{o}s-Moser theorem raises new difficulties:
\begin{itemize}
	\item The Erd\H{o}s-Moser theorem is not known to admit a universal instance.
	We will therefore need to integrate the information about the instance in the notion of condition.
	Moreover, the $\iniop$ operator will have to choose accordingly some new instance of~$\emo$
	at every iteration level. We need to make~$\iniop$ computable, but the collection of every infinite computable tournament functionals
	is not even computably enumerable.

	\item The notion of EM condition introduced in section~\ref{sect:emo-computable-reducibility} contains a $\Pi^{0,R}_1$ property
	ensuring extensibility. Since the tournament $R$ depends on the previous iteration which is being constructed,
	we have only access to a finite part of~$R$. We need therefore to ensure that whatever the extension of the finite
	tournament is, the condition will be extendible.
\end{itemize}

We shall address the above-mentioned problems one at a time
in subsections~\ref{subsect:enum-emo-infinite} and~\ref{subsect:emo-new-condition}.

\subsection{Enumerating the infinite tournaments}\label{subsect:enum-emo-infinite}

In section~\ref{sect:emo-computable-reducibility}, we were also confronted to the problem
of enumerating all infinite tournaments
and solved it by relativizing the construction to a low subuniform degree
in order to obtain a low sequence of infinite tournaments containing at least
every infinite computable tournament. We cannot apply the same
trick to handle the construction of an $\omega$-model of~$\emo$
as solutions to some computable tournaments may bound new tournaments
and so on. However, as we shall see, we can restrict ourselves
to primitive recursive tournaments to generate an $\omega$-model of~$\emo$.

Given a sequence of sets $X_0, X_1, \dots$, define $\Mcal_{\vec{X}}$ to be the $\omega$-structure
whose second-order part is the Turing ideal generated by $\vec{X}$, that is,
$$
\{ Y \in 2^\omega : (\exists i)[ Y \leq_T X_0 \oplus \dots \oplus X_i ]\}
$$

\begin{lemma}\label{lem:em-uniform-model}
There exists a uniformly computable sequence of infinite, primitive recursive tournament functionals
$T_0, T_1, \dots$ such that for every sequence of sets $X_0, X_1, \dots$ such that
$X_i$ is an infinite transitive subtournament of $T_i^{X_0 \oplus \dots \oplus X_{i-1}}$
for each $i \in \omega$,
$$
\Mcal_{\vec{X}} \not \models \amt \imp \Mcal_{\vec{X}} \models \emo \wedge \coh
$$
\end{lemma}
\begin{proof}
As $\rca \vdash \semo \wedge \coh \imp \emo$,
it suffices to prove that for every set $X$, 
\begin{itemize}
	\item[(i)] for every stable, infinite, $X$-computable tournament~$R$, there exists an infinite $X$-p.r.\ tournament~$T$
	such that every infinite $T$-transitive subtournament $X$-computes an infinite $R$-transitive subtournament.
	\item[(ii)] for every $X$-computable complete atomic theory~$T$ and every
uniformly $X$-computable sequence of sets $\vec{R}$, there exists an infinite $X$-p.r.\ tournament
such that every infinite transitive subtournament $X$-computes either an $\vec{R}$-cohesive set
or an atomic model of~$T$.
\end{itemize}

(i)
Fix a set $X$ and a stable, infinite, $X$-computable tournament~$R$.
Let $\tilde{f} : \omega \to 2$ be the $X'$-computable function defined by $\tilde{f}(x) = 0$
if $(\forall^\infty s) R(s, x)$ and $\tilde{f}(x) = 1$ if $(\forall^\infty s) R(x, s)$.
By Schoenfield's limit lemma~\cite{Shoenfield1959degrees}, there exists an $X$-p.r.
function $g : \omega^2 \to 2$ such that $\lim_s g(x, s) = \tilde{f}(x)$ for every $x \in \omega$.
Considering the $X$-p.r. tournament $T$ such that $T(x,y)$ holds iff $x < y$ and $g(x, y) = 1$
or $x > y$ and $g(x, y) = 0$, every infinite $T$-transitive subtournament $X$-computes an infinite
$R$-transitive subtournament.

(ii) Jockusch and Stephan proved in~\cite{Jockusch1993cohesive} that for every set~$X$,
and every uniformly $X$-computable sequence of sets~$\vec{R}$, 
every p-cohesive set relative to~$X$ computes an $\vec{R}$-cohesive set.
The author proved in~\cite{Patey2015Somewhere} that for every $X$-computable complete atomic theory~$T$,
there exists an $X'$-computable coloring $f : \omega \to \omega$ such that every infinite
set $Y$ \emph{thin} for $f$ (i.e.\ such that $f(Y) \neq \omega$) $X$-computes an atomic model of~$T$.
He also proved that for every such $X'$-computable coloring $f : \omega \to \omega$,
there exists an infinite, $X$-p.r. tournament~$R$ such that every infinite transitive subtournament
is either p-cohesive, or $X$-computes an infinite set thin for~$f$.
\end{proof}

We can therefore fix this computable enumeration $T_0, T_1, \dots$
of tournament functionals, and make $\iniop(n)$ return an empty condition
paired with $T_n$. Thus, taking at each iteration an infinite set satisfying one of the parts,
we obtain an $\omega$-model of~$\emo$.

\subsection{The new Erd\H{o}s-Moser conditions}\label{subsect:emo-new-condition}

Fix some primitive recursive tournament functional $R$.
According to the analysis of the Erd\H{o}s-Moser presented in section~\ref{sect:emo-computable-reducibility},
we would like to define the forcing conditions to be tuples~$(\vec{F}, \Ccal)$ where
\begin{itemize}
	\item[(a)] $\Ccal$ is a non-empty $\Pi^{0,D}_1$ $k$-cover class of $[t, +\infty)$ 
	for some $k, t \in \omega$
	\item[(b)] $F_\nu \cup \{x\}$ is $R^D$-transitive for every $Z_0 \oplus \dots \oplus Z_{k-1} \in \Ccal$,
	every $x \in Z_\nu$ and each $\nu < k$
	\item[(c)] $Z_\nu$ is included in a minimal $R^D$-interval of $F_\nu$
	for every $Z_0 \oplus \dots \oplus Z_{k-1} \in \Ccal$ and each~$\nu < k$.
\end{itemize}

However, at a finite stage, we have only access to a finite part of~$D$,
and therefore we cannot express the properties (a-c). Indeed,
we may have made some choices about the $F$'s such that $F_\nu \cup \{x\}$
is not $R^D$-transitive for every part~$\nu$, every $D$ satisfying the previous iterations and cofinitely many $x \in \omega$.
We need therefore to choose the $F$'s carefully enough so that whatever the extension of the finite tournament
to which we have access, we will be able to extend at least one of the $F$'s.

The initial condition $(\{\emptyset\}, \{\omega\})$ satisfies the properties (a-c) no matter what $D$ is, since $\{\omega\}$
does not depend on $D$.
Let us have a closer look at the question $Q2$ asked in section~\ref{sect:emo-computable-reducibility}.
For the sake of simplification, we will consider that the question is asked below the unique 
part of the initial condition. It therefore becomes:

\smallskip
{\itshape
Q3: Is there a finite set $E \subseteq \omega$ such that for every 2-partition $\tuple{E_0, E_1}$ of~$E$,
there exists an $R^D$-transitive subset $F_1 \subseteq E_i$ for some $i < 2$ such that $\varphi(D, F_1)$ holds?
}
\smallskip

Notice that this is a syntactic question since it depends on the purely formal variable $D$
representing the effective join of the sets constructed in the previous iterations.
Thanks to the usual query process, we are able to transform it into a concrete $\Sigma^0_1$
formula getting rid of the formal parameter~$D$, 
and obtain some answer that the previous layers guarantee to hold for every set $D$
satisfying the previous iterations.

If the answer is negative, then by compactness, for every set $D$ satisfying the previous iterations,
there is a 2-partition $Z_0 \cup Z_1 = \omega$ such that for every $i < 2$ and every $R^D$-transitive
subset $G \subseteq Z_i$, $\varphi(D, G)$ does not hold. For every set $D$,
the $\Pi^{0,D}_1$ class $\Ccal$ of such 2-partitions $Z_0 \oplus Z_1$ is therefore guaranted to be non-empty.
Note again that since $D$ is a syntactic variable, the class $\Ccal$ is also syntactic,
and purely described by finite means.

If the answer is positive, then we are given some finite set $E \subseteq \omega$ witnessing it.
Moreover, we are guaranted that for every set $D$ satisfying the previous iterations,
for every 2-partition $\tuple{E_0, E_1}$ of~$E$, there exists an $R^D$-transitive subset 
$F_1 \subseteq E_i$ for some $i < 2$ such that $\varphi(D, F_1)$ holds.
In we knew the set $D$, we would choose one ``good'' 2-partition $\tuple{E_0, E_1}$
as we do in section~\ref{sect:emo-computable-reducibility}. However, this choice
depends on infinitely many bits of information of $D$. We will need therefore to
try every 2-partition in parallel.

There is one more difficulty.
With this formulation, we are not able to find the desired extension, since $D$ is syntactic,
and therefore we do not know how to identify the color~$i$ and the actual set~$F_1$ given some 2-partition $\tuple{E_0, E_1}$.
Thankfully, we can slightly modify the question to ask to provide the witness $F_1$ for each such a partition
in the answer.

\smallskip
{\itshape
Q4: Is there a finite set $E \subseteq \omega$ and a finite function $g$ such that for every 2-partition $\tuple{E_0, E_1}$ of~$E$,
$g(\tuple{E_0, E_1}) = F_1$ for some $i < 2$ and some $R^D$-transitive subset $F_1 \subseteq E_i$ such that $\varphi(D, F_1)$ holds?
}
\smallskip

The question $Q4$ is equivalent to the question $Q3$, but provides a constructive witness~$g$
in the case of a positive answer as well. We can even formulate the question so that we know the relation $R^D$
over the set~$F_1$. Thus we are able to talk about minimal $R^D$-intervals of $F_1$.

Now, we can extend the initial condition $(\{\emptyset\}, \{\omega\})$
into some condition $(\vec{F}, \Ccal)$ as follows:
For each 2-partition $\tuple{E_0, E_1}$ of $E$, letting $F_1 = g(\tuple{E_0, E_1})$, 
for every minimal $R^D$-interval $I$, we create a part~$\nu = \tuple{E_0, E_1, I}$
and set $F_\nu = F_1$. Take some $t' > max(\vec{F})$
and let $\Ccal$ be the $\Pi^{0,D}_1$ class of covers $\bigoplus_\nu Z_\nu$ of $[t', +\infty)$ such that
for every part $\nu = \tuple{E_0, E_1, I}$
\begin{itemize}
	\item[(b')] $F_\nu \cup \{x\}$ is $R^D$-transitive for every $x \in Z_\nu$
	\item[(c')] $Z_\nu$ is included in the minimal $R^D$-interval $I$
\end{itemize}
Fix some set $D$ satisfying the previous iterations. We claim that $\Ccal$ is non-empty.
Any element $x \in [t', +\infty)$ induces a 2-partition $g(x) = \tuple{E_0, E_1}$ of $E$ by
setting $E_0 = \{ y \in E : R^D(y, x) \}$ and $E_1 = \{y \in E : R^D(x, y)\}$.
On the other hand, for every 2-partition $\tuple{E_0, E_1}$ of $E$,
we can define a partition of $[t', +\infty)$ by setting
$Z_{\tuple{E_0, E_1}} = \{ x \in [t',+\infty) : g(x) = \tuple{E_0, E_1} \}$.
By definition, $E_0 \to_{R^D} Z_{\tuple{E_0, E_1}}  \to_{R^D} E_1$.
Therefore, the cover $\bigoplus_\nu Z_\nu$ of $[t', +\infty)$ defined by
\[
Z_\nu = \cond{
	Z_{\tuple{E_0, E_1}} & \mbox{ if } \nu = \tuple{E_0, E_1, I}, I = (max(F_\nu), +\infty) \mbox{ and } F_\nu \subseteq E_0\\
	Z_{\tuple{E_0, E_1}} & \mbox{ if } \nu = \tuple{E_0, E_1, I}, I = (-\infty, min(F_\nu)) \mbox{ and } F_\nu \subseteq E_1\\
	\emptyset & \mbox{ otherwise}\\
}
\]
is in $\Ccal$ and witnesses the non-emptiness of $\Ccal$.

The problem of having access to only a finite part of the class~$\Ccal$ appears
more critically when considering the question below some part~$\nu$ of an arbitrary condition $c = (\vec{F}, \Ccal)$.
The immediate generalization of the question $Q4$ is the following.

\smallskip
{\itshape
Q5: For every cover $X_0 \oplus \dots \oplus X_{k-1} \in \Ccal$, is there a finite set $E \subseteq X_\nu$ and a finite function~$g$ such that for every 2-partition $\tuple{E_0, E_1}$ of~$E$,
$g(\tuple{E_0, E_1})$ is a finite $R^D$-transitive subset of some $E_j$ such that $\varphi(D, F_\nu \cup g(\tuple{E_0, E_1}))$ holds?
}
\smallskip

As usual, although this question is formulated in a $\Pi^0_2$ manner, it can be turned into a $\Sigma^{0,D}_1$ query 
using a compactness argument.

\smallskip
{\itshape
Q5': Is there some $r \in \omega$, a finite sequence of finite sets $E^0, \dots, E^{r-1}$ 
and a finite sequence of functions $g^0, \dots, g^{r-1}$ such that
\begin{itemize}
	\item[(1)] for every $X_0 \oplus \dots \oplus X_{k-1} \in \Ccal$, there is some $i < r$ such that $E^i \subseteq X_\nu$
	\item[(2)] for every $i < r$ and every 2-partition $\tuple{E_0, E_1}$ of~$E^i$,
	$g^i(\tuple{E_0, E_1})$ is a finite $R^D$-transitive subset of some $E_j$ such that $\varphi(D, F_\nu \cup g^i(\tuple{E_0, E_1}))$ holds?
\end{itemize}
}
\smallskip

In the case of a negative answer, we can apply the standard procedure consisting in refining the $\Pi^{0,D}_1$ class $\Ccal$
into some $\Pi^{0,D}_1$ class~$\Dcal$ forcing $\varphi(D, G)$ not to hold on every part refining the part~$\nu$ in $c$.
The class $\Dcal$ is non-empty since we can construct a member of it from a witness of failure of $Q5$.
The problem appears when the answer is positive. We are given some finite sequence $E^0, \dots, E^{r-1}$ 
and a finite sequence of functions $g^0, \dots, g^{r-1}$ satisfying (i) and (ii).
For every $D$, there is some $X_0 \oplus \dots \oplus X_{k-1} \in \Ccal$ and some $i < r$
such that $E^i \subseteq X_\nu$, but this $i$ may depend on~$D$. We cannot choose some~$E^i$
as we used to do in section~\ref{sect:emo-computable-reducibility}.

Following our moto, if we are not able to make a choice, we will try every possible case in parallel.
The idea is to define a condition $d = (\vec{E}, \Dcal)$ and a refinement function $f$ forking the part~$\nu$ into various parts, each one representing a possible scenario. For every part $\mu$ of $c$ which is different from $\nu$, 
create a part $\mu$ in $d$ and set $E_\mu = F_\mu$. For every $i < r$ and every 2-partition $\tuple{E_0, E_1}$ of $E^i$,
create a part~$\mu = \tuple{i, E_0, E_1}$ in $d$ refining $\nu$ and set $E_\mu = F_\nu \cup g^i(\tuple{E_0, E_1})$.
Accordingly, let $\Dcal$ be the $\Pi^{0,D}_1$ class of covers $\bigoplus_\mu Y_\mu$ of $[t, +\infty)$
such that there is some $i < r$ and some cover $X_0 \oplus \dots \oplus X_{k-1} \in \Ccal$ satisfying first
$E^i \subseteq X_\nu$, second $Y_\mu \subseteq X_{f(\mu)}$ for each part~$\mu$ of $d$
and third $Y_\mu = \emptyset$ if $\mu = \tuple{j, E_0, E_1}$ for some $j \neq i$.

The class $\Dcal$ $f$-refines $\Ccal$, but does not $f$-refine $\Ccal^{[\nu, E^i]}$ for some fixed $i < r$.
Because of this, the condition $d$ does not extends the condition $c$ in the sense of section~\ref{sect:emo-computable-reducibility}.
We shall therefore generalize the operator $\cdot \mapsto \Ccal^{[\nu, \cdot]}$ to define it over tuples of sets.

\smallskip
\emph{Restriction of a cover class}. Given some cover class $(k, Y, \Ccal)$,
some part~$\nu$ of $\Ccal$ and some $r$-tuple $E^0, \dots, E^{r-1}$ of finite sets, we denote by $\Ccal^{[\nu, \vec{E}]}$
the cover class $(k+r-1, Y, \Dcal)$ such that $\Dcal$ is the collection of 
\[
X_0 \oplus \dots \oplus X_{\nu-1} \oplus Z_0 \oplus \dots \oplus Z_{r-1} \oplus X_{\nu+1} \oplus \dots \oplus X_{k-1}
\]
such that $X_0 \oplus \dots \oplus X_{k-1} \in \Dcal$ and there is some $i < r$ such that
$E^i \subseteq X_\nu$, $Z_i = X_\nu$ and $Z_j = \emptyset$ for every $j \neq i$.
In particular, $\Ccal^{[\nu, \vec{E}]}$ refines $\Ccal$ with some refinement function $f$ which forks the part~$\nu$
into~$r$ different parts. Such a function $f$ is called the \emph{refinement function witnessing the restriction}.

We need to define the notion of extension between conditions accordingly.
A condition $d = (\vec{E}, \Dcal)$ \emph{extends} a condition $c = (\vec{F}, \Ccal)$
(written $d \leq c$) if there is a function $f : parts(\Dcal) \to parts(\Ccal)$ such that the following holds:
\begin{itemize}
	\item[(i)] $(E_\nu, dom(\Dcal))$ Mathias extends $(F_{f(\nu)}, dom(\Ccal))$ for each $\nu \in parts(\Dcal)$ 
	\item[(ii)] Every $\bigoplus_\mu Y_\mu \in \Dcal$ $f$-refines some $\bigoplus_\nu X_\nu \in \Ccal$
	such that for each part~$\mu$ of $d$, either $E_\mu \setminus H_{f(\mu)} \subseteq X_{f(\mu)}$,
	or $Y_\mu = \emptyset$.
\end{itemize}

Note that this notion of extension is coarser than the one defined in section~\ref{sect:emo-computable-reducibility}.
Unlike with the previous notion of extension, there may be from now on some part~$\mu$ of $d$ refining the part~$\nu$ of $c$,
such that $(E_\mu, Y_\mu)$ does not Mathias extend $(F_\nu, X_\nu)$ for some $\bigoplus_\mu Y_\mu \in \Dcal$
and every $\bigoplus_\nu X_\nu \in \Ccal$, but in this case, we make $(E_\mu, Y_\mu)$ non-extendible by ensuring that $Y_\mu = \emptyset$.

\subsection{Implementing the Erd\H{o}s-Moser module}

We are now ready to provide a concrete implementation of a module support and a module for~$\emo$.
Define the tuple $\Sb^{\emo} = \tuple{\Pb, \Ub, \parop, \iniop, \satop}$ as follows:
$\Pb$ is the collection of all conditions~$(\vec{F}, \Ccal, R)$ where $R$ is a primitive recursive
tournament functional and
\begin{itemize}
	\item[(a)] $\Ccal$ is a non-empty $\Pi^{0,D}_1$ $k$-cover class of $[t, +\infty)$ 
	for some $k, t \in \omega$
	\item[(b)] $F_\nu \cup \{x\}$ is $R^D$-transitive for every $Z_0 \oplus \dots \oplus Z_{k-1} \in \Ccal$,
	every $x \in Z_\nu$ and each $\nu < k$
	\item[(c)] $Z_\nu$ is included in a minimal $R^D$-interval of $F_\nu$
	for every $Z_0 \oplus \dots \oplus Z_{k-1} \in \Ccal$ and each~$\nu < k$.
\end{itemize}
Once again, $\Ccal$ is actually a $\Pi^{0,D}_1$ formula denoting a non-empty $\Pi^{0,D}_1$ class.
A condition $d = (\vec{E}, \Dcal, T)$ \emph{extends} 
$c = (\vec{F}, \Ccal, R)$
(written $d \leq c$) if $R = T$ and
there exists a function $f : parts(\Dcal) \to parts(\Ccal)$ such that the properties (i) and (ii)
mentioned above hold.

Given some condition $c = (\vec{F}, \Ccal, R)$,
$\parop(c) = \{\tuple{c, \nu} : \nu \in parts(\Ccal)\}$.
Define $\Ub$ as $\bigcup_{c \in \Pb} \parop(c)$, that is, the set of all pairs $\langle (\vec{F}, \Ccal, R), \nu \rangle$
where $\nu \in parts(\Ccal)$. The operator $\iniop(n)$ returns the condition $(\{\emptyset\}, \{\omega\}, R_n)$
where $R_n$ is the $n$th primitive recursive tournament functional.
Last, define $\satop(\tuple{c, \nu})$ to be the collection of all $R^D$-transitive subtournaments
satisfying the Mathias precondition $(F_\nu, X_\nu)$ where $X_\nu$ is \emph{non-empty} for some $\bigoplus_\nu X_\nu \in \Ccal$.
The additional non-emptiness requirement of $X_\nu$ in the definition of the $\satop$ operator
enables us to ``disable'' some part by setting $X_\nu = \emptyset$. Without this requirement,
the property (i) of a module support would not be satisfied. Moreover, since every cover class has an
acceptable part, there is always one part~$\nu$ in $\Ccal$ such that $\satop(\tuple{c,\nu}) \neq \emptyset$.

\begin{lemma}
The tuple $\Sb^{\emo}$ is a module support.
\end{lemma}
\begin{proof}
We must check that if~$d \leq_\Pb c$ for some~$c, d \in \Pb$, then there is a function~$g : \parop(d) \to \parop(c)$
such that $\satop(\nu) \subseteq \satop(g(\nu))$ for each~$\nu \in \parop(d)$.
Let $d = (\vec{E}, \Dcal, R)$ and $c = (\vec{F}, \Ccal, R)$ be such that $d \leq_\Pb c$.
By definition, there is a function $f : parts(\Dcal) \to parts(\Ccal)$ satisfying
the properties (i-ii). Let $g : \parop(d) \to \parop(c)$ be defined by $g(\tuple{d, \nu}) = \tuple{c, f(\nu)}$.
We claim that $g$ is a refinement function witnessing $d \leq_\Pb c$.
Let $G$ be any set in $\satop(\tuple{d, \nu})$. We will prove that $G \in \satop(\tuple{c, f(\nu)})$.
The set $G$ is an $R^D$-transitive subtournament
satisfying the Mathias condition $(E_\nu, X_\nu)$ where $X_\nu \neq \emptyset$ for some $\bigoplus_\nu X_\nu \in \Dcal$.
By (ii), since $X_\nu$ is non-empty, there is some $\bigoplus_\mu Y_\mu \in \Ccal$ 
such that $E_\nu \setminus F_{f(\nu)} \subseteq Y_{f(\nu)}$
and $X_\nu \subseteq Y_{f(\nu)}$. It suffices to show that $(F_\nu, X_\nu)$
Mathias extends $(F_{f(\nu)}, Y_{f(\nu)})$ to deduce that $G$ satisfies the Mathias condition $(F_{f(\nu)}, Y_{f(\nu)})$
and finish the proof. By (i), $F_{f(\nu)} \subseteq E_\nu$.
Since $E_\nu \setminus F_{f(\nu)} \subseteq Y_{f(\nu)}$ and $X_\nu \subseteq Y_{f(\nu)}$, we are done.
\end{proof}

We next define an implementation of the module $\Mb^{\emo} = \langle \Sb^{\emo}, \boxop, \unboxop, \progop \rangle$ as follows.
Given some condition $c = (\vec{F}, \Ccal, R)$, some $\nu \in parts(\Ccal)$ and some
$\Sigma^0_1$ formula $\varphi(D, G)$, $\unboxop(\tuple{c,\nu}, \varphi)$ returns the $\Sigma^0_1$ formula $\psi(D)$ which holds
if there is a finite sequence of finite sets $E^0, \dots, E^{r-1}$ and a finite sequence of functions $g^0, \dots, g^{r-1}$ such that
\begin{itemize}
	\item[(1)] for every $X_0 \oplus \dots \oplus X_{k-1} \in \Ccal$, there is some $i < r$ such that $E^i \subseteq X_\nu$
	\item[(2)] for every $i < r$ and every 2-partition $\tuple{E_0, E_1}$ of~$E^i$,
	$g^i(\tuple{E_0, E_1})$ is a finite $R^D$-transitive subset of some $E_j$ such that $\varphi(D, F_\nu \cup g^i(\tuple{E_0, E_1}))$ holds.
\end{itemize}

If the answer to $\psi(D)$ is $\tuple{\no}$, $\unboxop(\tuple{c,\nu}, \tuple{\no})$ returns
the tuple~$\tuple{d, f, b}$ where $d = (\vec{E}, \Dcal, R)$ is a condition such that $d \leq_f c$
and defined as follows. For every part $\mu \neq \nu$ of $c$, create
a part~$\mu$ in $d$ and set $E_\mu = F_\mu$. Furthermore, fork the part~$\nu$ into two parts $\nu_0$
and $\nu_1$ in $d$ and set $E_{\nu_i} = F_\nu$ for each~$i < 2$. 
Define $\Dcal$ to be the $\Pi^{0,D}_1$ class of all covers $\bigoplus_\mu Y_\mu$
$f$-refining some cover $\bigoplus_\nu X_\nu \in \Ccal$ and such that for every $i < 2$
and every finite $R^D$-transitive set $E \subseteq Y_{\nu_i}$, $\varphi(D, F_\nu \cup E)$ does not hold.
Moreover, $b : \parop(c) \to \ansop[D,G]$ is the constant function $\tuple{\no}$.

Suppose now that the answer to $\psi(D)$ is $a = \tuple{\yes, r, E^0, \dots, E^{r-1}, f^0, \dots, f^{r-1}, a'}$
where $a'$ is a function which on every $i < r$ and every 2-partition $\tuple{E_0, E_1} = i$,
returns an answer to $\varphi(D, F_\nu \cup g^i(\tuple{E_0, E_1}))$.
The function $\unboxop(\tuple{c,\nu}, a)$ returns the tuple~$\tuple{d, f, b}$ where $d$ is a condition such that $d \leq_f c$
and whose definition has been described in subsection~\ref{subsect:emo-new-condition}.
The function $b : \parop(d) \to \ansop[D, G]$ returns on every part~$\mu = \tuple{i, E_0, E_1}$ 
the tuple $\tuple{\yes, a'(i, E_0, E_1)}$.

Last, given some condition~$c = (\vec{F}, \Ccal, R)$ and some $\nu \in parts(\Ccal)$, 
$\progop(\tuple{c,\nu})$ is the query $\varphi(D, G) = (\exists n)[n \in G \wedge n > max(F_\nu)]$.
Note that we cannot force $\neg \varphi(D, G)$ on every part~$\tuple{c,\nu}$,
since every cover class has an acceptable part. Applying the query lemma infinitely many times
on the progress operator ensures that if we take any path through the infinite tree of the acceptable parts,
the resulting $R^D$-transitive subtournament will be infinite.

\begin{lemma}
The tuple $\Mb^{\emo}$ is a module.
\end{lemma}
\begin{proof}
We need to ensure that given some part~$\nu$ of some condition~$c = (\vec{F}, \Ccal, R)$
and some answer~$a$ to a $\Sigma^0_1$ formula $\psi(D) = \boxop(\tuple{c,\nu}, \varphi)$ where $\varphi(D, G)$ is a $\Sigma^0_1$ formula, 
$\unboxop(\tuple{c,\nu}, a)$ outputs a tuple $\tuple{d, f, b}$ where $d = (\vec{E}, \Dcal, R)$ is a condition such that 
$d \leq_f c$ where $f$ forks only part $\nu$ of~$c$,
and for every part~$\mu$ of~$d$ such that~$f(\tuple{d,\mu}) = \tuple{c,\nu}$, 
and every set~$G \in \satop(\tuple{d,\mu})$, $b(\tuple{d,\mu})$ is an answer to~$\varphi(D, G)$.

Suppose that $a = \tuple{\no}$.
By definition of $\satop(\tuple{d,\mu})$ and by construction of $d$, $G$ is $R^D$-transitive and satisfies the Mathias condition
$(E_{\nu_i}, Y_{\nu_i})$ for some $i < 2$ and some cover $\bigoplus_\mu Y_\mu \in \Dcal$.
In particular, $E_{\nu_i} = F_\nu$ and $Y_{\nu_i}$ is such that for every finite $R^D$-transitive set $E \subseteq Y_{\nu_i}$,
$\varphi(D, F_\nu \cup E)$ does not hold. In particular, taking $E = G \setminus F_\nu$, $\varphi(D, G)$ does not hold.

Suppose now that $a = \tuple{\yes, r, E^0, \dots, E^{r-1}, f^0, \dots, f^{r-1}, a'}$
where $a'$ is a function which on every $i < r$ and every 2-partition $\tuple{E_0, E_1} = i$,
returns an answer to $\varphi(D, F_\nu \cup g^i(\tuple{E_0, E_1}))$.
By definition of $\satop(\tuple{d,\mu})$ and by construction of $d$, $G$ is $R^D$-transitive and satisfies the Mathias condition
$(E_\mu, Y_\mu)$ for some cover $\bigoplus_\mu Y_\mu \in \Dcal$, where $\mu = \tuple{i, E_0, E_1}$.
By construction of $d$, $E_\mu = F_\nu \cup g^i(\tuple{E_0,E_1})$, and by definition of $g^i$,
$\varphi(D, F_\nu \cup g^i(\tuple{E_0,E_1}))$.  Since $G$ satisfies $(E_\mu, Y_\mu)$,
$F_\nu \cup g^i(\tuple{E_0,E_1}) \subseteq G$ and $G \setminus E_\mu \subseteq Y_\mu$. 
Therefore $\varphi(D, G)$ holds.
\end{proof}

\subsection{The separation}

We have defined a module $\Mb^{\emo}$ for the Erd\H{o}s-Moser theorem.
In this subsection, we explain how we create an $\omega$-model of $\emo$ which is not a model of $\amt$
from the infinite sequence of stage trees constructed in subsection~\ref{subsect:framework-construction}.
Given the uniform enumeration $R_0, R_1, \dots$ of all primitive recursive tournament functionals,
we shall define an infinite sequence of sets $X_0, X_1, \dots$ together with a $\Delta^0_2$ function $f$ such that
for every $s$,
\begin{itemize}
	\item[1.] $X_{s+1}$ is an infinite, transitive subtournament of $R^{X_0 \oplus \dots \oplus X_s}$
	\item[2.] $f$ dominates every $X_0 \oplus \dots \oplus X_s$-computable function.
\end{itemize}
By 2, any $\Delta^0_2$ approximation $\tilde{f}$ of the function $f$ 
is a computable instance of the escape property with no solution in $\Mcal_{\vec{X}}$,
that is, such that no function in $\Mcal_{\vec{X}}$ escapes~$f$.
By the computable equivalence between the escape property and the atomic model theorem (see subsection~\ref{subsect:dominating-amt}), 
$\Mcal_{\vec{X}} \not \models \amt$.
By Lemma~\ref{lem:em-uniform-model}, $\Mcal_{\vec{X}} \models \emo \wedge \coh$.

Start with $X_0 = \emptyset$ and the $\Delta^0_2$ enumeration $T_0 \geq T_1 \geq \dots$
of stage trees constructed in subsection~\ref{subsect:framework-construction},
and let $c_0 \geq c_1 \geq \dots$ be the sequence of their roots.
The set $U$ of their parts form an infinite, finitely branching tree,
whose structure is given by the refinement functions.
Moreover, by the construction of the sequence $T_0, T_1, \dots$, for every $s$, 
there is some part~$\nu$ in $c_{s+1}$ refining some part~$\mu$
in $c_s$ and which forces $\progop(\mu)$. Call such a part~$\nu$ a \emph{progressing part}. We may also
consider that every part of~$c_0$ is a progressing part, for the sake of uniformity.
By the implementation of $\progop$, if $\nu$ is a progressing part which refines some part~$\mu$,
$\mu$ is also a progressing part. Therefore, the set the progressing parts forms an infinite subtree $U_1$ of $U$.

Let $\nu_0, \nu_1, \dots$ be an infinite path through $U_1$.
Notice that $\satop(\nu_s) \neq \emptyset$.
Indeed, if $\satop(\nu_s) = \emptyset$, then the part~$\nu_s$ is empty in $\Ccal_s$, where $c_s = (\vec{E}_s, \Ccal_s)$,
and therefore we cannot find some progressing part~$\nu_{s+1}$ refining $\nu_s$.
Therefore, the set $\bigcap_s \satop(\nu_s)$ is non-empty. Let $X_1 \in \bigcap_s \satop(\nu_s)$.
By definition of $\satop(\nu_s)$, $X_1$ is a transitive subtournament of $R^{X_0}$.
By definition of $\progop$, for every $s$ and every set $G \in \satop(\nu_s)$, there is some~$n \in G$ such that $n > s$.
Therefore, the set $X_1$ is infinite, so the property 1 is satisfied.

Repeat the procedure with the sequence of stage trees $T_1^{[\nu_1]} \geq T_2^{[\nu_2]} \geq \dots$ and so on.
We obtain an infinite sequence of sets $X_0, X_1, \dots$ satisfying the property 1.
Let $f$ be the $\Delta^0_2$ function which on input $x$, returns $max(U_x)+1$ where
$U_x$ is the finite set stated in the domination lemma (Lemma~\ref{lem:domination-lemma}) for stage trees of depth $x$.
Fix some Turing index $e$ such that $\Phi^{X_0 \oplus \dots \oplus X_i}_e$ is total.
By the domination lemma, for every $x \geq max(e,i)$, $\Phi^{X_0 \oplus \dots \oplus X_i}_e(x) \in U_x < f(x)$.
Therefore the function~$f$ dominates every $X_0 \oplus \dots \oplus X_i$-computable function.
This finishes the proof of Theorem~\ref{thm:emo-not-amt}.

\section{Separating combined principles from~$\amt$}\label{sect:separating-amt-combined}

The domination framework has two purposes. First, it emphasizes on the key elements
of the construction and gets rid of the implementation technicalities by abstracting 
the main operations into operators. Second, it enables us to separate conjunctions
of principles from $\amt$, using the ability to compose modules into a compound one.
In this section, we will take advantage of the latter to prove that $\emo$
is not strong enough to prove~$\amt$, even when allowing compactness arguments.

\begin{theorem}\label{thm:emo-coh-amt}
There is an $\omega$-model of $\emo \wedge \coh \wedge \wkl$ which
is not a model of $\amt$.
\end{theorem}

In subsection~\ref{subsect:composing-modules}, we will show how to compose multiple modules to obtain
separations of conjunctions of principles from $\amt$. Then, in subsection~\ref{subsect:module-for-wkl}, we will
provide a module for $\wkl$ and will show how to choose properly the sequence
of sets $X_0, X_1, \dots$ to obtain an $\omega$-model of~$\wkl$. 

\subsection{Composing modules}\label{subsect:composing-modules}

When building the second-order part~$\Ical$ of an $\omega$-model of a countable collection of principles $\Psf_0, \Psf_1, \dots$,
we usually interleave the instances of the various $\Psf$'s so that each instance receives attention after a finite number of iterations.
This is exactly what we will do when composing module supports $\Sb_i = \tuple{\Pb_i, \Ub_i, \parop_i, \iniop_i, \satop_i}$
for $\Psf_i$ for each $i \in \omega$, in order to obtain a compound 
module support $\Sb = \tuple{\Pb, \Ub, \parop, \iniop, \satop}$ for $\bigwedge_{i \in \Nb} \Psf_i$.
The domain of the partial order~$\Pb$ is obtained by taking the disjoint union of the partial orders $\Pb_i$.
Therefore $\Pb = \{ \tuple{c,i} : i \in \Nb \wedge c \in \Pb_i \}$.
The order is defined accordingly: $\tuple{d, j} \leq_\Pb \tuple{c,i}$ if $i = j$
and $d \leq_{\Pb_i} c$. Similarly, $\Ub = \{ \tuple{\nu,i} : i < \Nb \wedge \nu \in \Ub_i \}$,
$\parop(\tuple{c,i}) = \{\tuple{\nu, i} : \nu \in \parop_i(c)\}$ and $\satop(\tuple{\nu,i}) = \satop_i(\nu)$.

The key element of the composition is the definition of $\iniop(n)$, which will
return $\iniop_i(m)$ if $n$ codes the pair $(m, i)$.
This way, infinitely many iterations are responsible for making $\Ical$ satisfy $\Psf_i$ for each $i \in \Nb$.
The construction within the domination framework therefore follows the usual construction of a model satisfying two principles.

The property~(i) in the definition of a module support for~$\Sb$ inherits from
the property~(i) of $\Sb_i$ for each $i \in \Nb$. Indeed, if $\tuple{d,j} \leq_\Pb \tuple{c,i}$,
then $j = i$ and $d \leq_{\Pb_i} c$. By the property~(i) of $\Mb_i$, there is a function
$f : \parop_i(d) \to \parop_i(c)$ such that $\satop_i(\nu) \subseteq \satop_i(f(\nu))$
for each $\nu \in \parop_i(d)$. Let $g : \parop(\tuple{d,i}) \to \parop(\tuple{c,i})$
be defined by $g(\tuple{\nu,i}) = \tuple{f(\nu), i}$.
$\satop(\tuple{\nu,i}) = \satop_i(\nu) \subseteq \satop_i(f(\nu)) = \satop(g(\tuple{\nu,i})$.

Given a module $\Mb_i = \tuple{\Sb_i, \boxop_i, \unboxop_i, \progop_i}$ for $\Psf_i$ for each~$i \in \Nb$,
the definition of the compound module $\Mb = \tuple{\Sb, \boxop, \unboxop, \progop}$ for $\bigwedge_{i \in \Nb} \Psf_i$
does not contain any particular subtlety.
Simply redirect $\boxop(\tuple{\nu,i}, \varphi)$ to $\boxop_i(\nu, \varphi)$,
$\unboxop(\tuple{\nu,i}, a)$ to $\unboxop_i(\nu, a)$, and $\progop(\tuple{\nu,i})$ to $\progop_i(\nu)$.
Again, the properties of a module support for $\Mb$ inherit the properties for $\Mb_i$.

\subsection{A module for $\wkl$}\label{subsect:module-for-wkl}

Weak K\"onig's lemma states for every infinite binary tree
the existence of an infinite path through it.
The usual effective construction of such a path follows
the classical proof of K\"onig's lemma:
we build the path by finite approximations and consider 
the infinite subtree below the finite path we constructed so far.
The difficulty consists of finding which ones, among the finite extensions candidates,
induce an infinite subtree.

First note that we do not share the same concerns as for the Erd\H{o}s-Moser theorem
about the choice of an instance, since $\wkl$ admits a universal instance which is the tree
whose paths are completions of Peano arithmetics. Moreover,
this universal instance is a primitive recursive tree functional.

It is natural to choose the infinite, computable binary tree functionals as our forcing conditions.
A condition (tree) $U$ \emph{extends} $T$ if $U^D \subseteq T^D$. 
A set $G$ \emph{satisfies} the condition $T$ if $G$ is an infinite path through $T^D$.
Let us now see how we decide some $\Sigma^0_1$ query $\varphi(D, G)$. Consider the following question:

\smallskip
{\itshape
Q6: Is the set $T^D \cap \{\sigma \in 2^{<\omega} : \neg \varphi(D, \sigma) \}$ finite?
}
\smallskip

Let $\Gamma^D_\varphi = \{\sigma \in 2^{<\omega} : \neg \varphi(D, \sigma) \}$.
Whenever $\varphi(D, \tau)$ holds and $\rho \succeq \tau$, $\varphi(D, \rho)$ holds,
thus $\Gamma^D_\varphi$ is a tree.
At first sight, the question Q6 seems $\Sigma^{0,D}_2$. However, $T^D \cap \Gamma^D_\varphi$ is a tree,
so the question can be formulated in a $\Sigma^{0,D}_1$ way as follows:

\smallskip
{\itshape
Q6': Is there some length $n$ such that $T^D \cap \Gamma^D_\varphi$ has no string of length $n$?
}
\smallskip

If the answer is negative, the extension $T \cap \Gamma_\varphi$
is valid and forces $\varphi(D, G)$ not to hold.
If the answer is positive, the condition $T$ already forces $\varphi(D, G)$ to hold.
Note that there is a hidden application of our moto ``if you cannot choose, try every possibilities in parallel''.
Indeed, in many forcing arguments involving weak K\"onig's lemma,
we $\emptyset'$-choose an extensible string $\sigma \in T$ such that $\varphi(D, \sigma)$ holds
and $T^{D, [\sigma]}$ is infinite. However, we meet the same problem as in the Erd\H{o}s-Moser case,
that is, we are unable to decide which of the $\sigma$'s will be extensible into an infinite subtree.
To be more precise, for every $\sigma \in 2^n$, there may be some $D$ such that the set $T^{D, [\sigma]}$ is finite.
By taking $T$ as our extension forcing $\varphi(D, G)$ to hold, we take in reality $\bigcup_{\sigma \in 2^n \cap T} T^{[\sigma]}$,
that is, we take the union of the candidate extensions $T^{[\sigma]}$.
We are now ready to define a module support $\Sb^{\wkl} = \tuple{\Pb, \Ub, \parop, \iniop, \satop}$ for $\wkl$.

The set $\Pb$ is the set of conditions as defined above.
Each condition has only one part which can be identified as the condition itself,
therefore $\Ub = \Pb$. Accordingly, $\parop(T) = \{T\}$.
The function $\iniop(n)$ always returns the universal instance of $\wkl$.
Last, $\satop(T)$ is the collection of the infinite paths through $T^D$.

\begin{lemma}
$\Sb^{\wkl}$ is a module support.
\end{lemma}
\begin{proof}
We need to check the property (i) of a module support.
Let~$U \leq_\Pb T$ for some conditions~$T$ and~$U \in \Pb$.
Define $f : \parop(U) \to \parop(T)$ as the function $f(U) = T$.
We claim that $\satop(U) \subseteq \satop(f(U))$ for each~$U \in \parop(U)$.
Since $\parop(U) = \{U\}$, we need to check that $\satop(U) \subseteq \satop(T)$,
which is immediate since $U \subseteq T$.
\end{proof}

We now define the module $\Mb^{\wkl} = \tuple{\Sb^{\wkl}, \boxop, \unboxop, \progop}$ as follows.
Given some tree $T$ and some query $\varphi(D, G)$, $\boxop(T, G)$ is the formula
$\psi(D) = (\exists n)[T^D \cap \Gamma^D_\varphi \cap 2^n = \emptyset]$.
Recall that $\Gamma^D_\varphi = \{\sigma \in 2^{<\omega} : \neg \varphi(D, \sigma) \}$.
If the answer to the question $\psi(D)$ is $\tuple{\no}$,
$\unboxop(T, \tuple{\no})$ returns the tuple $\tuple{T \cap \Gamma_\varphi, f, b}$
where $f : \parop(T \cap \Gamma_\varphi) \to \parop(T)$ is trivially defined by $f(T \cap \Gamma_\varphi) = T$
and $b$ is the constant function returning $\tuple{\no}$ everywhere.
In the answer to the question $\psi(D)$ is $a = \tuple{\yes, n, a'}$, where $n$ is the integer witnessing
$T^D \cap \Gamma^D_\varphi \cap 2^n = \emptyset$ and $a'$ witnesses the other existentials variables in $\varphi(D, G)$,
$\unboxop(T, a)$ returns the tuple $\tuple{T, id, b}$
where $id$ is the identify refinement function and $b$ is the constant function returning~$\tuple{\yes, a'}$ everywhere.
No progress is needed for~$\wkl$. Therefore, $\progop(T)$ can be chosen to be any formula.


\begin{lemma}
$\Mb^{\wkl}$ is a module.
\end{lemma}
\begin{proof}
We need to ensure that, given the unique part $T$ of the condition $T$
and some answer $a$ to a $\Sigma^0_1$ formula $\psi(D) = \boxop(T, \varphi)$
where $\varphi(D, G)$ is a $\Sigma^0_1$ formula, $\unboxop(T, a)$
outputs a tuple $\tuple{U, f, b}$ where $U$ is an extension of~$T$,
$f : \parop(U) \to \parop(T)$ is defined by $f(U) = T$
and for every set $G \in \satop(U)$, $b(U)$ is an answer to $\varphi(D, G)$.

Suppose that $a = \tuple{\no}$. By definition of $\unboxop(T, \tuple{\no})$,
$U = T \cap \Gamma_\varphi$ and $b$ is the constant function $\tuple{\no}$. 
By definition of $\satop(U)$, $G$ is an infinite path through $T^D \cap \Gamma^D_\varphi$.
Let $\sigma$ be any initial segment of $G$. In particular, $\sigma \in \Gamma^D_\varphi$. Unfolding the definition of $\Gamma^D_\varphi$n
$\varphi(D, \sigma)$ does not hold. Therefore $\varphi(D, G)$ does not hold.

Suppose now that $a = \tuple{\yes, n, a'}$, where $a'$ witnesses the existential variables of $\varphi(D, G)$.
Again, by definition of $\unboxop(T, a)$, $U = T$, $f$ is the identify function 
and $b$ is the constant function returning $\tuple{\yes, a'}$ everywhere.
By definition of $\satop(U)$, $G$ is an infinite path through $T^D$.
Let $\sigma$ be an initial segment of $G$ of length $n$. Since $T^D \cap \Gamma^D_\varphi \cap 2^n = \emptyset$,
$\varphi(D, \sigma)$ holds, so $\varphi(D, G)$ holds.
\end{proof}

Finally, we explain how to extract a solution to the universal instance of $\wkl$ below some set $D$,
given the infinite decreasing sequence of stage trees constructed in subsection~\ref{subsect:framework-construction}.
Given the sequence $T_0 \geq T_1 \geq \dots$ whose roots are $T_0 \geq T_1 \geq \dots$,
there is no much choice since each condition $T_s$ has only one part, that is, the tree $T_s$ itself.
By compactness, $\bigcap_s T_s^D$ is infinite. Take any infinite path~$G$ through~$\bigcap_s T_s^D$.
This completes the proof of Theorem~\ref{thm:emo-coh-amt}.


\subsection{Beyond the atomic model theorem}

We conclude this section by a discussion on the generality of the domination framework and
its key properties.

Ramsey-type theorems satisfy one common core combinatorial property: given an instance $I$ of a principle $\Psf$,
for every infinite set $X \subseteq \Nb$, there is a solution of~$Y \subseteq X$ of $I$.
This property makes Ramsey-type principles combinatorially weak. Indeed, Solovay~\cite{Solovay1978Hyperarithmetically}
proved that the sets computable by every solution to a given instance~$I$ of $\Psf$ are precisely
the hyperarithmetical ones.  Moreover, Groszek and Slaman~\cite{Groszek2007Moduli} proved that the hyperarithmetical 
sets are precisely the sets~$S$ admitting a modulus, namely,
a function $f$ such that every function dominating $f$ computes~$S$. These results put together can be interpreted
as stating that the coding power of Ramsey-type principles comes from the sparsity of their solutions.
If an instance can force its solutions~$H = \{x_0 < x_1 < \dots \}$ to have arbitrarily large gaps, then the principal function
$p_H$ defined by $p_H(n) = x_n$ will be fast-growing, and contain some computational power.

The strength of many principles in reverse mathematics can be explained 
in terms of the ability to ensure gaps in the solutions.
$\aca$ has instances whose solutions are everywhere-sparse, in that the principal function
of the solutions dominated the modulus function of $\emptyset'$.
Some principles such as $\coh$, $\amt$ or $\fip$ imply the existence of hyperimmune sets,
which are sets sparse enough so that their principal function is not dominated by any computable function.
These sets have infinitely many gaps, but their repartition cannot be controlled.

Another important aspect of the hole-based analysis is their definitional complexity.
For example, $\amt$ has the ability to ensures $\Delta^0_2$ gaps, which gives it more computational power
than $\coh$ or $\emo$ which can only have $\Delta^0_1$ gaps. This is the main feature used by the domination framework
to prove that $\coh \wedge \emo$ does not imply $\amt$.
This framework was designed to exploit this weakness of the principles, and is therefore relatively specific to
the atomic model theorem. However, some weakenings of $\amt$, such as the finite intersection property,
share some similar properties, in that they can also be purely characterized in terms of hyperimmunity properties.
The author leaves open the following question:

\begin{question}
Does $\coh$ imply $\fip$ in $\rca$?
\end{question}


\vspace{0.5cm}

\noindent \textbf{Acknowledgements}. The author is thankful to his PhD advisor Laurent Bienvenu
for useful comments and discussions. The author is funded by the John Templeton Foundation (`Structure and Randomness in the Theory of Computation' project). The opinions expressed in this publication are those of the author(s) and do not necessarily reflect the views of the John Templeton Foundation.

\vspace{0.5cm}

\bibliographystyle{plain}
\bibliography{../bibliography}

\clearpage
\appendix

\section{Evolution of the local zoo}

In this last section, we give a short history of the zoo related to the Erd\H{o}s-Moser theorem.
In Figure~\ref{fig:local-zoo}, we present the various implications proven
between the principles $\emo$, $\sts^2$, $\sads$ and $\amt$. An arrow
denotes an implication over $\rca$. A dotted arrow from a principle $\Psf$
to a principle $\Qsf$ denotes the existence of an $\omega$-model of $\Psf$ which is not a model of~$\Qsf$.

\begin{figure}[htbp]
\begin{center}
\begin{tikzpicture}[x=2cm, y=1.7cm, 
	node/.style={minimum size=2em},
	impl/.style={draw,very thick,-latex},
	strict/.style={impl}, 
	nonimpl/.style={draw, very thick, dotted, -latex},
	edgelabel/.style={inner sep=0pt}]

	\node[node] (coh+em+wkl) at (2, 3) {$\coh+\emo+\wkl$};
	\node[node] (em) at (2, 2)  {$\emo$};
	\node[node] (sts2) at (1, 1) {$\sts(2)$};
	\node[node] (sads) at (3, 1) {$\sads$};
	\node[node] (amt) at (2, 0) {$\amt$};

	\draw[strict] (coh+em+wkl) -- (em);
	\draw[strict] (sts2) -- (amt) node [edgelabel, midway, above right=5pt] {\tiny (2)};
	\draw[strict] (sads) -- (amt) node [edgelabel, midway, above left=5pt] {\tiny (1)};
	
	\draw[nonimpl] (em) -- (sads) node [edgelabel, midway, below left=2pt] {\tiny (3)};
	\draw[nonimpl] (em) -- (sts2) node [edgelabel, midway, below right=2pt] {\tiny (4)};
	\draw[nonimpl] (coh+em+wkl) -- (sts2) node [edgelabel, midway, below right=1pt] {\tiny (5)};
	\draw[nonimpl] (coh+em+wkl) -- (sads) node [edgelabel, midway, below left=1pt] {\tiny (5)};
	\draw[nonimpl] (coh+em+wkl) .. controls (0,1) .. (amt) node [edgelabel, midway, right=3pt] {\tiny (6)};;
\end{tikzpicture}
\end{center}
\caption{Evolution of the zoo}\label{fig:local-zoo}
\end{figure}
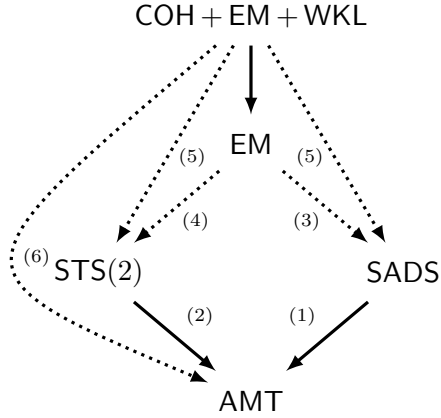

Justification of the arrows:
\begin{itemize}
	\item[(1)] Hirschfeldt, Shore and Slaman~\cite{Hirschfeldt2009atomic} proved that~$\amt$ is a consequence of~$\sads$ over~$\rca$. 
	\item[(2)] The author proved in~\cite{Patey2015Somewhere} that $\sts(2)$ implies $\amt$ over~$\rca$ using a similar argument.
	\item[(3)] Lerman, Solomon and Towsner~\cite{Lerman2013Separating} separated $\emo$ from~$\sads$ using an iterated forcing construction.
	\item[(4)] The author noticed in~\cite{Patey2013note} that the forcing of Lerman, Solomon and Towsner
	could be adapted to separate $\emo$ from $\sts(2)$ over $\rca$.
	\item[(5)] Wang~\cite{Wang2014Definability} used the notion of preservation of $\Delta^0_2$ definitions
	to separate $\coh+\emo+\wkl$ from $\sads$ and $\sts(2)$ over $\rca$.
	\item[(6)] This is the main result of the current paper.
\end{itemize}

\end{document}